\documentclass[notitlepage,11pt,reqno]{amsart}
\usepackage[foot]{amsaddr}
\usepackage{amssymb,nicefrac,bm,upgreek,mathtools,verbatim}
\usepackage[final]{hyperref}
\usepackage[mathscr]{eucal}
\usepackage{dsfont}
\usepackage[normalem]{ulem}
\usepackage{amsopn}

\usepackage[margin=1in]{geometry}
\allowdisplaybreaks
\newcommand{\stkout}[1]{\ifmmode\text{\sout{\ensuremath{#1}}}\else\sout{#1}\fi}
\newtheorem{lemma}{Lemma}[section]
\newtheorem{theorem}{Theorem}[section]
\newtheorem{proposition}{Proposition}[section]

\theoremstyle{definition}

\newtheorem{assumption}{Assumption}[section]

\newtheorem{example}{Example}[section]
\newtheorem{algorithm}{Algorithm}[section]

\theoremstyle{remark}
\newtheorem{remark}{Remark}[section]
\numberwithin{theorem}{section}
\numberwithin{equation}{section}

\hypersetup{
  colorlinks=true,
  citecolor=mblue,
  linkcolor=mblue,
  urlcolor = blue,
  anchorcolor = blue,
  frenchlinks=false,
  pdfborder={0 0 0},
  naturalnames=false,
  hypertexnames=false,
  breaklinks}
\usepackage[capitalize,nameinlink]{cleveref}
\usepackage[abbrev,msc-links,nobysame]{amsrefs} 
\crefname{section}{Section}{Sections}
\crefname{subsection}{Section}{Sections}
\crefname{condition}{Condition}{Conditions}
\crefname{hypothesis}{Hypothesis}{Conditions}
\crefname{assumption}{Assumption}{Assumptions}
\crefname{lemma}{Lemma}{Lemmas}
\crefname{fact}{Fact}{Facts}

\Crefname{figure}{Figure}{Figures}

\crefformat{equation}{\textup{#2(#1)#3}}
\crefrangeformat{equation}{\textup{#3(#1)#4--#5(#2)#6}}
\crefmultiformat{equation}{\textup{#2(#1)#3}}{ and \textup{#2(#1)#3}}
{, \textup{#2(#1)#3}}{, and \textup{#2(#1)#3}}
\crefrangemultiformat{equation}{\textup{#3(#1)#4--#5(#2)#6}}%
{ and \textup{#3(#1)#4--#5(#2)#6}}{, \textup{#3(#1)#4--#5(#2)#6}}%
{, and \textup{#3(#1)#4--#5(#2)#6}}

\Crefformat{equation}{#2Equation~\textup{(#1)}#3}
\Crefrangeformat{equation}{Equations~\textup{#3(#1)#4--#5(#2)#6}}
\Crefmultiformat{equation}{Equations~\textup{#2(#1)#3}}{ and \textup{#2(#1)#3}}
{, \textup{#2(#1)#3}}{, and \textup{#2(#1)#3}}
\Crefrangemultiformat{equation}{Equations~\textup{#3(#1)#4--#5(#2)#6}}%
{ and \textup{#3(#1)#4--#5(#2)#6}}{, \textup{#3(#1)#4--#5(#2)#6}}%
{, and \textup{#3(#1)#4--#5(#2)#6}}

\crefdefaultlabelformat{#2\textup{#1}#3}
\let\oldtocsection=\tocsection

\let\oldtocsubsection=\tocsubsection

\let\oldtocsubsubsection=\tocsubsubsection

\renewcommand{\tocsection}[2]{\hspace{0em}\oldtocsection{#1}{#2}}
\renewcommand{\tocsubsection}[2]{\hspace{1em}\oldtocsubsection{#1}{#2}}
\renewcommand{\tocsubsubsection}[2]{\hspace{2em}\oldtocsubsubsection{#1}{#2}}
\newcommand{\vertiii}[1]{{\left\vert\kern-0.25ex\left\vert\kern-0.25ex\left\vert #1 
    \right\vert\kern-0.25ex\right\vert\kern-0.25ex\right\vert}}
\newcommand{\lamstr}{\lambda^{\!*}}% lambda with corrected 'star'
\newcommand{\Uadm}{\mathfrak U}
\newcommand{\Act}{\mathbb{U}}
\newcommand{\Usm}{\mathfrak U_{\mathsf{sm}}}

\newcommand{\Um}{\mathfrak{U}_{\mathsf{m}}}

\newcommand{\fB}{{\mathfrak{B}}}  % Borel Sets
\newcommand{\cB}{{\mathcal{B}}}  % Borel Maps
\newcommand{\sB}{{\mathscr{B}}}  % Ball
\newcommand{\cC}{{\mathcal{C}}}   % Continuous functions
\newcommand{\sD}{{\mathscr{D}}}   % Finite set in S
\newcommand{\sE}{{\mathscr{E}}} 
\newcommand{\sF}{{\mathfrak{F}}}   % sigma field
\newcommand{\sH}{{\mathscr{H}}}  % class of invariant measures/History
\newcommand{\cK}{{\mathcal{K}}}  % partition set for running cost
\newcommand{\sK}{{\mathscr{K}}}  % admissible state action space
\newcommand{\cT}{{\mathcal{T}}}
\newcommand{\Lyap}{{\mathcal{V}}}  % Lyapunov
\newcommand{\cX}{{\mathcal{X}}}

\newcommand{\RR}{\mathds{R}}
\newcommand{\NN}{\mathds{N}}

\DeclareMathOperator{\Exp}{\mathbb{E}}
\DeclareMathOperator{\Prob}{\mathbb{P}}
\newcommand{\D}{\mathrm{d}}

\newcommand{\Ind}{\mathds{1}}   % indicator function
\newcommand{\df}{:=}

\DeclareMathOperator*{\Argmin}{Arg\,min}

\newcommand{\order}{{\mathscr{O}}}
\newcommand{\sorder}{{\mathfrak{o}}}

\newcommand{\uuptau}{{\Breve\uptau}}

\newcommand{\lamd}{\lambda^{\rm d}_{1}}
\newcommand{\lamc}{\lambda^{\rm c}_{1}}

\newcommand{\norm}[1]{\lVert#1\rVert}

\newcommand{\Bnorm}[1]{\Bigl\lVert#1\Bigr\rVert}

\usepackage{color}
\definecolor{dmagenta}{rgb}{.4,.1,.5}
\definecolor{dblue}{rgb}{.0,.0,.5}
\definecolor{mblue}{rgb}{.0,.0,.7}
\definecolor{ddblue}{rgb}{.0,.0,.4}
\definecolor{dred}{rgb}{.7,.0,.0}
\definecolor{dgreen}{rgb}{.0,.5,.0}
\definecolor{Eeom}{rgb}{.0,.0,.5}

\begin{document}
\title[Risk-sensitive control of Markov processes on countable state space]
{Ergodic risk-sensitive control of Markov processes on countable state space revisited}

\author[Anup Biswas]{Anup Biswas}
\address{Department of Mathematics,
Indian Institute of Science Education and Research,
Dr.\ Homi Bhabha Road, Pune 411008, India}
\email{$\lbrace$anup,somnath$\rbrace$@iiserpune.ac.in}

\author[Somnath Pradhan]{Somnath Pradhan}
%\email{somnath@iiserpune.ac.in}

%%%%%%%%%%%%%%%%%%%%%%%%%%%%%%%%%%%%%%%%%%%%%%%%%%%%%%%%%%%%%%%%%%%%%%%%%%%%%%%%
\begin{abstract}
We consider a large family of discrete and continuous time controlled Markov processes and study an ergodic risk-sensitive minimization problem. Under a blanket stability assumption, we provide a complete analysis to this problem. In particular, we establish uniqueness of the value function and verification result for optimal stationary Markov controls, in addition to the existence results. We also revisit this problem under a near-monotonicity condition but without any stability hypothesis. Our results also include policy improvement algorithms both in discrete and continuous time frameworks.
\end{abstract}
%%%%%%%%%%%%%%%%%%%%%%%%%%%%%%%%%%%%%%%%%%%%%%%%%%%%%%%%%%%%%%%%%%%%%%%%%%%%%%%%
\keywords{Risk-sensitive control, ergodic cost criterion, stochastic representation, verification result,
Markov decision problem, near-monotone cost.}

\subjclass[2010]{90C40, 91B06, 60J10}
%%%%%%%%%%%%%%%%%%%%%%%%%%%%%%%%%%%%%%%%%%%%%%%%%%%%%%%%%%%%%%%%%%%%%%%%%%%%%%%%
\maketitle
\tableofcontents
%%%%%%%%%%%%%%%%%%%%%%%%%%%%%%%%%%%%%%%%%%%%%%%%%%%%%%%%%%%%%%%%%%%%%%%%%%%%%%%%

\section{Introduction}
Let {\bf X} be a controlled Markov process (CMP), either discrete or continuous time, taking values in a
discrete state space $S$. Let $\Uadm$ be the class of admissible controls which also includes history
dependent controls. We study the minimization problem of ergodic exponential cost given by
$$\sE_i (c, \zeta)=\lim_{T\to\infty}\frac{1}{T}\log\Exp_i^\zeta\left[e^{\sum_{t=0}^{T-1} 
c(X_t,\zeta_t)}\right]\quad \text{(for discrete time)},$$
$$\sE_i (c, \zeta)=\lim_{T\to\infty}\frac{1}{T}\log\Exp_i^\zeta\left[e^{\int_{0}^{T} 
c(X_t,\zeta_t)\D t}\right]\quad \text{(for continuous time)},$$
where $c$ is the running cost and $\zeta\in\Uadm$. More precisely, we are interested in the 
optimal value
$$\lamstr\,=\, \inf_{i\in S}\, \inf_{\zeta\in\Uadm} \sE_i(c, \zeta),$$
and the characterization of all optimal stationary Markov controls. Our main results,
\cref{T2.1,T3.1}, establish
existence of an eigenpair $(\lamstr, \psi^*), \psi^*>0,$ satisfying
\begin{align}\label{main-eigeqn}
\begin{rcases}
e^{\lamstr} \psi^*(i) &= \min_{u\in\Act(i)} \left[e^{c(i, u)}\sum_{j} \psi^*(j) P(j|i, u)\right]
\quad \text{(for discrete time)},
\\
\lamstr\psi^*(i) &= \min_{u\in\Act(i)}\left[\sum_{j\in S} \psi^*(j)q(j|i,u) + c(i,u)\psi^*(i)\right]
\quad \text{(for continuous time)},
\end{rcases}
\end{align}
and also show that any minimizing selector of \eqref{main-eigeqn} is an optimal control.
In addition, We  show that
$\psi^*$ is unique upto a normalization and any optimal stationary Markov control
 is a measurable selector of
\eqref{main-eigeqn}. We also propose a policy improvement algorithm (PIA) and establish its convergence.

To the best of our knowledge, risk-sensitive optimization problems have been first considered 
in the seminal paper by Howard and Matheson \cite{Howard-71}, but it was only lately that this topic gained renewed interest due to applications in finance and large deviation theory.
In this respect we mention the interesting work of Kontoyiannis and Meyn \cite{KM03} studying
multiplicative ergodic theorem and large deviation principle for geometrically stable 
Markov processes.
 Some of the
early works on finite horizon risk-sensitive control includes Jacobson \cite{Jacobson-73},
Speyer, Deyst and Jacobson \cite{Speyer-74}, Speyer \cite{Speyer-76}, Gheorghe \cite{Gheorghe},
Whittle \cite{Wh81}, James, Baras and Elliott \cite{JBE94}, etc. 
Bielecki, Hern\'{a}ndez-Hern\'{a}ndez and Pliska \cite{BHP99} consider the ergodic risk-sensitive 
problem for CMP with a finite state space and establish the existence of a unique solution to 
\eqref{main-eigeqn}.
Ergodic risk-sensitive control
for discrete time CMP with countable state space 
is studied by Borkar and Meyn \cite{BorMey02} (see also, Hern\'{a}ndez-Hern\'{a}ndez
and Marcus \cite{HHM96}). Di Masi and Stettner
\cite{DMS99,DMS07} consider the problem in a general state space. Later
Shen, Stannat and Obermayer \cite{SSO13},
B\"{a}uerle and Rieder \cite{BR14} extend these results for a wider class of utility functions.
Let us also mention the work of Basu and Ghosh \cite{BG14}, B\"{a}uerle and Rieder \cite{BR17}
which study zero-sum game with ergodic risk-sensitive cost criterion. 
Most of the above mentioned works, with the exception of \cite{DMS99,BR17}, establish existence of a solution to \eqref{main-eigeqn} and show that every minimizing 
selector is an optimal control. So a natural question is whether all the optimal stationary Markov controls are obtained in this fashion from \eqref{main-eigeqn}. Also, given any eigenpair 
$(\lamstr, V), V>0,$ satisfying \eqref{main-eigeqn} whether we can find an optimal Markov control through
a measurable selection. This is also related with the uniqueness of $\psi^*$. In \cite{DMS99,BR17},
the authors establish uniqueness of $\psi^*$ under a more restrictive setting
(see \cite[(A1)]{DMS99}\cite[(E2)]{BR17}). The uniqueness of $\psi^*$ in these papers is a
consequence of the contraction property of certain map associated to \eqref{main-eigeqn}. Also, this uniqueness result is established among a class of functions belonging to a certain weighted Banach space and arguments of these articles do not give uniqueness in the class of all positive continuous functions. 
Furthermore, the results of \cite{DMS99,BR17} can not be used to obtain a verification result
for optimal stationary Markov controls. Let us also mention our result \cref{T2.4} which establishes 
existence of an optimal stationary Markov control under a near-monotonicity assumption on the cost but
does not impose any stability hypothesis.

On the other hand, the literature on risk-sensitive control problems for continuous time 
CMP are very few, especially for ergodic risk-sensitive control problems. Wei \cite{QW16},
Guo, Liu and Zhang \cite{GLZ19} investigate finite horizon risk-sensitive criterion 
for continuous time CMP taking values in a countable state space. An infinite horizon discounted
cost criterion is considered by Guo and Liao \cite{GL19}. Wei and Chen \cite{WC16} consider the
ergodic risk-sensitive criterion for a finite-state continuous time CMP and 
establish the existence of
an optimal control using the nonlinear eigen-equation \eqref{main-eigeqn}.
 The articles that are close to the 
problem we are considering in this paper are by Ghosh and Saha \cite{GS14}, Suresh Kumar and
Pal \cite{SKP15}, Guo and Huang \cite{GH21}. \cite{GS14} studies the problem under a stability hypothesis whereas
\cite{SKP15} imposes a near-monotonicity assumption on the cost. Both the articles obtain
the existence of a principal eigenfunction satisfying \eqref{main-eigeqn} and show that any
minimizing selector is an optimal Markov control. 
Recently,  Guo and Huang \cite{GH21} study a similar control problem for continuous time CMP satisfying a blanket geometric-stability
condition, and the existence of solutions to \eqref{main-eigeqn} and the existence
of an optimal stationary Markov control are established. It should be noted
that the stability hypothesis in \cite[Assumption~3.1]{GH21} is stronger
than our \cref{A3.4}. Uniqueness of the value function $\psi^*$ is also
established in \cite{GH21} by imposing a further set of conditions
(see Assumption~5.1 and ~6.1 there). In this article we do not impose any
such conditions to obtain the uniqueness of $\psi^*$.
We take a different approach to attack this
problem and establish the existence of a unique eigenfunction and verification result for
optimal stationary Markov controls (see \cref{T3.1}), under a 
blanket stability hypothesis. In \cref{T3.2,T3.3} we consider the problem under
a condition of near-monotonicity on the cost function and establish existence of optimal 
stationary Markov controls. 
Our approach in this article is inspired from
the work of Biswas \cite{Biswas-11a}, Arapostathis et.\ al.\ \cite{ABS19}
(see also \cite{AB18}) where ergodic
risk-sensitive control is studied for non-degenerate controlled diffusion. It
should be observed that the ideas of \cite{Biswas-11a,AB18,ABS19} can not be adapted in a
straightforward manner to the present setting. These papers use several analytic tools
such as Harnack's inequality, Sobolev estimate, monotonicity of Dirichlet principal eigenvalues for their analysis. We do not have similar estimates in hand. On the other hand, our state space being discrete, we get an advantage in the passage of several limits using a standard
diagonalization argument.

As mentioned above we also provide a PIA for both discrete and continuous time setup. In their
work \cite{BorMey02}, Borkar and Meyn propose a PIA for norm-like cost function and establish its 
convergence, provided some additional hypotheses hold \cite[Theorem~5.4]{BorMey02}. Some of these
conditions are not easily verifiable. Ghosh and Saha \cite{GS14} (see also \cite{GH21}) prove convergence of PIA for 
a finite state CMP. Both the papers
 \cite{BorMey02,GS14} assume their action space to be finite.
 In a recent work, Arapostathis, Biswas and Pradhan \cite{ABP20} establish convergence of their proposed PIA for non-degenerate controlled diffusions, provided the 
diffusion satisfies certain stability hypothesis. In \cref{S-PIA} we propose a PIA and
show that the algorithm converges to the optimal value $\lamstr$. For our result we make use
of a characterization of the Perron-Frobenius eigenvalue (see \eqref{E4.1}).

The rest of the article is organized as follows: In \cref{S-DT} we consider the discrete time
CMP and our main results of this section are \cref{T2.1,T2.4}. \cref{S-CT} studies a similar problem
for continuous time CMP. Finally, in \cref{S-PIA} we consider the policy improvement algorithms. 

%%%%%%%%%%%%%%%%%%%%%%%%%%%%%%%%%%%%%%%%%%%%%%%%%%%%%%%%%%%%%%%%%%%%%%%%%%%%%%%%
%%%%%%%%%%%%%%%%%%%%%%%%%%%%%%%%%%%%%%%%%%%%%%%%%%%%%%%%%%%%%%%%%%%%%%%%%%%%%%%%
\section{Risk-sensitive control for Discrete time CMP}\label{S-DT}
\subsection{Description of the problem} 
We consider a controlled Markov process $\textbf{X} \df \{X_0, X_1,\dots\}$ on a denumerable state space $S \df \{1,2,\dots\}$ controlled by a control process $\zeta \df \{\zeta_0,\zeta_1,\dots \}$ taking values in $\Act$. Here $\Act$ is a Borel space endowed with the Borel $\sigma$ algebra $\fB(\Act)$. For every $i\in S$, $\Act(i) \in \fB(\Act)$ stands for the nonempty compact set of all admissible actions when the system is at the state $i$. The space of all admissible state action pairs is given by $\sK \df \{(i,u) : i\in S, u\in \Act(i)\}$. For each 
$A\in \fB(S)$ the controlled stochastic kernel $P(A |\cdot):\sK\to [0,1]$ is Borel measurable. 
We denote by $c:\sK \to \RR_{+}$ the one-stage cost function. For each $t\in\NN$, the space $\sH_t$ denotes the {\it admissible histories} upto time $t$, where $\sH_0 \df S$, $\sH_t = \sK\times\sH_{t-1}$.
A generic element $h_t$ of $\sH_t$ is a vector of the form
$$h_t=(x_0, u_0, x_1, u_1, \ldots, x_{t-1}, u_{t-1}, x_t),
\quad \text{with}\;\;(x_s, u_s)\in\sK,\quad 0\leq s\leq t-1,$$
and $x_t\in S$, denotes the observable history of the process upto time $t$.
We also denote by $\sF_n=\fB(\sH_n)$. An {\it admissible control}  is a sequence $\zeta = \{\zeta_0,\zeta_1,\dots \}$ where for each $t\in\NN$\,, $\zeta_t : \sH_t\to\Act$ is a measurable map satisfying $\zeta_t(h_t)\in\Act({x_t})$, for all $h_t\in\sH_t$\,. The set of all admissible policies is denoted by $\Uadm$\,. It is well known that for a given initial state $i\in S$ and policy $\zeta\in\Uadm$ there exist unique probability measure $\Prob_i^{\zeta}$ on 
$(\Omega, \fB(\Omega))$, where $\Omega=(S\times\Act)^\infty$, (see \cite[p.4]{H89}, \cite{ABFGM93}) satisfying the following
\begin{equation}\label{Markov1}
\Prob_i^\zeta(X_0=i)=1,\quad \text{and}\quad
\Prob_i^{\zeta}(X_{t+1} \in A | \sH_t, \zeta_t) = P(A | X_t, \zeta_t)\quad \forall \,\, A\in \fB(S)\,.
\end{equation}
The corresponding expectation operator is denoted by $\Exp_i^{\zeta}$\,. A policy $\zeta\in\Uadm$ is said to be a Markov policy if $\zeta_t(h_t) = v_t(x_t)$ for all $h_t\in\sH_t$\,, for some measurable map $v_t:S\to\Act$ such that $v_t(i)\in\Act(i)$ for all $i\in S$\,. The set of all Markov policies is denoted by $\Um$\,. If the map $v_t$ does not have any explicit time dependence, that is, $\zeta_t (h_t) = v(x_t)$ for all $h_t\in\sH_t$, then $\zeta$ is called a stationary Markov strategy and  we denote the 
set of all stationary Markov strategies by $\Usm$. From \cite[p.6]{H89} (also see \cite{ABFGM93}), it is easy to see that under any Markov policy $\zeta\in\Um$ the corresponding stochastic process $\textbf{X}$ is strong Markov. For each $\zeta\in\Uadm$ the ergodic risk-sensitive cost is given by
\begin{equation}\label{EErgocost}
\sE_i(c, \zeta) \,\df\, \limsup_{T\to\infty} \, \frac{1}{T}\,
\log \Exp_i^{\zeta} \left[e^{\sum_{t = 0}^{T-1} c(X_t, \zeta_t)}\right],
\end{equation} where $\textbf{X}$ is the discrete time CMP (DTCMP)
 corresponding to the control $\zeta\in\Uadm$, with initial state $i$.
Our aim is to minimize \cref{EErgocost} over all admissible policies $\Uadm$. In other words,
we are interested in the quantity
\begin{equation}\label{lamstr}
\lambda^*= \, \inf_{i\in S}\, \inf_{\zeta\in \Uadm}\sE_i(c, \zeta).
\end{equation}
 A policy $\zeta^{*}\in \Uadm$ is said to be optimal if for all $i\in S$
$$\sE_i(c,\zeta^{*}) \, = \, \inf_{i\in S}\,\inf_{\zeta\in \Uadm}\sE_i(c, \zeta).$$
One of our chief goals in this article is to characterize all the optimal stationary Markov controls. 
\begin{assumption}\label{A1.1}
We impose the following conditions on the DTCMP 
\begin{itemize}
\item[(a)] For each $i\in S$ and any bounded measurable function $f:S\to\RR$ the maps $u\mapsto c(i, u)$ and $u \mapsto \sum_{j\in S} f(j)P(j | i, u)$ are continuous on $\Act(i)$\,.
\item[(b)] There exists a state $i_0\in S$ such that
\begin{equation*}
P(j|i_0, u) > 0 \quad \text{for all} \quad j\in S\setminus\{i_0\},\; u\in \Act(i_0)\,.
\end{equation*}   
\end{itemize}
\end{assumption}
\cref{A1.1}(a) is a quite routine assumption for discrete time CMP. \cref{A1.1}(b)
will be used to show that the sequence of Dirichlet eigenfunctions does not vanish in the limit
(see \cref{L2.4} below).
It is also possible to consider other type of condition instead \cref{A1.1}(b). We refer to
\cref{R2.3} for further discussion.

A function $g:S\to \RR$ is said to be \textit{norm-like} if for every $\kappa\in \RR$ the set
$\{i\in S : g(i)\leq \kappa\}$ is either empty or finite.
We also impose the following Foster–Lyapunov condition on the dynamics.

\begin{assumption}\label{EA2.2}
We assume that the DTCMP \textbf{X} is irreducible under every stationary Markov control in
$\Usm$. In (a) and (b) below the function $\Lyap$ on $S$ takes values in $[1, \infty)$ and
$\widehat{C}$ is a positive constant. We assume that one of the following holds.
\begin{itemize}
\item[(a)] For some positive constant $\beta\in (0,1)$ and a finite set $\cK$ it holds that 
\begin{equation}\label{EA2.2A}
\sup_{u\in\Act(i)} \sum_{j\in S} \Lyap(j) P(j|i,u) \le (1 - \beta)\Lyap(i) + \widehat{C}\, \Ind_{\cK}(i)\quad \forall \quad i\in S\,. 
\end{equation}
Also, assume that $\norm{c}_\infty\df\sup_{i\in S}\sup_{u\in\Act(i)}c(i,u) < \gamma$ where $\beta = (1-e^{-\gamma})$\, (i.e., $\gamma = \log(\frac{1}{1-\beta})$).

\item[(b)] For a finite set $\cK$ and a norm-like function $\ell:S\to \RR_+$ it
holds that
\begin{equation}\label{EA2.2B}
\sup_{u\in\Act(i)} \sum_{j\in S} \Lyap(j) P(j|i,u) \le \widehat{C} \Ind_{\cK}(i) + (1-\beta_i)\Lyap(i) \quad \forall \quad i\in S\,.
\end{equation}
where $1-e^{- \ell(i)} =\beta_i$. Moreover, the function $\ell(\cdot)-\max_{u\in\Act(\cdot)} c(\cdot, u)$ is norm-like.
\end{itemize} 
\end{assumption}
\cref{EA2.2B} will be useful to treat problems with unbounded running cost. Among others,
\cref{EA2.2B} implies that \eqref{lamstr} is finite. Similar condition is
also used by Balaji and Meyn \cite[Theorem~1.2]{BM00} in the study of multiplicative ergodicity. \eqref{EA2.2B} also used in the work
of Arapostathis et.\ al. \cite{ABS19} to study the ergodic risk-sensitive 
control of diffusions.
It is easily seen that $u\mapsto \sum_{j\in S}f(j)P(j|i,u)$ is lower-semicontinuous in $\Act(i)$ for all positive $f\in\order(\Lyap)$ and $i\in S$, where $\order(\Lyap)$ denotes the space of all functions $f$ satisfying
$\sup_{k\in S}\frac{|f|(k)}{\Lyap(k)}<\infty$. By $\sorder(\Lyap)$ we denote the subset
of $\order(\Lyap)$ consists of function $f$ satisfying 
$\lim_{k\to\infty}\frac{|f(k)|}{\Lyap(k)}=0$.

\begin{example}
\cref{A1.1,EA2.2} are satisfied by a large family of controlled Markov chains. To illustrate, we
consider the following elementary queueing model 
$$Q_{k+1}=[(1-\theta) Q_k-\zeta_k + A_{k+1}]_+\,, \quad k\geq 0,$$
where $\theta>0$ denotes the reneging rate. The control $\zeta_k$ takes integer values in
some bounded set and $\{A_k, k\geq 1\}$ is an i.i.d. sequence and the support of the common marginal distribution is equal to $\mathbb{Z}_+$. Also, assume that $\Exp[A_1]=a<\infty$.
It is easy to see that $i_0=0$ satisfies \cref{A1.1}(b). Take $\Lyap(i)=i +1$. Then it is easy to 
check that for any $v\in\Usm$, we have
$$\Exp^v_i[\Lyap(Q_1)]\leq 1 + (1-\theta) \Lyap(i) + a.$$
Thus, we can choose $\beta<\theta$ in \eqref{EA2.2A}.

Furthermore, if we assume that $a_1:=\log\Exp[e^{\gamma A_1}]<\infty$
for some $\gamma>0$, then letting $\Lyap(i) = e^{\gamma i}$ we see that
\begin{align*}
\Exp^v_i[\Lyap(Q_1)]\leq e^{\gamma(1-\theta) i}\Exp[e^{\gamma A_1}]
= e^{-\theta i + a_1}\Lyap(x)
\leq \widehat{C} \Ind_\cK + (1-(1-e^{-[\theta i - a_1]_+})) \Lyap(i),
\end{align*}
where $\cK=\{j\in S\; :\: \theta j - a_1\leq 0 \}$ and 
$\widehat{C}=\max_{j\in\cK}  e^{-\theta j + a_1}\Lyap(j)$. Thus \eqref{EA2.2B} holds for the choice of $\beta_i= (1-e^{-[\theta i - a_1]_+})$,
$i\in S$.
\end{example}

It can be easily shown that $\sE_i$ is finite for any $\zeta\in\Uadm$ under \cref{EA2.2}.
\begin{lemma}\label{L-extra}
Grant \cref{EA2.2}. Then there exists a constant $\kappa$ such that
\begin{equation}\label{EL2.3E}
\sE_i(c, \zeta)\leq \kappa \quad \text{for all}\; i\in S,\; \zeta\in\Uadm.
\end{equation}
\end{lemma}

\begin{proof}
We only provide a proof under \cref{EA2.2}(b) and the proof under \cref{EA2.2}(a)
is obvious since $c$ is bounded. Since $\cK$ is finite, for some constant $\kappa_1$ we can write \eqref{EA2.2B}
as 
\begin{equation}\label{EL2.3C}
\sup_{u\in\Act(i)} \sum_{j\in S} \Lyap(j) P(j|i,u) \le e^{\kappa_1-\ell(i)}\Lyap(i) \quad \forall \quad i\in S\,.
\end{equation}
Thus, by successive conditioning and using \cref{Markov1}, we deduce from \eqref{EL2.3C} that
\begin{equation}\label{EL2.3D}
\Exp_i^{\zeta}\left[e^{\sum_{t=0}^{T-1}(\ell(X_t)-\kappa_1)}\Lyap(X_T)\right] \le \Lyap(i)\,
\quad \text{for all}\; i\in S.
\end{equation}
Since $\Lyap\geq 1$, taking logarithm on both side of \eqref{EL2.3D}, dividing both sides by $T$ and
letting $T\to\infty$ we obtain
\begin{equation*}
\sE_i(\ell, \zeta)\leq \kappa_1\quad \text{for all}\; i\in S.
\end{equation*}
On the other hand, $\ell-\max_{u\in\Act(\cdot)} c(\cdot, u)$ is norm-like. Thus, for some constant $\kappa_2$,
we have $\max_{u\in\Act(i)}c(i, u)\leq \ell(i) +\kappa_2$ for all $i\in S$. Hence we obtain
\begin{equation*}
\sE_i(c, \zeta)\leq \kappa_1 + \kappa_2 \quad \text{for all}\; i\in S,\; \zeta\in\Uadm.
\end{equation*}
This completes the proof.
\end{proof}

Now we are ready to state our first main result of this section.
\begin{theorem}\label{T2.1}
Grant \cref{A1.1,EA2.2}. Then the following hold.
\begin{itemize}
\item[(i)] There exists a unique positive function $\psi^*$, $\psi^*(i_0)=1$, (where $i_0$ is a reference state as in \cref{A1.1}) satisfying
\begin{equation}\label{ET2.1A}
e^{\lamstr}\psi^*(i) = \min_{u\in\Act(i)}\left[e^{c(i,u)}\sum_{j\in S} \psi^*(j)P(j|i,u)\right]\quad\text{for}\,\, i\in S\,.
\end{equation}
\item[(ii)] A stationary Markov control $v\in\Usm$ is optimal if and only if it satisfies
\begin{equation}\label{ET2.1B}
\min_{u\in\Act(i)}\left[e^{c(i,u)}\sum_{j\in S} \psi^*(j)P(j|i,u)\right]
= \left[e^{c(i,v(i))}\sum_{j\in S} \psi^*(j)P(j|i,v(i))\right]\quad \text{for}\; i\in S.
\end{equation}
\end{itemize}
\end{theorem}
As mentioned in the introduction the existence part of \cref{T2.1} is not new and
existence of $\psi^*$ and stationary optimal Markov control has been obtained for a more
general class of DTCMP; for instance, CMP taking values in a general state space. Our contribution in
\cref{T2.1} comes from the uniqueness of $\psi^*$ and the verification result which shows that all the optimal stationary Markov controls are nothing but measurable selectors of \eqref{ET2.1B}.
 The rest of this section is devoted to the proof of
\cref{T2.1}. Our approach is quite  different from the one considered in \cite{BorMey02,DMS99,DMS07}.
We take a more direct approach by considering the {\it Dirichlet eigenvalue} problems on finite sets
and then pass to the limit by increasing the finite sets to $S$. This approach was first considered 
by Biswas in \cite{Biswas-11a} to study the risk-sensitive control problem for controlled diffusions
with a near-monotone cost function. The key steps in proving \cref{T2.1} can be summarized as follows:
we consider a collection of increasing finite sets $\sD_n$, increasing to $S$, and find a 
Dirichlet eigenpair $(\rho_n, \psi_n)$ in $\sD_n$ for every $n$(see \cref{L2.1} below). 
Then in \cref{L2.4} we 
show that there is a subsequence
of these eigenpairs converging to a positive eigenpair $(\psi^*, \rho)$
of \eqref{ET2.1A} and $\rho=\lamstr$ (see \cref{L2.6}). In \cref{L2.6} we then show that $\psi^*$ is unique upto a normalizing constant.

Let $\sD$ be a finite set in $S$ such that $i_0\in\sD$. Define a space 
$$\cB_{\sD} =\, \{f:S\to \RR \mid f \,\,\mbox{is Borel measurable and}\,\, f(i) = 0\,\,\forall \,\, i\in \sD^c\}\,.$$
We begin with the following standard result which is required to apply Kre\u{\i}n-Rutman theorem
in \cref{L2.1}.
\begin{proposition}\label{P1.1}
Suppose $c < 0$ in $\sD$. Then for each $f\in \cB_{\sD}$, there exists a unique solution 
$\phi\in\cB_{\sD}$ satisfying
\begin{equation}\label{EP1.1A}
\min_{u\in\Act(i)}\left[e^{c(i,u)}\sum_{j\in S}\phi(j)P(j|i,u) + f(i)\right] \,=\, \phi(i), \quad \forall\,\, i\in \sD\,, \quad\text{and\ } \phi(i) = 0\quad \forall\,\, i\in \sD^c\,.
\end{equation}
Moreover, the unique solution $\phi$ is given by
\begin{equation}\label{EP1.1B}
\phi(i) \,=\, \inf_{\zeta\in\Um}\Exp_i^{\zeta}\left[\sum_{k = 0}^{\uptau - 1} e^{\sum_{t=0}^{k-1}c(X_t, \zeta_t)} f(X_k)\right]\quad\forall \,\, i\in \sD\,,
\end{equation}
where $\uptau = \uptau(\sD)\,=\,\inf\{t>0\,\colon X_t\notin \sD\}$ denotes
the first exit time from $\sD$.
\end{proposition}

\begin{proof}
Fix $f\in\cB_{\sD}$. Define a map $\cT:\cB_\sD\to\cB_\sD$ as follows: for
$g\in \cB_\sD$, $Tg$ is given by
$$\cT g(i)= \min_{u\in\Act(i)}\left[e^{c(i,u)}\sum_{j\in S}g(j)P(j|i,u) + f(i)\right]\quad i\in \sD,\quad \cT g(i)=0\quad \text{for}\; i\in \sD^c\,.$$
Letting $\norm{g}_\sD=\max\{|g(i)|\; :\; i\in \sD\}$ we get from above that
$$\norm{\cT g_1-\cT g_2}_\sD\,\leq \vartheta \norm{g_1-q_2}_\sD\quad
\text{where} \quad \vartheta=\max_{i\in\sD}\max_{u\in\Act(i)} e^{c(i, u)}<1.$$
Thus $\cT$ is a contraction. From Banach fixed point theorem we find 
a unique $\phi$ satisfying \eqref{EP1.1A}. \eqref{EP1.1B} is standard and follows
from Dynkin's formula.

\end{proof}

%%%%%%%%%%%%%%%%%%%%%%%%%%%%%%%%%%%%%%%%%%%%%%%%%%%%%%%%%%%%%%%%%%%%%%%%%%%%%%%%%%%%%%%%%%%%%%%%%%%
Next we recall a version of the nonlinear Kre\u{\i}n-Rutman theorem from \cite[Section~3.1]{A18} (cf. \cite{Krein-Rutman}).
\begin{theorem}\label{T-KR}
Let $\mathcal{X}$ be an ordered Banach space and $\mathcal{C}\subset \mathcal{X}$ be a nonempty
closed cone
satisfying $\cX=\cC-\cC \, \df \{x-y : x,y\in \cC\}$. Let $\cT: \cX\to\cX$ be an order-preserving, $1$-homogeneous (that is, $\cT({\alpha} x) = {\alpha}\cT(x)$ for all $x\in \cX$ and ${\alpha}\in \RR_{+}$), completely continuous map such that for some nonzero $\xi\in\cC$ and $M > 0$, we have 
$M \cT(\xi) \succeq \xi$. Then there exist nontrivial $x_0\in\cC$ and $\lambda_0>0$ 
satisfying $\cT x_0 = \lambda_0 x_0$. 
\end{theorem}

Here $\succeq$ denotes the partial ordering in $\cX$ with respect to the cone $\cC$, that is,
$x\succeq y$ if and only if $x-y\in \cC$. Also, we recall that a map $\cT:\cX\to\cX$ is called completely continuous if it is continuous and compact.
We let $\cC = \cB_{\sD}^+\subset \cB_\sD$, the cone of nonnegative functions vanishing outside $\sD$.
Applying \cref{T-KR} we then establish the existence of an eigenpair to the Dirichlet problem
in $\sD$.
%%%%%%%%%%%%%%%%%%%%%%%%%%%%%%%%%%%%%%%%%%%%%%%%%%%%%%%%%%%%%%%%%%%%%%%%%%%%%%%%%%%%%%%%%%%%%%%%%%%%%%%%
\begin{lemma}\label{L2.1}
There exists a pair $(\lambda_{\sD}, \psi_{\sD})\in\RR_+\times \cB^+_{\sD}$
, $\psi_D\neq 0$,
satisfying
\begin{equation}\label{EL2.1A}
\lambda_{\sD}\, \psi_{\sD}(i) = \min_{u\in\Act(i)}\left[e^{c(i,u)}\sum_{j\in S} \psi_{\sD}(j) P(j|i,u) \right].
\end{equation}
Moreover, we have
\begin{equation}\label{EL2.1B}
\rho_{\sD}\df\log\lambda_{\sD} \le \inf_{\zeta\in\Uadm}\limsup_{T\to\infty}\frac{1}{T}\log \Exp_i^{\zeta} \left[e^{\sum_{t = 0}^{T-1} c(X_t, \zeta_t)}\right]\,,
\end{equation}
for all $i$ satisfying $\psi_\sD(i)\neq 0$.
\end{lemma}

\begin{proof}
Suppose $c < -\delta$ in $\sD$, for some positive constant $\delta$. Define an operator 
$\cT :\cB_{\sD}\to\cB_{\sD}$ by 
\begin{equation}\label{EL2.1C}
\cT(f)(i) \,\df\, \phi(i) = \inf_{\zeta\in\Um}\Exp_i^{\zeta}\left[\sum_{k = 0}^{\uptau - 1} e^{\sum_{t=0}^{k-1}c(X_t, \zeta_t)} f(X_k)\right]\quad\text{for} \,\, i\in \sD,
\quad \text{and}\quad \phi(i)=0\quad \text{for}\; i\in\sD^c.
\end{equation} 
In view of \cref{P1.1}, it is clear that the map $\cT$ is well defined. From \cref{EL2.1C}, it is
also easily seen that $\cT(\lambda f) = \lambda \cT(f)$ for all $\lambda \geq 0$. Again, since $c < -\delta$, from the definition of $\cT$ it is straightforward to check that 
\begin{equation*}
\|\cT(f) - \cT(g)\|_{\sD} \le \kappa_1 \| f - g \|_{\sD}\,,
\end{equation*}
for some constant $\kappa_1 > 0$. Therefore $\cT:\cB_{\sD}\to\cB_{\sD}$ is continuous. 

Let $f,g\in\cB_{\sD}$ such that $f \succeq g$. Then, we have
\begin{equation*}
(\cT(f) - \cT(g))(i) \ge \ \inf_{\zeta\in\Um}\Exp_i^{\zeta}\left[\sum_{k = 0}^{\uptau - 1} e^{\sum_{t=0}^{k-1}c(X_t, \zeta_t)} (f - g)(X_k)\right]\ge 0.
\end{equation*} 
Thus, we obtain $T(f) \succeq T(g)$. Let $\{f_m\}$ be a bounded sequence in $\cB_{\sD}$. From
\eqref{EL2.1C} we then have $\|\phi_m\|_{\sD} \le \kappa_{2}$ for some constant $\kappa_{2} > 0$, where $\phi_m=\cT f_m$. Then by a standard diagonalization argument we deduce that there exists
a  $\phi\in\cB_{\sD}$ satisfying $\|\phi_{m_k} - \phi\|_{\sD} \to 0$, as $m_k\to\infty$,
for some subsequence $\{\phi_{m_k}\}$. This implies that $\cT:\cB_{\sD}\to\cB_{\sD}$ is compact and
hence $\cT$ is completely continuous.  

Let $f\in\cB_{\sD}$ be such that $f(i_0)=1$ and $f(j) = 0$ for all $j\neq i_0$.
By \eqref{EP1.1A} we then have $\phi(i_0)\geq f(i_0)>0$. Thus $\cT f\succeq f$.

Thus we can apply \cref{T-KR} to find a nonzero $\psi_\sD\in\cC$ and $\lambda>0$
such that $\cT(\psi_\sD) = \lambda\psi_\sD$. Applying \cref{P1.1}
we obtain
\begin{equation*}
\min_{u\in\Act(i)}\left[\lambda e^{c(i,u)}\sum_{j\in S}\psi_\sD(j)P(j|i,u) + \psi_\sD(i)\right] \,=\, \lambda\psi_\sD(i), \quad \forall\,\, i\in \sD\,, \quad\text{and\ } \psi_\sD(i) = 0\quad \forall\,\, i\in \sD^c\,.
\end{equation*}
Defining $\lambda_\sD=\frac{\lambda-1}{\lambda}$,
we get from above that
\begin{equation}\label{EL2.1D}
\min_{u\in\Act(i)}\left[e^{c(i,u)}\sum_{j\in S}\psi_\sD(j) P(j|i,u)\right] = \lambda_\sD\psi_\sD(i)\quad\forall\,\, i\in\sD\,.
\end{equation}
Since $\psi_\sD\geq 0$ and $\psi_\sD(i)>0$ for some $i\in \sD$, it follows from \eqref{EL2.1D} that $\lambda_{\sD}\ge 0$. For general $c$, we can start with $c-\norm{c}_\sD-\delta$ and then  
multiply both sides of \eqref{EL2.1D} with $e^{\norm{c}_\sD + \delta}$.
This gives \eqref{EL2.1A}. 

To prove \eqref{EL2.1B} consider a state $i$ satisfying $\psi_\sD(i)\neq 0$. Note that there is nothing to
prove if $\lambda_\sD=0$. So we assume $\lambda_\sD>0$. Now,
due to \eqref{Markov1}, for any bounded function $F$ we have
\begin{equation}\label{EL2.1E}
Y_n \df \sum_{k=1}^{n-1} \bigl(e^{\sum_{t=0}^{k-1} c(X_t, \zeta_t)}F(X_k)
-e^{\sum_{t=0}^{k-1} c(X_t, \zeta_t)} \sum_{j\in S} F(j) P(j|X_{k-1} , \zeta_{k-1})\bigr)
\end{equation}
is a $\sF_n=\fB(\sH_n)$ martingale. Thus by optional sampling theorem, 
$\{Y_{n\wedge \uptau}, \sF_{n\wedge\uptau}\}$ is also a Martingale. Since, by \eqref{EL2.1D},
$$e^{\sum_{t=0}^{k} (c(X_t, \zeta_t)-\rho_\sD)} \sum_{j\in S} \psi_\sD(j) P(j|X_{k} , \zeta_{k})
-e^{\sum_{t=0}^{k-1} (c(X_t, \zeta_t)-\rho_\sD)} \psi_\sD(X_{k})\ge 0\quad \text{for}\; k=0, \ldots, \uptau-1,$$
($\rho_\sD$ is given by \cref{EL2.1B}) it follows that
\begin{equation*}
\psi_{\sD}(i) \le \Exp_{i}^{\zeta}\left[ e^{\sum_{t=0}^{T-1}(c(X_t,\zeta_{t})-\rho_\sD)}\psi_{\sD}(X_{T})\Ind_{\{T < \uptau\}}\right]\,.
\end{equation*} 
This in turn, gives
\begin{equation*}
e^{T\rho_{\sD}}\psi_{\sD}(i) \le \max_{\sD}\psi_\sD\,\Exp_{i}^{\zeta}\left[ e^{\sum_{t=0}^{T-1}c(X_t,\zeta_{t})} \right]\,.
\end{equation*} Now taking logarithm both sides, dividing by $T$ and letting $T\to \infty$, we obtain
\begin{equation*}
\rho_{\sD} \le \limsup_{T\to\infty}\frac{1}{T}\log \Exp_{i}^{\zeta}\left[ e^{\sum_{t=0}^{T-1}c(X_t,\zeta_{t})}\right]\,.
\end{equation*} Since $\zeta\in\Uadm$ is arbitrary we have \cref{EL2.1B}. This completes the proof.
\end{proof}

%%%%%%%%%%%%%%%%%%%%%%%%%%%%%%%%%%%%%%%%%%%%%%%%%%%%%%%%%%%%%%%%%%%%%%%%%%%%%%%%%%%%%%%%%%%%%%%%%%%%
Now we prove certain estimate which will play important role in our subsequent analysis.
Let $\sD_{n}$ be an increasing sequence of finite subsets of $S$ such that $\cup_{n=1}^{\infty} \sD_{n} = S$. With no loss of generality, we assume that $i_0\in\sD_n$ for all $n$.

\begin{lemma}\label{L2.2}
Suppose that \cref{EA2.2} holds and consider a finite set $\sB\subset S$ containing $\cK$.
Then for any $\zeta\in\Uadm$ we have the following.
\begin{itemize}
\item[(i)] Under \cref{EA2.2}(a) we have
\begin{equation}\label{EL2.2A}
\Exp_i^{\zeta}\left[e^{\gamma \uuptau}\Lyap(X_{\uuptau})\right] \,\le\, \Lyap(i)\ \ \text{ for all\ } i\in\sB^c\,,
\end{equation}
where $\uuptau=\uuptau(\sB)\,=\,\inf\{t>0\,\colon X_t\in \sB\}$\,.
\item[(ii)] Under \cref{EA2.2}(b) we have
\begin{equation}\label{EL2.2B}
\Exp_i^{\zeta}\left[e^{\sum_{t=0}^{\uuptau-1}\ell(X_t)}\Lyap(X_{\uuptau})\right] \,\le\, \Lyap(i)\ \ \text{ for all\ } i\in\sB^c\,.
\end{equation}
\end{itemize}
\end{lemma}
\begin{proof}
We only show (ii) and (i) follows by setting $\ell=\gamma$. Fix $\sD_n$ so that $\sB\subset \sD_n$ and
$i\in \sD_n\setminus\sB$. Choose $m$ large enough so that 
\begin{equation}\label{EL2.2C}
\sup_{u\in\Act(i)} e^{\ell(i)}\sum_{j\in S} \Lyap_m(j) P(j|i,u) \le \Lyap_m(i) \quad 
\text{for all} \quad i\in \sD_n\setminus\sB\,,
\end{equation}
where $\Lyap_m=\min\{\Lyap, m\}$. Taking $F=\Lyap_m$ in \eqref{EL2.1E}, it follows that
$\{Y_{t\wedge\uptau_\circ}, \sF_{t\wedge\uptau_\circ}\}$ is a martingale where 
$\uptau_\circ=\uuptau\wedge\uptau_n$ and $\uptau_n$ is the first exit time of {\bf X} from $\sD_n$.
Using \eqref{EL2.2C} we thus obtain
$$\Exp^\zeta_{i}\left[ e^{\sum_{t=0}^{T\wedge\uuptau\wedge\uptau_n-1}\ell(X_t)}
\Lyap_m(X_{T\wedge\uuptau\wedge\uptau_n})\right]
\leq \Lyap_m(i).$$
First we let $m\to\infty$ and then $T\to\infty$ to obtain 
$$\Exp^\zeta_{i}\left[ e^{\sum_{t=0}^{\uuptau\wedge\uptau_n-1}\ell(X_t)}\Lyap(X_{\uuptau})\right]
\leq \Lyap(i).$$
Now let $n\to\infty$ and use Fatou's lemma to obtain \eqref{EL2.2B}.
\end{proof}
%%%%%%%%%%%%%%%%%%%%%%%%%%%%%%%%%%%%%%%%%%%%%%%%%%%%%%%%%%%%%%%%%%%%%%%%%%%%%%%%%%%%%%%%%%%%%%%%%%%%%%
Denote by $(\rho_n, \psi_n)$ the eigenpair in the domain $\sD_n$, obtained by \cref{L2.1}. Next
we are interested to find a limit of these eigenpairs as $n\to\infty$. Recall that in case
of nondegenerate controlled diffusion such limits are easily obtained by applying Harnack's
inequality and monotonicity property $\rho_n$ (cf. \cite[Lemma~2.1]{Biswas-11a},\cite[Theorem~3.4]{AB18}). Such tools are not available in the current situation. Below we produce a different argument to pass this limit.

\begin{lemma}\label{L2.3}
Grant \cref{A1.1,EA2.2}. Let $(\rho_n, \psi_n)$ be the Dirichlet eigenpair in $\sD_n$ satisfying
\begin{equation}\label{EL2.3A}
e^{\rho_{n}}\psi_{n}(i) = \min_{u\in\Act(i)}\left[e^{c(i,u)}\sum_{j\in S} \psi_{n}(j) P(j|i,u)\right]\quad\forall\,\, i\in \sD_n\,.
\end{equation}
Then the following hold.
\begin{itemize}
\item[(i)]
\begin{equation}\label{EL2.3B}
\limsup_{n\to\infty} \rho_{n} \le \inf_{\zeta\in\Uadm}\limsup_{T\to\infty}\frac{1}{T}\log \Exp_{i_0}^{\zeta}\left[ e^{\sum_{t=0}^{T-1}c(X_t,\zeta_{t})}\right]\,.
\end{equation}
Furthermore, $\rho_n$ is bounded for all large $n$.
\item[(ii)] $\liminf_{n\to\infty}\rho_n \ge 0$\,.
\end{itemize} 
\end{lemma}

\begin{proof}
First we consider (i). \eqref{EL2.3B} follows from \eqref{EL2.1B} since $\psi_n(i_0)>0$ for all $n$,
by \cref{A1.1}(b).
In view of \cref{L2.1} and \eqref{EL2.3E},  we see that $\rho_n$ is bounded 
from above by  $\kappa$. Next we show that $\rho_n$ is bounded below. 
Suppose, on the contrary, that along a subsequence $\rho_n\to - \infty$ as $n\to \infty$. This implies that $\rho_n < 0$ for all large enough $n$.
Using \cref{A1.1}(b) and \eqref{EL2.3A}, we have $\psi_n(i_0) > 0$. 
Dividing  $\psi_n$ by $\psi_n(i_0)$ we can ensure that $\psi_n(i_0) = 1$ for all $n$. Rewriting \cref{EL2.3A}, we obtain
\begin{equation}\label{EL2.3F}
1 = \psi_n(i_0) = e^{- \rho_n}\min_{u\in\Act(i_0)}\left[ e^{c(i_0, u)}\sum_{j\in S}\psi_n(j) P(j|i_0, u)\right]\,.
\end{equation}
 Since $(c(i_0,u) - \rho_n) > 0$ for all $n$ large , we get
\begin{equation}\label{EL2.3G}
1 \ge \sum_{j\in S}\psi_n(j) P(j|i_0, \hat{v}_n(i_0))\,,
\end{equation} 
for any minimizing selector $\hat{v}_n$ of \cref{EL2.3F}.
This in turn, implies that 
$$\psi_n(j)\leq \sup_{u\in \Act(i_0)} \frac{1}{P(j|i_0, u)}\quad \text{for}\; j\neq i_0,$$
for all $n$ large.
Applying a standard diagonalization argument we find a non-negative function $\psi$ with $\psi(i_0) = 1$ such that along a further subsequence $\psi_n \to\psi$ componentwise. Also, since $\Act(i)$ is
compact for each $i\in S$, we have $\hat{v}_n\to \hat{v}$ along a further subsequence. Hence letting $n\to\infty$ in \cref{EL2.3G} we obtain 
\begin{equation}\label{EL2.3H}
1 \ge \sum_{j\in S}\psi(j) P(j|i_0, \hat{v}(i_0))\,.
\end{equation}
Writing \eqref{EL2.3A} as
$$e^{\rho_n} \psi_n(i) \geq \left[ \sum_{j\in S}\psi_n(j) P(j|i, v_n(i))\right]$$
and letting $n\to \infty$ we obtain
$$0\geq \sum_{j\in S}\psi(j) P(j|i, \hat{v}(i)).$$
Choosing $i=i_0$ and applying \cref{A1.1}(b) it follows that $\psi(i_0)=1$ and $\psi(j)=0$ for all $j\neq i_0$. It is also easily seen that $\{\psi(X_n), \sF_n\}$ is a super-martingale where ${\bf X}$ is the
Markov process under the stationary Markov control $\hat{v}$. Hence by Doob's martingale convergence
theorem $\psi(X_n)\to Y$ almost surely, as $n\to\infty$. On the other hand, ${\bf X}$ is recurrent, which follows from \cref{EA2.2}, ${\bf X}$ visits every state (in particular, $i_0$) of $S$ infinitely often. Thus,
$\psi(X_n)$ can not converge. This is a contradiction. Hence $\rho_n$ must be bounded from below. This
completes the proof of (i).

Next we consider (ii).
Suppose that $\hat{\rho} = \liminf_{n\to\infty} \rho_n < 0$. Then along a suitable subsequence $\rho_n$ converges to $\hat{\rho}$. Hence, for all $n$ large enough we have $(c(i, u) - \rho_n) > 0$. Thus, repeating the above argument we find a nonnegative function $\phi: S\to\RR$ with $\phi(i_0) = 1$ satisfying (see \eqref{EL2.3H})
\begin{equation*}
\phi(i) \ge \Exp_{i}^{\hat{v}}\left[\phi(X_1)\right]\quad\forall\,\, i\in S\,,
\end{equation*}
for some stationary Markov control $\hat{v}$.
This in turn implies that $\{\phi(X_t)\}$ is a supermartingale and therefore, the above argument
gives us $\phi\equiv 1$. Now passing limit in \eqref{EL2.3F} and using Fatou's lemma we have
$$1=\phi(i)\geq e^{c(i, \hat{v}(i))-\hat{\rho}}>1\,.$$
This is a contradiction. This gives us (ii).
\end{proof}

\begin{remark}\label{R2.1}
It should be observed from the proof of \cref{L2.3} that if we assume 
$\inf_{\zeta\in \Uadm}\sE_i(c, \zeta)$ to be finite for every $i$ then the conclusion of \cref{L2.3}
holds, provided the CMP is recurrent under every stationary Markov control and \cref{A1.1} holds.
\end{remark}

%%%%%%%%%%%%%%%%%%%%%%%%%%%%%%%%%%%%%%%%%%%%%%%%%%%%%%%%%%%%%%%%%%%%%%%%%%%%%%%%%%%%%%%%%%%%%%%%%%%%%
Now with the help of \cref{L2.3} we can pass the limit in $(\rho_n, \psi_n)$ to obtain
the following result.
\begin{lemma}\label{L2.4}
Grant \cref{A1.1,EA2.2}. Then there exists $(\rho, \psi^*)\in\RR_+\times\order(\Lyap)$\,, $\psi^* > 0$, satisfying
\begin{equation}\label{EL2.4A}
e^{\rho}\psi^*(i) = \min_{u\in\Act(i)}\left[e^{c(i,u)}\sum_{j\in S} \psi^*(j) P(j|i,u)\right]\quad\forall\,\, i\in S\,.
\end{equation} 
Moreover, we have the following.
\begin{itemize}
\item[(i)]
\begin{equation}\label{EL2.4B}
\rho \le \inf_{\zeta\in\Uadm}\limsup_{T\to\infty}\frac{1}{T}\log \Exp_{i}^{\zeta}\left[ e^{\sum_{t=0}^{T-1}c(X_t,\zeta_{t})}\right]\quad \text{for all}\; i\in S.
\end{equation}
In particular, $\rho\leq \inf_{i\in S}\, \inf_{\zeta\in \Uadm}\sE_i(c, \zeta)=\lamstr$.
\item[(ii)] There exists a finite set $\sB$ containing $\cK$ such that for any 
minimizing selector $v^*$ of \cref{EL2.4A}, it holds that
\begin{equation}\label{EL2.4C}
\psi^*(i) = \Exp_i^{v^*} \left[e^{\sum_{t=0}^{\uuptau(\sB) - 1}(c(X_t, v^*(X_t)) - \rho)}\psi^*(X_{\uuptau(\sB)})\right]\quad\forall \,\, i\in \sB^c\,.
\end{equation}  
\end{itemize}
\end{lemma}

\begin{proof}
Since $\{\rho_n\}$ is a bounded sequence and $\liminf_{n\to\infty}\rho_n \ge 0$, by \cref{L2.3},
there exist a subsequence, denoted by $\{\rho_n\}$ itself,
such that $\rho_n\to\rho$, as $n\to\infty$, and $\rho \ge 0$\,. Again, since $c \geq 0$ we can choose a finite set $\sB$ containing $\cK$ such that 
\begin{itemize}
\item under \cref{EA2.2}(a), since $\|c\|_{\infty} < \gamma$, we have
\begin{equation}\label{EL2.4D}
(\max_{u\in\Act(i)}c(i,u) - \rho_n) < \gamma\quad\forall\,\, i\in\sB^c\,, \quad \text{for all large}\; n;
\end{equation}
\item under \cref{EA2.2}(b), since $\ell(\cdot)-\max_{u\in\Act(\cdot)} c(\cdot, u)$ is norm-like, we have
\begin{equation}\label{EL2.4E}
(\max_{u\in\Act(i)}c(i,u) - \rho_n) < \ell(i)\quad\forall\,\, i\in\sB^c\,,
 \quad \text{for all large}\; n.
\end{equation}
\end{itemize}
Now we scale $\psi_n$ by multiplying a suitable scalar so that it touches $\Lyap$ from below.
To do so, define 
\begin{equation*}
\theta_n \,=\, \sup\{ \kappa>0 \;\colon\; (\Lyap - \kappa\psi_n) > 0 \quad\text{in\ } S\}\,.
\end{equation*} 
Since $\psi_n$ vanishes in $\sD^c_n$ and $\psi_n\gneq 0$, it is easily seen that $\theta_n$ is finite.
We replace $\psi_n$ by $\theta_n\psi_n$ and claim that $\psi_n$ touches $\Lyap$ inside $\sB$\,. Suppose, on the contrary, that the claim is not true. Then there exists a state $i_1\in\sB^c$ so that $(\Lyap - \psi_n)(i_1) = 0$ and $\Lyap-\psi_n>0$ in $\sB\cup \sD^c_n$.
 For all $n$ large enough, from \cref{EL2.3A} and \eqref{EL2.4D}, for any $\zeta\in\Usm$ we have
 under \cref{EA2.2}(a)
\begin{align*}
\psi_n(i_1) & \le \Exp_{i_1}^{\zeta}\left[e^{\sum_{t=0}^{\uuptau(\sB)\wedge N - 1}(c(X_t, \zeta_t) - \rho_n)}\psi_n(X_{(\uuptau(\sB)\wedge N)}) \Ind_{\{\uuptau(\sB)\wedge N < \uptau_n\}}\right]\\
& \le \Exp_{i_1}^{\zeta}\left[e^{\sum_{t=0}^{\uuptau(\sB)\wedge N - 1}\gamma}\psi_n(X_{(\uuptau(\sB)\wedge N)}) \Ind_{\{\uuptau(\sB)\wedge N < \uptau_n\}}\right]\,,
\end{align*} where $\uptau_n = \uptau(\sD_n)$\,. Since $\psi_n \le \Lyap$, using \cref{L2.2} and dominated convergence, we let $N\to\infty$ to obtain
\begin{equation}\label{EL2.4F}
\psi_n(i_1) \le \Exp_{i_1}^{\zeta}\left[e^{\gamma\uuptau(\sB)}\psi_n(X_{\uuptau(\sB)})\right]\,.
\end{equation} 
Combining \cref{EL2.2A} and \cref{EL2.4F}, we deduce that
\begin{equation*}
0 = (\Lyap - \psi_n)(i_1) \ge \Exp_{i_1}^{\zeta}\left[e^{\gamma\uuptau(\sB)}(\Lyap - \psi_n)(X_{\uuptau(\sB)})\right] > 0\,.
\end{equation*}
This is a contradiction. Thus $\psi_n$ touches $\Lyap$ inside $\sB$\,. In view of \cref{EL2.4E}, it is easily seen that a similar conclusion holds under \cref{EA2.2}(b). 

Thus, by a standard diagonalization argument, there exist $\psi^*\le\Lyap$ such that 
$\psi_n(i)\to\psi^*(i)$ , as $n\to \infty$, for all $i\in S$. Next we show that
\begin{equation}\label{EL2.4G}
\lim_{n\to\infty} \min_{u\in\Act(i)}\left[ e^{c(i,u)}\sum_{j\in S}\psi_n(j) P(j|i, u)\right] \to \min_{u\in\Act(i)}\left[ e^{c(i,u)}\sum_{j\in S}\psi^*(j) P(j|i, u)\right]\quad \forall\,\, i\in S\,.
\end{equation}
Since $\psi_n\leq \Lyap$, by dominated convergence theorem, for any $\zeta_0\in\Act(i)$ we have
\begin{align*}%\label{ET1.3D}
\limsup_{n\to\infty} \min_{u\in\Act(i)}\left[ e^{c(i,u)}\sum_{j\in S}\psi_n(j) P(j|i, u)\right] & \le \limsup_{n\to\infty}\left[ e^{c(i,\zeta_0)}\sum_{j\in S}\psi_n(j) P(j|i, \zeta_0)\right]\nonumber\\
& = \left[ e^{c(i,\zeta_0)}\sum_{j\in S}\psi^*(j) P(j|i, \zeta_0)\right].
\end{align*}
From the arbitrariness of $\zeta_0$ it then follows that
\begin{equation}\label{EL2.4H}
\limsup_{n\to\infty} \min_{u\in\Act(i)}\left[ e^{c(i,u)}\sum_{j\in S}\psi_n(j) P(j|i, u)\right]
\leq \min_{u\in\Act(i)}\left[ e^{c(i,u)}\sum_{j\in S}\psi^*(j) P(j|i, u)\right].
\end{equation}
Again, let $v_n$ be a minimizing selector  of \cref{EL2.3A}. Since $\Act(i)$ is compact,  by a standard diagonalization argument we have along some subsequence $v_n(i)$ converges to $v(i)$ in $\Act(i)$ for all $i\in S$\,. Thus by generalized Fatou's lemma \cite[Lemma~8.3.7]{HL99}, we deduce that
\begin{align}\label{EL2.4I}
\lim_{n\to\infty} \left[ e^{c(i,v_n(i))}\sum_{j\in S}\psi_n(j) P(j|i, v_{n}(i))\right] = \left[ e^{c(i,v(i))}\sum_{j\in S}\psi^*(j) P(j|i, v(i))\right]\,.
\end{align} 
Combining \cref{EL2.4H} and \cref{EL2.4I} we thus obtain  \cref{EL2.4G}. Therefore, letting $n\to\infty$ in \cref{EL2.3A}, we see that the pair $(\rho, \psi^*)\in\RR_+\times\order(\Lyap)$ satisfies
\begin{equation}\label{EL2.4J}
e^{\rho}\psi^*(i) = \inf_{u\in\Act(i)}\left[e^{c(i,u)}\sum_{j\in S} \psi^*(j) P(j|i,u)\right]\quad\forall\,\, i\in S\,.
\end{equation}
Moreover, since $(\Lyap - \psi_n) = 0$ at some point in $\sB$ for all $n$ large enough, we have $(\Lyap - \psi^*) = 0$ at some point in $\sB$. Since $\Lyap \ge 1$, we have $\psi$ is nontrivial. In fact, we have $\psi^* > 0$. If not, let $\psi^*(j) =0$ for some $j\in S$. Then for any minimizing selector $v$ of \cref{EL2.4J}, $\psi^*$ satisfies
\begin{equation}\label{EL2.4JJ}
e^{c(i, v(i))-\rho}\Exp_{i}^{v}\left[\psi^*(X_1)\right] = \psi^*(i) \quad\forall \,\, i\in S\,.
\end{equation}
Let $\psi^*(i)>0$ for some $i\in \sB$.
Since {\bf X} is irreducible under $v$, there exists a $n\in\NN$ and distinct $i_1, i_2,\ldots, i_{n}\in S$ satisfying 
$$P(i|i_{n}, v(i_n))P(i_n|i_{n-1},v(i_{n-1}))\cdots P(i_1|j,v(j))>0.$$
From \eqref{EL2.4JJ} this implies that $0=\psi^*(j)=\psi^*(i_1)=\ldots=\psi^*(i_n)=\psi^*(i)$ which is a 
contradiction. Thus we must have $\psi^*>0$ in $S$.
This gives us \eqref{EL2.4A}.

Next we consider (i).
Since $\psi_n\to\psi^*>0$ as $n\to\infty$, for any given state $i\in S$ we have
$\psi_n(i)>0$ for all large $n$. Now \eqref{EL2.4B} follows from \eqref{EL2.1B}
and the fact $\rho_n\to  \rho$, as $n\to \infty$.

Now we consider (ii).
Suppose that \cref{EA2.2}(a) holds. Let $v^*\in\Usm$ be a minimizing selector of \eqref{EL2.4J}.
Using \cref{EL2.3A}, for all $n$ large enough, we have (see \eqref{EL2.1E})
\begin{align}\label{EL2.4K}
\psi_n(i) \le & \Exp_{i}^{v^*}\left[e^{\sum_{t=0}^{\uuptau(\sB)\wedge N - 1}(c(X_t, v^*(X_t)) - \rho_n)}\psi_n(X_{(\uuptau(\sB)\wedge N)}) \Ind_{\{\uuptau(\sB)\wedge N < \uptau_n\}}\right]\nonumber\\
\le & \Exp_{i}^{v^*}\left[e^{\sum_{t=0}^{\uuptau(\sB) - 1}(c(X_t, v^*(X_t)) - \rho_n)}\psi_n(X_{(\uuptau(\sB))}) \Ind_{\{\uuptau(\sB) < \uptau_n \wedge N\}}\right]\nonumber\\ 
& + \Exp_{i}^{v^*}\left[e^{\sum_{t=0}^{N - 1}(c(X_t, v^*(X_t)) - \rho_n)}\psi_n(X_{N}) \Ind_{\{N < \uptau_n \wedge \uuptau(\sB)\}}\right]
\end{align}
for all $i\in\sB^c\cap\sD_n$. Since $\psi_n \le \Lyap$, we obtain
\begin{align}\label{EL2.4L}
\Exp_{i}^{v^*}\left[e^{\sum_{t=0}^{N - 1}(c(X_t, v^*(X_t)) - \rho_n)}\psi_n(X_{N}) \Ind_{\{N < \uptau_n \wedge \uuptau(\sB)\}}\right] &\le e^{N(\|c\|_{\infty} - \rho_n -\gamma )}\Exp_{i}^{v^*}\left[e^{N\gamma}\Lyap(X_{N}) \Ind_{\{N < \uptau_n \wedge \uuptau(\sB)\}}\right] \nonumber\\
&\le e^{N(\|c\|_{\infty} - \rho_n -\gamma )}\Lyap(i)
\end{align}
by \eqref{EL2.2A}. Letting $N\to \infty$ in \cref{EL2.4L}, it follows that
\begin{equation*}
\lim_{N\to\infty} \Exp_{i}^{v^*}\left[e^{\sum_{t=0}^{N - 1}(c(X_t, v^*(X_t)) - \rho_n)}\psi_n(X_{N}) \Ind_{\{N < \uptau_n \wedge \uuptau(\sB)\}}\right] = 0\,.
\end{equation*}
Thus taking limit $N\to \infty$ in \cref{EL2.4K}, we deduce
\begin{equation}\label{EL2.4M}
\psi_n(i) \le  \Exp_{i}^{v^*}\left[e^{\sum_{t=0}^{\uuptau(\sB) - 1}(c(X_t, v^*(X_t)) - \rho_n)}\psi_n(X_{\uuptau(\sB)}) \Ind_{\{\uuptau(\sB) < \uptau_n\}}\right].
\end{equation} 
Again, since $\psi_n \le \Lyap$, using \cref{L2.2} and dominated convergence theorem,
 we let $n\to\infty$ in \eqref{EL2.4M} to obtain
\begin{equation}\label{EL2.4N}
\psi^*(i) \le \Exp_{i}^{v^*}\left[e^{\sum_{t=0}^{\uuptau(\sB) - 1}(c(X_t, v^*(X_t)) - \rho)}\psi^*(X_{\uuptau(\sB)})\right]\quad\forall\,\, i\in\sB^c\,.
\end{equation} 
On the other hand, from \cref{EL2.4J}, we have
\begin{equation*}
\psi^*(i) = \Exp_{i}^{v^*}\left[e^{\sum_{t=0}^{\uuptau(\sB)\wedge N - 1}(c(X_t, v^*(X_t)) - \rho)}\psi^*(X_{(\uuptau(\sB)\wedge N)}) \right]\quad\forall\,\, i\in\sB^c\,.
\end{equation*} 
Thus letting $N\to\infty$ and using Fatou's lemma, we obtain
\begin{equation}\label{EL2.4O}
\psi^*(i) \ge \Exp_{i}^{v^*}\left[e^{\sum_{t=0}^{\uuptau(\sB) - 1}(c(X_t, v^*(X_t)) - \rho)}\psi^*(X_{\uuptau(\sB)})\right]\quad\forall\,\, i\in\sB^c\,.
\end{equation} 
Combining \cref{EL2.4N} and \cref{EL2.4O} we get \cref{EL2.4C}. 
A similar argument works under \cref{EA2.2}(b). Hence the proof. 
\end{proof}

\begin{remark}\label{R2.2}
It is easy to check that one can also apply the argument of \cref{L2.4} for every stationary Markov control. More precisely, if we impose \cref{A1.1,EA2.2},
then for every Markov control $v\in\Usm$ there exists $(\rho_v,\psi_v)\in\RR_+\times\order(\Lyap)$,
$\psi_v>0$, satisfying
\begin{equation}\label{ER2.2A}
e^{\rho_v}\psi_v(i) = e^{c(i,v(i))}\sum_{j\in S} \psi_v(j) P(j|i,v(i))\quad\forall\,\, i\in S\,.
\end{equation}
Furthermore, $\rho_v\leq \inf_i\,\sE_i(c,v)$ and for some finite set $\sB\supset\sK$ 
\begin{equation}\label{ER2.2B}
\psi_v(i) = \Exp_i^{v} \left[e^{\sum_{t=0}^{\uuptau(\sB) - 1}(c(X_t, v(X_t)) - \rho_v)}\psi_v(X_{\uuptau(\sB)})\right]\quad\forall \,\, i\in \sB^c\,.
\end{equation}
\end{remark}
%%%%%%%%%%%%%%%%%%%%%%%%%%%%%%%%%%%%%%%%%%%%%%%%%%%%%%%%%%%%%%%%%%%%%%%%%%%%%%%%%%%%%%%%%%%%%%%%%%%%%%%
Next we show that $\rho=\lamstr$. The proof in the controlled diffusion setting uses 
Girsanov transformation and ergodicity property of the twisted process (cf. 
Arapostathis et.\ al.\ \cite[Theorem~3.2]{ABS19}). It seems difficult to apply 
a similar approach in 
the present setting. To overcome this difficulty we follow an approach used in \cite[Lemma~4.4]{AB-19}. It should also be kept in mind that the Dirichlet eigenvalues obtained in \cite{AB-19} has certain monotonicity property which is
not very clear in the present setting.

\begin{lemma}\label{L2.5}
Assume \cref{A1.1} and \cref{EA2.2}(a). Let $(\rho, \psi^*)$ be the eigenpair in \eqref{EL2.4A}. Then
$\rho<\norm{c}_\infty$, provided $c$ is not a constant.
\end{lemma}

\begin{proof}
Note that $\rho\leq \norm{c}_\infty$, by \eqref{EL2.4B}. We
suppose, on the contrary, that  $\rho=\norm{c}_\infty$. Let $v^*$ be a minimizing selector of \eqref{EL2.4A}. It then follows from \eqref{EL2.4A} that
$$\psi^*(i)= e^{c(i, v^*(i))-\rho} \Exp_i^{v^*}[\psi^*(X_1)]\leq \Exp_i^{v^*}[\psi^*(X_1)]\quad \forall\; i\in S.$$
Thus, $\{\psi^*(X_t), \sF_t\}$ is a sub-martingale. On the other hand, \eqref{EL2.4C} implies that
$\psi^*$ is bounded. Hence by Doob's theorem, $\psi^*(X_t)$ must converge as $t\to\infty$. Since
{\bf X} is recurrent under $v^*$, this is possible only if $\psi^*$ is a constant. From \eqref{EL2.4A}
we then get $\norm{c}_\infty=\rho=\min_{u\in\Act(i)} c(i, u)$ which is possible if $c$ is a constant.
This is a contradiction. Hence we must have $\rho<\norm{c}_\infty$.
\end{proof}

Now we are ready to show that $\rho=\lamstr$. 
To this aim we perturb the cost function as follows: 
\begin{itemize}
\item Under \cref{EA2.2}(a): let $\alpha > 0$ be a small number such that $\|c\|_{\infty} + \alpha < \gamma$. We define
\begin{equation*}
\Tilde{c}_n(u,i)\,= \, c(u, i)\Ind_{\sD_n}(i) + (\|c\|_{\infty} + \alpha)\Ind_{\sD_n^c}(i)\quad\forall\,\, u\in \Act(i),\,\, i\in S\,.
\end{equation*} 
It is evident that $\|\Tilde{c}_{n}\|_{\infty} < \gamma$. 
\item[•] Under \cref{EA2.2}(b): We define
\begin{equation*}
\Tilde{c}_n(u,i)\,= \, c(u, i) + \frac{1}{n}h(i)\quad\forall\,\, u\in \Act(i),\,\, i\in S\,,
\end{equation*} 
where $h(i)= [\ell(i)-\max_{u\in\Act(i)}c(i, u)]_+$. Recall that $h$ is a norm-like function.
For large enough $n$ it is evident that $\ell(\cdot)-\max_{u\in\Act(\cdot)}\tilde{c}_n(\cdot, u)$ is also norm-like.
\end{itemize}
Thus, the conclusion of \cref{L2.4} hold if we replace $c$ by $\Tilde{c}_{n}$.
\begin{lemma}\label{L2.6}
Assume \cref{A1.1,EA2.2}. Then any minimizing selector of \cref{EL2.4A}, that is, any $v^*\in\Usm$ satisfying
\begin{equation}\label{EL2.6A}
\min_{u\in\Act(i)}\left[ e^{c(i,u)}\sum_{j\in S} \psi^*(j) P(j | i, u)\right] = \left[ e^{c(i,v^*(i))}\sum_{j\in S} \psi^*(j) P(j | i, v^*(i))\right]\quad\forall\,\, i\in S\,,
\end{equation}
is an optimal control and $\rho=\lamstr$. Moreover, $\psi^*$ is the unique solution of \eqref{EL2.4A} with $\psi^*(i_0)=1$.
\end{lemma}

\begin{proof}
Let $v^*\in\Usm$ be a minimizing selector as in \cref{EL2.6A}. From \cref{R2.2}, there exists 
an eigenpair $(\rho_{v^*,n}, \psi_{v^*,n})\in\RR_{+}\times\order(\Lyap)$, $\psi_{v^*,n} > 0$,
satisfying
\begin{equation}\label{EL2.6B}
e^{\rho_{v^*,n}} \psi_{v^*,n}(i) = \left[ e^{\Tilde{c}_n(i, v^*(i))}\sum_{j\in S} \psi_{v^*,n}(j)P(j | i, v^*(i))\right]\quad\forall\,\, i\in S\,,
\end{equation} 
and
\begin{equation}\label{EL2.6C}
0 \le \rho_{v^*,n} \le \limsup_{T\to\infty}\frac{1}{T}\log \Exp_{i}^{v^*}\left[ e^{\sum_{t=0}^{T-1}\Tilde{c}_n(X_t, v^{*}(X_t))}\right]=\sE_i(\tilde{c}_n, v^*)\quad \forall\; i\in S.
\end{equation} 
From the proof of \cref{L2.4}, there exist a finite set $\Tilde{\sB}$, dependent on $n$, containing $\cK$ such that $(\Tilde{c}_{n}(i, u) - \rho_{v^*,n}) \ge 0$ in $\Tilde{\sB}^c$ (under \cref{EA2.2}(a) we
may take $\Tilde{\sB} = \sD_n$, by \cref{L2.5}, and since under \cref{EA2.2}(b), $\tilde{c}_n$ is 
norm-like we can choose suitable finite set $\Tilde{\sB}$ satisfying the required condition). 
Rewrite \cref{EL2.6B} as
\begin{equation}\label{EL2.6D}
\psi_{v^*,n}(i) = \left[ e^{(\Tilde{c}_n(i, v^*(i)) - \rho_{v^*,n})}\sum_{j\in S} \psi_{v^*,n}(j) P(j | i, v^*(i))\right]\quad \forall\; i\in S.
\end{equation}
Then by the Markov property of $\textbf{X}$, it follows from \eqref{EL2.6D} that
\begin{equation*}
\psi_{v^*,n}(i) = \Exp_{i}^{v^*}\left[ e^{\sum_{t = 0}^{\uuptau(\Tilde{\sB})\wedge N - 1}
(\Tilde{c}_n(X_t, v^*(X_t)) - \rho_{v^*,n})}\psi_{v^*,n}(X_{(\uuptau(\Tilde{\sB})\wedge N)})\right]\,.
\end{equation*}
Letting $N\to\infty$ and using Fatou's lemma, for all $i\in \tilde{\sB}^c$, we deduce that
\begin{align*}
\psi_{v^*,n}(i) 
& \ge \Exp_{i}^{v^*}\left[ e^{\sum_{t = 0}^{\uuptau(\Tilde{\sB}) - 1}(\Tilde{c}_n(X_t, v^*(X_t)) - \rho_{v^*,n})}\psi_{v^*,n}(X_{\uuptau(\Tilde{\sB})})\right]
\\
& \ge \left(\min_{\Tilde{\sB}}\psi_{v^*,n}\right)\Exp_{i}^{v^*}\left[ e^{\sum_{t = 0}^{\uuptau(\Tilde{\sB}) - 1}(\Tilde{c}_n(X_t, v^*(X_t)) - \rho_{v^*,n})}\right]
\\
& \ge \left(\min_{\Tilde{\sB}}\psi_{v^*,n}\right)\,.
\end{align*}
Thus $\psi_{v^*, n}$ is bounded below by a positive constant.
Again using the Markov property of $\textbf{X}$ and applying Fatou's lemma in \cref{EL2.6D},
 we obtain that
\begin{align*}
\psi_{v^*,n}(i) & 
\ge \Exp_{i}^{v^*}\left[ e^{\sum_{t = 0}^{T - 1}(\Tilde{c}_n(X_t, v^*(X_t)) - \rho_{v^*,n})}
\psi_{v^*,n}(X_{T})\right]
\\
& \ge \left(\min_{\Tilde{\sB}}\psi_{v^*,n}\right)\Exp_{i}^{v^*}\left[ e^{\sum_{t = 0}^{T - 1}(\Tilde{c}_n(X_t, v^*(X_t)) - \rho_{v^*,n})}\right]\,.
\end{align*}
Taking logarithm on both sides, dividing by $T$ and letting $T\to \infty$, we get
\begin{align*}
\rho_{v^*,n} 
&\ge \limsup_{T\to\infty}\frac{1}{T}\log \Exp_{i}^{v^*}\left[ e^{\sum_{t=0}^{T-1}\Tilde{c}_n(X_t, v^{*}(X_t))}\right]\nonumber
\\
&\ge \limsup_{T\to\infty}\frac{1}{T}\log \Exp_{i}^{v^*}\left[ e^{\sum_{t=0}^{T-1}c(X_t, v^{*}(X_t))}\right]\,.
\end{align*}
Using \eqref{EL2.6C}, this of course, implies $\sE_i(c, v^*)\leq \sE_i(\tilde{c}_n, v^*)=\rho_{v^*,n}$
for all $n$. From the definition of $\tilde{c}_n$, it is evident that $\{\rho_{v^*,n}\}$ is
a decreasing sequence which is bounded from below. Since $\norm{\tilde{c}_n}_\infty<\gamma$ 
(for \cref{EA2.2}(a)), it is easily seen that the stochastic representation \eqref{ER2.2B} holds
for $\psi_{v^*, n}$ with the same choice of $\sB$, independent of $n$. In view of \cref{L2.4}, we can even have
$\psi_{v^*, n}\leq \Lyap$ and it touches $\Lyap$ inside $\sB$.
Thus, using a diagonalization argument, there exists a pair $(\tilde\rho, \psi_{v^*})\in\RR_{+}\times\order(\Lyap)$, $\psi_{v^*} > 0$ satisfying
\begin{equation}\label{EL2.6E}
e^{\tilde\rho} \psi_{v^*}(i) = \left[ e^{c(i, v^*(i))}\sum_{j\in S} \psi_{v^*}(j) P(j | i, v^*(i))\right]\quad\forall\,\, i\in S\,,
\end{equation}  
and $\lim_{n\to\infty}\rho_{v^*,n}=\tilde{\rho}\geq \sE_i(c, v^*)\geq \rho$ for all $i\in S$. To complete the first part of the
proof we only need to show that $\tilde{\rho}=\rho$. From \cref{L2.2} 
and dominated convergence theorem (on \eqref{ER2.2B} for each $n$),
we obtain that
\begin{equation}\label{EL2.6F}
\psi_{v^*}(i) = \Exp_i^{v^*} \left[e^{\sum_{t=0}^{\uuptau(\sB) - 1}(c(X_t, v^*(X_t)) - \tilde\rho)}\psi_{v^*}(X_{\uuptau(\sB)})\right]\quad\forall \,\, i\in \sB^c\,.
\end{equation}
Since $\tilde\rho\ge \rho$, using \cref{EL2.4C} we have
\begin{equation}\label{EL2.6G}
\psi^*(i) \ge \Exp_i^{v^*} \left[e^{\sum_{t=0}^{\uuptau(\sB) - 1}(c(X_t, v^*(X_t)) - \tilde\rho)}\psi^*(X_{\uuptau(\sB)})\right]\quad\forall \,\, i\in \sB^c\,.
\end{equation} 
Then \cref{EL2.6F} and \cref{EL2.6G} implies that
\begin{equation}\label{EL2.6H}
(\psi^* - \psi_{v^*})(i) \ge  \Exp_i^{v^*} \left[e^{\sum_{t=0}^{\uuptau(\sB) - 1}(c(X_t, v^*(X_t)) - \tilde\rho)}(\psi^* - \psi_{v^*})(X_{\uuptau(\sB)})\right]\quad\forall \,\, i\in \sB^c\,.
\end{equation} 
Rescale $\psi_{v^*}$, by multiplying with a suitable positive constant, so that $(\psi^* - \psi_{v^*}) \ge 0$ in
 $\sB$ and $(\psi^* - \psi_{v^*})(\hat{i}) = 0$ for some $\hat{i}\in\sB$. Thus from \cref{EL2.6H}, we deduce that $(\psi^* - \psi_{v^*}) \ge 0$ in $S$. From \cref{EL2.6A} and \cref{EL2.6E}, we get
\begin{equation*}
0 = (\psi^* - \psi_{v^*})(\hat{i}) \ge \Exp_{\hat{i}}^{v^*} \left[e^{(c(\hat{i}, v^*(\hat{i})) - \rho^{v^*})}(\psi^* - \psi^{v^*})(X_1)\right]\,.
\end{equation*}
This is similar to \eqref{EL2.4JJ} and thus a similar argument gives $\psi^*=\psi_{v^*}$ in $S$.
\cref{EL2.6A} and \cref{EL2.6E} then give us $\tilde\rho=\rho$. Hence $\rho=\lamstr=\sE_i(c, v^*)$
for all $i$. This completes the first part of the proof. 

Next we show that $\psi^*$ is unique upto a normalization. Let $V$ be a positive solution to 
\begin{equation}\label{EL2.6J}
e^{\lamstr}V(i) = \min_{u\in\Act(i)}\left[e^{c(i,u)}\sum_{j\in S} V(j) P(j|i,u)\right]\quad\forall\,\, i\in S\,.
\end{equation}
Choose a minimizing selector $v$ of \eqref{EL2.6J}, that is,
\begin{equation}\label{EL2.6K}
e^{\lamstr}V(i) = e^{c(i,v(i))}\sum_{j\in S} V(j) P(j|i,v(i)) \quad\forall\,\, i\in S\,.
\end{equation}
From the proof of \eqref{EL2.4N} we then get
\begin{equation}\label{EL2.6L}
\psi^*(i) \le \Exp_{i}^{v}\left[e^{\sum_{t=0}^{\uuptau(\sB) - 1}(c(X_t, v(X_t)) - \lamstr)}\psi^*(X_{\uuptau(\sB)})\right]\quad\forall\,\, i\in\sB^c\,,
\end{equation}
for some suitable finite set $\sB$. On the other hand using \eqref{EL2.6K}, we have (see \eqref{EL2.4O})
\begin{equation}\label{EL2.6M}
V(i) \ge \Exp_{i}^{v}\left[e^{\sum_{t=0}^{\uuptau(\sB) - 1}(c(X_t, v(X_t)) - \lamstr)}\psi^*(X_{\uuptau(\sB)})\right]\quad\forall\,\, i\in\sB^c\,.
\end{equation}
Combining \cref{EL2.6J,EL2.6K,EL2.6M} and using the arguments above (see \eqref{EL2.6H}) we can conclude
that $\psi^*=V$ (upto a multiplicative constant). Hence the proof.
\end{proof}

Now we complete the proof of \cref{T2.1}
\begin{proof}[Proof of \cref{T2.1}]
(i) follows from \cref{L2.4,L2.6}. Furthermore, \cref{L2.6} also gives us that any measurable
selector of \eqref{EL2.6A} is an optimal stationary Markov control. Thus to complete the 
proof of (ii) we only need to show that if for any $\hat{v}\in\Usm$ we have
$\sE_i(c, \hat{v})=\lamstr$ then $\hat{v}$ satisfies \eqref{EL2.6A}.

From \cref{R2.2} there exist $(\rho_{\hat{v}}, \psi_{\hat{v}})\in\RR_{+}\times\order(\Lyap)$, 
$\psi_{\hat{v}} > 0$, satisfying
\begin{equation}\label{ET2.1D}
e^{\rho_{\hat{v}}}\psi^{\hat{v}}(i) = e^{c(i,\hat{v}(i))}\Sigma_{j\in S} \psi^{\hat{v}}(j) P(j|i,\hat{v}(i))\quad\forall\,\, i\in S\,.
\end{equation} 
Moreover, for some finite set $\sB$ containing $\cK$ 
\begin{equation}\label{ET2.1E}
\psi_{\hat{v}}(i) = \Exp^v_i \left[e^{\Sigma_{t=0}^{\uuptau(\sB) - 1}(c(X_t, \hat{v}(X_t)) - \rho_{\hat{v}})}\psi_{\hat{v}}(X_{\uuptau(\sB)})\right]\quad\forall \,\, i\in \sB^c\,.
\end{equation}
Proof of \cref{L2.6} also gives us $\rho_{\hat{v}}=\lamstr$. Then using the arguments
of \cref{L2.6} (see \eqref{EL2.6L} and \eqref{EL2.6M}) it can be easily shown that
$\psi_{\hat{v}}=\kappa \psi^*$ for some positive $\kappa$. Hence the result follows
from \eqref{ET2.1D} and \eqref{EL2.4A}.
This completes the proof.
\end{proof}

%%%%%%%%%%%%%%%%%%%%%%%%%%%%%%%%%%%%%%%%%%%%%%%%%%%%%%%%%%%%%%%%%%%%%%%%%%%%%%%%%%%%%%%%%%%%%%%%%%%%%%%%
\subsection{Near-monotone cost} In this section we replace the \cref{EA2.2} with a near-monotone assumption stated below.

\begin{assumption}\label{A2.3}
Define $\lambda_{\rm m}=\inf_{i\in S}\inf_{v\in \Usm}\sE_i(c, v)$. We assume that
$$\inf_{v\in\Usm}\sE_i(c, v)<\infty\quad \forall\; i\in S,$$
and the cost function $c$ satisfies the {\it near-monotone} condition with respect to $\lambda_{\rm m}$,
that is,
\begin{equation}\label{EA2.3A}
\liminf_{n\to\infty}\, \inf_{k\geq n}\inf_{u\in\Act(k)} c(k, u)\, >\, \lambda_{\rm m}.
\end{equation}
\end{assumption}
We recall $\lamstr = \, \inf_{i\in S}\, \inf_{\zeta\in \Uadm}\sE_i(c, \zeta)$ from \cref{lamstr}.
Note that by \cref{A2.3}, $\lamstr\leq\lambda_{\rm m}<\infty$. Near-monotone
cost penalizes transient behaviour of the CMP.
 The following result gives an existence of an
optimal stationary Markov control. In the following we are not able to get a nonlinear Poisson
equation, instead we obtain an inequality.

\begin{theorem}\label{T2.3}
Grant \cref{A1.1,A2.3}. We also assume that {\bf X} is recurrent under any control $v\in\Usm$.
Then there exists a positive function $\psi^*$ satisfying
\begin{equation}\label{ET2.3A}
e^{\lamstr} \psi^*(i) \geq \min_{u\in\Act(i)}\left[e^{c(i,u)}\sum_{j\in S} \psi^*(j)P(j|i,u)\right]\quad\text{for}\,\, i\in S\,.
\end{equation}
Futhermore, we have $\lamstr=\lambda_{\rm m}$ and any measurable selector of \eqref{ET2.3A} is an
optimal stationary Markov control.
\end{theorem}

\begin{proof}
From \cref{L2.1}, let $(\rho_n, \psi_n)$ be the Dirichlet eigenpair in $\sD_n$ satisfying
\begin{equation}\label{ET2.3B}
e^{\rho_{n}}\psi_{n}(i) = \min_{u\in\Act(i)}\left[e^{c(i,u)}\sum_{j\in S} \psi_{n}(j) P(j|i,u)\right]\quad\forall\,\, i\in \sD_n\,.
\end{equation}
As mentioned in \cref{R2.1}, the conclusion of \cref{L2.3} holds under the hypothesis of \cref{T2.3}.
Since $\psi_n(i_0)>0$, we normalize $\psi_n$ to satisfy $\psi_n(i_0)=1$. Since $c\geq 0$, using
\eqref{ET2.3B} and \cref{L2.3}(i) we obtain
$$\psi_n(j)\leq \sup_{u\in \Act(i_0)} \frac{1}{P(j|i_0, u)}\exp{(\inf_{v\in\Usm}\sE_{i_0}(c, v))}\quad \text{for}\; j\neq i_0,$$
for all $n$ large. Thus, using a diagonalization argument, we can extract a subsequence, denoting it
by $(\rho_n, \psi_n)$, such that $\rho_n\to \rho$ and $\psi_n\to\psi^*$ componentwise, as $n\to\infty$.
Moreover, $\psi^*(i_0)=1$. Using the compactness of the measurable selectors of
\eqref{ET2.3B} and applying Fatou's lemma it is also easy to see that 
\begin{equation}\label{ET2.3C}
e^{\rho}\psi^*(i) \ge \min_{u\in\Act(i)}\left[e^{c(i,u)}\sum_{j\in S} \psi^*(j) P(j|i,u)\right]\quad\forall\,\, i\in S\,.
\end{equation}
From the arguments of \cref{L2.4} (see \eqref{EL2.4JJ}) it also evident that $\psi^*>0$ in $S$.
Therefore, given a state $i\in S$, we can find large $n$ so that $\psi_n(i)>0$ and thus, applying \eqref{EL2.1B}, we obtain $\rho\leq \lamstr$. Since $\rho\leq\lambda_{\rm m}$, applying the near-monotonicity
condition \eqref{EA2.3A}, we find a finite set $\sB$ such that 
\begin{equation}\label{ET2.3D}
\inf_{u\in \Act(i)} c(i, u)-\rho>0\quad \text{for}\; i\in \sB^c.
\end{equation}
Now consider a measurable selector $v$ of \eqref{ET2.3C} and by $\uuptau=\uuptau(\sB)$ we denote the first hitting time to $\sB$. From the proof of \cref{L2.2} we then see that
$$ \psi^*(i)\geq \Exp_i^v\left[e^{\sum_{t=0}^{\uuptau-1}(c(X_t, v(X_t))-\rho)} \psi^*(X_{\uuptau})\right]
\ge \min_{j\in \sB}\psi^*(j)
\quad i\in\sB^c.$$
Thus, $\inf_{i\in S}\psi^*(i)>0$. 
Using the Markov property of $\textbf{X}$,
 we obtain from \eqref{ET2.3C} that
\begin{align*}
\psi^*(i) & 
\ge \Exp_{i}^{v}\left[ e^{\sum_{t = 0}^{T - 1}(c(X_t, v^*(X_t)) - \rho)}
\psi^*(X_{T})\right]
\\
& \ge \left(\min_{\sB}\psi^*\right)\Exp_{i}^{v}\left[ e^{\sum_{t = 0}^{T - 1}(c(X_t, v(X_t)) - \rho)}\right]\,.
\end{align*}
Taking logarithm on both sides, dividing by $T$ and letting $T\to \infty$, we get
$\rho\ge \sE_i(c, v)\geq \lambda_{\rm m}$ for all $i\in S$. Thus $\rho=\lamstr=\lambda_{\rm m}$ and 
$v$ is an optimal stationary Markov control.
\end{proof}

\cref{T2.3} should be compared with \cite[Theorem~3.6]{BorMey02}. Though our condition on the controlled Markov chains is little stronger than those of \cite{BorMey02}, our method neither need the cost $c$ to be 
norm-like nor we assume the action set to be finite. 

\begin{remark}\label{R2.3}
\cref{A1.1}(b) can be replaced by other similar assumption. For instance,
 if the killed process communicates with every state in $\sD_n$ from
$i_0$ before leaving the domain $\sD_n$, for large $n$, then our method applies. More precisely,
we can replace \cref{A1.1}(b) with the following: for all large $n$
we have
$$\inf_{v\in\Usm}\Prob^v_{i_0}(\breve{\uptau}_{j}<\uptau_{\sD_n})>0\,\quad \text{for all}\; j\in \sD_n\setminus\{i_0\}\,,
 $$
where $\breve{\uptau}_{i_0}$ denotes the hitting time to $i_0$. In other words, for every $j\in \sD_n\setminus\{i_0\}$ and $v\in\Usm$, there exists distinct $i_1, i_2, \ldots, i_m \in \sD_n\setminus\{i_0\}$
satisfying
$$P(i_1|i_0, v(i_0))P(i_2|i_1, v(i_1))\cdots P(j|i_m, v(i_m))>0\,.$$
 The Birth-Death Markov chain is
a typical example of a CMP satisfying the above condition.
Note that under this new assumption 
we can show that $\psi_n(i_0)>0$ for large $n$.
\end{remark}

It is also possible to relax the recurrence hypothesis under stationary Markov control.
To this aim we introduce the following assumption.
\begin{assumption}\label{A2.4}
There exists a function $W:S\to [1, \infty)$ satisfying $W(i)\geq i$ for all large $i$ and 
\begin{equation}\label{EA2.4A}
\sup_{u\in\Act(i)}\sum_{j\in S} (W(j)-W(i)) P(j|i, u)\leq g(i)\quad \text{for}\; i\in S,
\end{equation}
for some function $g:S\to \RR$ satisfying $\lim_{i\to\infty} g(i)=0$. Furthermore, for some
$\eta>0$ we have
\begin{equation}\label{EA2.4B}
\min_{u\in\Act(i)}P(i-1|i, u)\,\geq\, \eta\quad \text{for all}\; i\geq 2,
\end{equation}
and $P(\cdot|1, u)$ supported in a finite set $C$, independent of $u$. In addition, also
assume that for $\sD_n\df\{1, \ldots, n\}$, $v\in\Usm$ and any $j\in\sD_n\setminus\{1\}$ there exists
distinct $i_1, i_2, \ldots, i_k\in \sD_n$ we have
\begin{equation}\label{EA2.4C}
P(i_1|1, v(1))P(i_2|i_1, v(i_1))\cdots P(j|i_k, v(i_k))>0\,.
\end{equation}
\end{assumption}

Note that \eqref{EA2.4A} does not guarantee that DTCMP {\bf X} is recurrent under every
stationary Markov control. To illustrate, let us consider the following classical example
of Birth-Death process.
\begin{example}\label{Eg2.2}
Suppose that $\Act$ is a compact metric space and $\lambda, \mu:S\times\Act\to [0, 1]$ be such that
$\lambda(i, u) + \mu(i, u)=1$ and $\mu(i, u)\geq\eta>0$ for all $i\geq 1$. Let $\{p_{ij}\}$ be a collection of nonnegative numbers satisfying
$$\sum_{j\geq 1} p_{ij}=1,\quad \sum_{j\geq 1} jp_{ij}<\infty\quad \text{for all}\; i\geq 1.$$
We define
\[
P(j|i, u)=\left\{\begin{array}{llll}
\lambda(i, u)p_{ik} & \text{if}\; j=i+k,\, i, k\geq 1,
\\
\mu(i, u) & \text{if}\; j=i-1, \, i\geq 1,
\\
0 & \text{if}\; j\leq i-2, \, i\geq 1,
\\
1 & \text{if}\; j=1,\, i=0.
\end{array}
\right.
\]
Then for $W(i)=i+1$ we have
$$\sup_{u\in\Act}\sum_{j\in S} (W(j)-W(i)) P(j|i, u)\leq \sup_{u\in\Act}
|\lambda(i, u)\sum_{k\geq 1}kp_{ik}-\mu(i, u)|\df g(i).$$
Thus if we assume, $g(i)\to 0$ as $i\to\infty$, we get \eqref{EA2.4A} and \eqref{EA2.4B}. Furthermore,
if we assume $\lambda(i, u)>0, p_{i,1}>0$, then we also have \eqref{EA2.4C}.

{\bf X} need not be recurrent under \eqref{EA2.4A}. For instance, if we take $p_{i1}=1$ for all $i$ and
$$\lambda(i, u)=\lambda(i)=\frac{(i+1)^2}{i^2+ (i+1)^2}, \quad 
\mu(i, u)=\mu(i)=\frac{i^2}{(i+1)^2}\lambda(i).$$
Then $\lambda(i)+\mu(i)=1$ and 
$$|\lambda(i)-\mu(i)|\leq \frac{2i+1}{i^2+(i+1)^2}\to 0\quad \text{as}\; i\to \infty.$$
But 
$$\sum_{n\geq 1}\, \prod_{i=1}^n\frac{\mu(i)}{\lambda(i)}=\sum_{n\geq 1} \frac{1}{(n+1)^2}\, <\infty.$$
This implies that {\bf X} is transient.
\end{example}

We establish the following result.

\begin{theorem}\label{T2.4}
Grant \cref{A2.3,A2.4}. Also assume that {\bf X} is irreducible under any stationary
Markov control.
Then there exists a positive function $\psi^*$ satisfying
\begin{equation}\label{ET2.4A}
e^{\lamstr} \psi^*(i) \geq \min_{u\in\Act(i)}\left[e^{c(i,u)}\sum_{j\in S} \psi^*(j)P(j|i,u)\right]\quad\text{for}\,\, i\in S\,.
\end{equation}
Futhermore, we have $\lamstr=\lambda_{\rm m}$ and any measurable selector of \eqref{ET2.4A} is an
optimal stationary Markov control.
\end{theorem}

\begin{proof}
Since {\bf X} may not be recurrent under a stationary Markov control, the proof of \cref{T2.3} does not
work. We have to modify the proof.
We begin with the  Dirichlet eigenpair $(\rho_n, \psi_n)$ in $\sD_n$ satisfying
\begin{equation}\label{ET2.4B}
e^{\rho_{n}}\psi_{n}(i) = \min_{u\in\Act(i)}\left[e^{c(i,u)}\sum_{j\in S} \psi_{n}(j) P(j|i,u)\right]\quad\forall\,\, i\in \sD_n\,.
\end{equation}
Due to \eqref{EA2.4C} we must have $\psi_n(1)>0$. This in turn, implies from \eqref{EA2.4B}  that
$\psi_n>0$ in $\sD_n$. As a consequence we have $\rho_n\leq \lamstr$,
by \cref{L2.1}, and $\rho_n> -\infty$ for all $n$.
Note that we can not apply \eqref{A1.1}(ii) anymore to find an upper
bound for $\psi_n$. Instead we use \eqref{EA2.4C}. Denote by $\uuptau_j$ the first hitting time to 
$j$. We claim that if $j\in \sD_n$, then
\begin{equation}\label{ET2.4C}
\inf_{v\in\Usm} \Prob_1^v (\uuptau_j\leq n\wedge\uptau_n)>\kappa(n, j)
\end{equation}
for some positive constant $\kappa(n, j)$. Suppose, on the contrary, that the claim is not true.
 Then there
exists $v_m\in\Usm$ such that $\Prob_1^{v_m} (\uuptau_j\leq n\wedge\uptau_n)\to 0$ as $m\to\infty$.
Using the compactness of $\Act(i)$, we can extract a subsequence of $v_m$, denoted by the original one,
so that $v_m(i)\to v(i)$ for all $i$ as $m\to\infty$. By \cref{A1.1}(i) we then see that
the law of ${\mathbf X_m}$ converges to {\bf X}, where ${\mathbf X_m}$
({\bf X}) is the DTCMP governed by $v_m$ ($v$, respectively).
Therefore, for every $k\leq n$,
\begin{align*}
&\Prob^v_{1}(X_i\in \sD_n\setminus\{1, j\}, X_k=j\quad \text{for all}\; 1\leq i\leq k-1)
\\
&\quad=\lim_{m\to\infty} \Prob^{v_m}_{1}(X_{m,i}\in \sD_n\setminus\{1, j\}, X_{m,k}=j\quad \text{for all}\; 1\leq i\leq k-1)
\\
&\quad \leq \lim_{m\to\infty} \Prob^{v_m}_{1}(\uuptau_j\leq n \wedge \uptau_n)=0.
\end{align*}
This clearly, contradicts \eqref{EA2.4C}. Hence we must have \eqref{ET2.4C}. From the monotonicity
of $\uptau_n$ it then follows that for $\sD_n\supset\sD_m\ni j$, we have
\begin{equation}\label{ET2.4D}
\inf_{v\in\Usm} \Prob_1^v (\uuptau_j\leq m\wedge\uptau_n)
\geq \inf_{v\in\Usm} \Prob_1^v (\uuptau_j\leq m\wedge\uptau_m)\geq \kappa(m, j).
\end{equation}
Now we normalize $\psi_n$ to satisfy $\psi_n(1)=1$. Let $v_n$ be a minimizing selector of \eqref{ET2.4B}.
Thus, using optional sampling theorem (see \eqref{EL2.1E}), it follows from \eqref{ET2.4B} that
\begin{align*}
1=\psi_n(1)&=\Exp_1^{v_n}\left[e^{\sum_{t=0}^{\uuptau_j\wedge m\wedge\uptau_n-1}(c(X_t, v_n(X_t))-\rho_n) }\psi_n(X_{\uuptau_j\wedge m\wedge\uptau_n})\right]\nonumber
\\
&\geq e^{-\lamstr m} \psi_n(j) \inf_{v\in\Usm} \Prob_1^v (\uuptau_j\leq m\wedge\uptau_n)\nonumber
\\
&\geq e^{-\lamstr m}\psi_n(j) \kappa(m ,j)
\end{align*}
using \eqref{ET2.4D}. Choosing $m=j+1$, this indeed gives us
\begin{equation}\label{ET2.4E}
\psi_n(j)\leq \frac{1}{\kappa(j+1, j)}e^{(j+1)\lamstr}\quad \text{for all} n > j\,.
\end{equation}
This gives the upper bound on $\{\psi_n\}$. Therefore, from the proof of \cref{L2.3} it can be easily seen that $\{\rho_n\}$ is bounded from below. (Otherwise, we must have a $\psi^*\gneq 0$ satisfying
$\psi^*(1)=1$ and
$$0\geq \sum_{j\in S} \psi^*(j) P(j|i, \hat{v}(i))\quad \text{for all}\; i\in S.$$
Then $\psi^*(i)=0$ on the support of $P(\cdot|1, \hat{v}(1))$. Repeated use of \eqref{EA2.4B} thus
gives us $\psi^*(1)=0$ which is not possible).
Thus we can find a subsequence of $(\rho_n, \psi_n)$
converging to $(\rho, \psi^*)$ and 
\begin{equation}\label{ET2.4F}
e^{\rho}\psi^*(i) \ge \min_{u\in\Act(i)}\left[e^{c(i,u)}\sum_{j\in S} \psi^*(j) P(j|i,u)\right]\quad\forall\,\, i\in S\,.
\end{equation}
Since {\bf X} is irreducible under any stationary Markov policy, it follows that $\psi^*>0$ (see \eqref{EL2.4JJ}). 

Next we show that $\rho=\lamstr=\sE_i(c, v)$ for all $i$, where $v$ is a minimizing selector of 
\eqref{EL2.4JJ}. Actually, the proof would follow from the arguments of \cref{T2.3} if we could show
that {\bf X} is recurrent under the policy $v$. Observe that for any $i\in S$, we have from
\eqref{ET2.4F} that
$$\psi^*(i+1) \geq e^{-\rho} \psi^*(i) P(i|i+1, v(i+1))\geq \eta e^{-\lamstr} \psi(i).$$
Thus
\begin{equation}\label{ET2.4G}
\psi^*(j)\geq (\eta e^{-\lamstr})^{(j-1)} \psi(1)= e^{-\tilde\kappa j}\quad j\geq 2,
\end{equation}
for some constant $\tilde\kappa$. Also, applying Dynkin's formula in \eqref{EA2.4A} we obtain
\begin{equation}\label{ET2.4H}
\Exp_i^v[W(X_t)]-W(i) \leq \Exp_i^v\left[\sum_{t=0}^{T-1} g(X_t)\right] \quad T\geq 0.
\end{equation}
In view of \eqref{ET2.4G} and \cref{A2.4} we have $\log\psi^*(i)\geq - \tilde{\kappa} W(i) - \widehat{\kappa}$ for some constant $\widehat{\kappa}$. 
Using \eqref{ET2.4F} and the Markov property of {\bf X} it follows that
$$\psi^*(1)\geq \Exp_1^v\left[e^{\sum_{t=0}^{T-1}(c(X_t,v(X_t))-\rho)}\psi^*(X_T)\right].$$
Taking logarithm on both sides and dividing by $T$, we get
\begin{align}\label{ET2.4I}
0 &\geq \frac{1}{T}\Exp_0^v\left[\sum_{t=0}^{T-1}(c(X_t,v(X_t))-\rho)\right]
-\frac{1}{T} (\tilde{\kappa}\Exp_1^v[W(X_T)] +\widehat{\kappa})\nonumber
\\
&\geq \frac{1}{T}\Exp_1^v\left[\sum_{t=0}^{T-1}(c(X_t,v(X_t))-\tilde{\kappa} g(X_t)-\rho)\right]
-\frac{1}{T}(\tilde{\kappa}W(1)+\widehat\kappa),
\end{align}
by \eqref{ET2.4H}. Now suppose that {\bf X} is not recurrent under $v$. Since {\bf X} is irreducible,
it must be transient and therefore, it can not have any invariant probability measure  
by \cite[Proposition~10.1.1]{MT93}. Thus, by \cite[Theorem~12.1.2]{MT93}, for any finite
set $B\subset S$ we have
\begin{equation}\label{ET2.4J}
\frac{1}{T}\Exp_1^v\left[\sum_{t=0}^{T-1}\Ind_B(X_t)\right]\to 0\quad \text{as}\; T\to\infty.
\end{equation}
Let $B_\circ$ be a finite set such that $\min_{u\in\Act(i)} (c(i, u)-\tilde{\kappa}g(i)-\rho)>\delta>0$, for some $\delta>0$, and $i\in B^c_\circ$.
This is possible due to \eqref{EA2.3A} and the fact $\lim_{i\to\infty} g(i)=0$. Applying \eqref{ET2.4J}
we obtain
\begin{align*}
&\liminf_{T\to\infty}\frac{1}{T}\Exp_1^v\left[\sum_{t=0}^{T-1}(c(X_t,v(X_t))-\tilde{\kappa} g(X_t)-\rho)\right]
\\
&\geq \liminf_{T\to\infty}\frac{1}{T}\Exp_1^v\left[\sum_{t=0}^{T-1}\Ind_{B^c_\circ}(X_t)(c(X_t,v(X_t))-\tilde{\kappa} g(X_t)-\rho)\right] 
- \max_{B_\circ}|g-\rho| 
\limsup_{T\to\infty}\frac{1}{T}\Exp_1^v\left[\sum_{t=0}^{T-1}\Ind_{B_\circ}(X_t)\right]
\\
&\geq \delta \liminf_{T\to\infty}\frac{1}{T}\Exp_1^v\left[\sum_{t=0}^{T-1}\Ind_{B^c_\circ}(X_t)\right]
\\
&= \delta \liminf_{T\to\infty}\frac{1}{T}\Exp_1^v\left[\sum_{t=0}^{T-1}\Ind_{S}(X_t)\right]
\\
&=\delta\,.
\end{align*}
But this leads to a contradiction to \eqref{ET2.4I} when we let $T\to\infty$ in \eqref{ET2.4I}.
 Therefore {\bf X}
must be recurrent under $v$. Now rest of the argument follows from \cref{T2.3}.
\end{proof}

%%%%%%%%%%%%%%%%%%%%%%%%%%%%%%%%%%%%%%%%%%%%%%%%%%%%%%%%%%%%%%%%%%%%%%%%%%%%%%%%%%%%%%%%%%%%%%%%%%%%%%

\section{Risk-sensitive control of continuous time CMP}\label{S-CT}
In this section we consider continuous time CMP $\mathbf{X}=\{X_t$\,,$t\ge 0\}$, on a countable state space $S$, controlled by the control process $\zeta_t$\,, $t\ge 0$\,, taking values in $\Act$.
As before, $\Act$ is the action space of the controller, which is assumed to be a Borel space with Borel $\sigma$ algebra $\fB(\Act)$. For each $i\in S$, let $\Act(i)$ be the space of all admissible actions of the controller when the system is at state $i$.
 Let $\sK \df \{(i,u) : i\in S, u\in \Act(i)\}$ be set of all feasible state action pair.
 As before, we denote by $c:\sK \to \RR_{+}$ the running cost function. The transition rates $q(j|i, u)$, $u\in \Act(i)$\,, $i,j\in S$, satisfy the condition $q(j|i,u)\ge 0$ for all $u\in\Act(i), i, j\in S$ and $j\neq i$. In addition,
we also impose that 
\begin{assumption}\label{A3.1}
\begin{itemize}
\item[(a)] For each $i\in S$, the admissible action space $\Act(i)$ is a nonempty compact subset of $\Act$\,.
\item[(b)] The model is conservative:
\begin{equation*}
\sum_{j\in S} q(j|i,u) = 0\quad\forall\,\, u\in\Act(i),\,\, i\in S\,.
\end{equation*}
\item[(c)] The model is stable: 
\begin{equation*}
q(i) \,\df\, \sup_{u\in\Act(i)} (- q(i|i,u))\,=\sup_{u\in\Act(i)}\sum_{j\neq i} q(j|i, u)\, <\,  \infty\quad\forall \,\, i\in S\,.
\end{equation*}
For each $i,j\in S$, $q(j|i,u)$ is a measurable map on $\Act(i)$.
\end{itemize}
\end{assumption}

Following \cite{K85} (see also \cite{GL19,GZ19,PZ20}) we briefly describe the evolution of the continuous time CMP (CTCMP). Let $S_{\infty}\,\df\, S\cup \{i_{\infty}\}$ for an isolated point $i_{\infty}\notin S$. Define the canonical sample space $\Omega \,\df\, (S\times (0, \infty))^{\infty}\cup \{(i_0, \theta_1, i_1,\dots,\theta_m,i_m,\infty,i_{\infty},\infty,i_{\infty},\dots)\mid \theta_k \neq \infty, i_k \neq i_{\infty}\quad\text{for all}\quad 0\le k \le m,\,\, m\ge 1\}\,,$ with Borel $\sigma$-algebra $\fB(\Omega)$\,. For each sample point $\omega = (i_0, \theta_1, i_1,\ldots,\theta_m,i_m,\ldots)\in\Omega$, we set $T_0(\omega) = 0$, $T_k(\omega) = \theta_1 + \theta_2 +\dots +\theta_k$, and define $T_{\infty}(\omega) = \lim_{k\to\infty}T_k(\omega)$\,. Now we define a controlled process $\{X_t\}_{t\ge 0}$ on $(\Omega, \fB(\Omega))$ by
\begin{equation}\label{ES1}
X_t = \sum_{k\ge 0} \Ind_{\{T_{k} \le t < T_{k+1}\}}i_k + \Ind_{\{t\ge T_{\infty}\}}i_{\infty}\quad \text{for}\,\, t\ge 0\,.
\end{equation} From \cref{ES1}, it is clear that for any $m\ge 1$ and $\omega \in\Omega$, $T_m(\omega)$ denotes the $m$-th jump moment of the process $X_t$, $i_m$ is the state of the controlled process on $[T_m, T_{m+1})$ and $\theta_m = T_m - T_{m-1}$ denotes the waiting time between jumps (or, sojourn time) at state $i_{m-1}$\,. Also, we add an isolated point $u_{\infty}\notin \Act$ to $\Act$ and let $\Act_{\infty} = \Act\cup \{u_{\infty}\}$ and $\Act(i_{\infty}) = \{u_{\infty}\}$. We do not want to consider our process after the time $T_{\infty}$. Thus we assume that $i_{\infty}$ is an absorbing state, that is, $q(j|i_{\infty}, u_{\infty}) = 0$ for all $j\in S$. Also, assume that $c(i_{\infty}, u) = 0$ for all $u\in\Act_{\infty}$\,. Consider a filtration $\{\mathfrak{F}_{t}\}_{t\ge 0}$ where $\mathfrak{F}_{t} \,\df\, \sigma((T_{m} \le s, X_{T_m}\in A) : 0\le s \le t,\,\, m \ge 0, A\subset S)$, and let $\Tilde{\mathfrak{F}} \,\df\, \sigma(\mathcal{A}\times \{0\}, \mathcal{B}\times (s,\infty) : \mathcal{A}\in \mathfrak{F}_0, \mathcal{B}\in\mathfrak{F}_{s-})$ be the $\sigma$-algebra of predictable sets in $\Omega\times(0, \infty)$ with respect to $\mathfrak{F}_{t}$, where $\mathfrak{F}_{s-} \,\df\, \vee_{t < s} \mathfrak{F}_{t}$\,. 

An admissible policy $\zeta = \{\zeta_t\}_{t\ge 0}$ is a measurable map from $(\Omega\times (0,\infty), \Tilde{\mathfrak{F}})$ to $(\Act_{\infty}, \fB(\Act_{\infty})$ satisfying
$\zeta_t(\omega)\in\Act(X_{t-}(\omega))$ for all $\omega\in\Omega$ and $t\ge 0$\,. Let $\Uadm$ be the space of all admissible policies. An admissible policy $\zeta$ is said to be a Markov policy if $\zeta_t(w) = \zeta_{t}(X_{t-}(\omega))$ for all $\omega\in\Omega$ and $t\ge 0$\,. The space of all Markov policies is denoted by $\Um$\,. If the Markov policy $\zeta$ does not have any explicit time dependency then it is called a stationary Markov policy and $\Usm$ denotes the space of all stationary Markov strategies\,. For each $i\in S$ and $\zeta\in \Uadm$, it is well known that (cf.
\cite{K85,GL19,GZ19,BS78b}) there exist unique probability measure $\mathcal{P}_{i}^{\zeta}$ on $(\Omega, \fB(\Omega))$ such that $\mathcal{P}_{i}^{\zeta}(X_0 = i) = 1$\,. Let $\Exp_i^{\zeta}$ be the corresponding expectation operator. Also, from \cite[pp.13-15]{GHL09}, we know that $\{X_t\}_{t\ge 0}$ is a Markov process under any $\zeta\in\Um$ (in fact, strong Markov). 

Under some policies the process $\{X_t\}_{t\ge 0}$ may be explosive, in order to avoid explosion of the CTCMP, we impose the following (see \cite{GL19,GZ19},\cite[Assumption~2.2]{GHL09}).
\begin{assumption}\label{A3.1Exp}
There exist a function $\Tilde{\Lyap}: S \to [1,\infty)$ and constants $C_0\neq 0, C_1>0$ and $C_2\ge 0$ such that
\begin{itemize}
\item[(a)] $\sum_{j\in S} \Tilde{\Lyap}(j)q(j|i,u) \le C_0\Tilde{\Lyap}(i) + C_{2}$ for all $(i, u)\in \sK$\,;
\item[(b)] $q(i) \le C_1 \Tilde{\Lyap}(i)$ for all $i\in S$\,.
\end{itemize} 
\end{assumption}
For the rest of this section we are going to assume that \cref{A3.1Exp} holds. Note that
\cref{A3.1Exp} holds if $\sup_{i\in S} q(i)<\infty$. In this case we can choose $\tilde{\Lyap}$ to
be a suitable constant.
 From \cite[Theorem~3.1]{GP11} (see also, \cite[Proposition~2.2]{GL19}) it also
follows that, under \cref{A3.1Exp}, $\Prob^\zeta_i (T_\infty=\infty)=1$ for all $i\in S$ and $\zeta\in \Uadm$.
 
We also assume the following for our CTCMP (compare with \cref{A1.1})
\begin{assumption}\label{A3.3}
\begin{itemize}
\item[(a)] For each $i\in S$, the map $u\mapsto c(i,u)$ is continuous on $\Act(i)$\,.
\item[(b)] For each $i\in S$ and bounded measurable function $f:S\to\RR$, the map $u\mapsto \Sigma_{j\in S}f(j)q(j|i,u)$ is continuous on $\Act(i)$\,.
\item[(c)] There exists $i_0\in S$ such that $q(j|i_0,u) > 0$ for all $j\neq i_0$ and $u\in\Act(i_0)$.
\end{itemize}
\end{assumption}

For each  admissible control $\zeta$ the ergodic risk-sensitive cost is given by
\begin{equation}\label{EErgoCcost}
\sE_i(c, \zeta) \,\df\, \limsup_{T\to\infty} \, \frac{1}{T}\,
\log \Exp_i^{\zeta} \left[e^{\int_{0}^{T} c(X_t, \zeta_t)\D t}\right],
\end{equation} where $\textbf{X}$ is the CTCMP corresponding to $\zeta$ with initial state $i$.
As before, our aim is to minimize \cref{EErgoCcost} over all admissible policies in $\Uadm$. A policy $\zeta^{*}\in \Uadm$ is said to be optimal if for all $i\in S$ 
\begin{equation}\label{00A2}
\sE_i(c,\zeta^{*}) \, = \, \inf_{i\in S}\inf_{\zeta\in \Uadm}\sE_i(c, \zeta)=\lamstr\quad \text{for all}\; i\,.
\end{equation}
We also introduce the following Lyapunov condition. Recall that a stationary Markov 
process {\bf X} with rate matrix $Q=[q(j|i)]$ is irreducible if for any $i, j \in S, i\neq j,$
there exists distinct $i_1, i_2, \ldots, i_k\in S$ satisfying $q(i_1|i)\cdots q(j|i_k)>0$ (cf.
\cite[p.~107]{GHL09}).

%%%%%%%%%%%%%%%%%%%%%%%%%%%%%%%%%%%%%%%%%%%%%%%%%%%%%%%%%%%%%%%%%%%%%%%%%%%%%%%%
\begin{assumption}\label{A3.4}
We assume that the CTCMP \textbf{X} is irreducible under every stationary Markov control in
$\Usm$. In (a) and (b) below the function $\Lyap$ on $S$ takes values in $[1, \infty)$ and
$\widehat{C}$ is a positive constant. We assume that one of the following hold.
\begin{itemize}
\item[(a)] For some positive constant $\gamma$ and a finite set $\cK$ it holds that 
\begin{equation}\label{Lyap1}
\sup_{u\in\Act(i)} \sum_{j\in S} \Lyap(j) q(j|i,u) \le  \widehat{C}\Ind_{\cK}(i) - \gamma\Lyap(i) \quad \forall \quad i\in S\,. 
\end{equation}
Also, assume that $\norm{c}_\infty\df\sup_{i\in S}\sup_{u\in\Act(i)}c(i,u) < \gamma$.
\item[(b)] For a finite set $\cK$ and a norm-like function $\ell:S\to \RR_+$ it
holds that
\begin{equation}\label{Lyap2}
\sup_{u\in\Act(i)} \sum_{j\in S} \Lyap(j) q(j|i,u) \le \widehat{C} \Ind_{\cK}(i) - \ell(i)\Lyap(i) \quad \forall \quad i\in S\,. 
\end{equation}
Moreover, the function $\ell(\cdot)-\max_{u\in\Act(\cdot)} c(\cdot, u)$ is norm-like.
\end{itemize}
\end{assumption}
Applying \cite[Proposition~6.3.3]{And91}, we see that the CTCMP is strongly ergodic under every
stationary Markov control, provided \cref{A3.4} holds. Moreover, letting $\uuptau_j$ to be the
first hitting time to $j$, it also follows that $\Exp^\zeta_i[\uuptau_j]<\infty$ for all $i\neq j$
and $\zeta\in\Usm$. 

Before we proceed further, let us present a modified example from \cite[Example~1.3]{GHL09}.
\begin{example}\label{E3.1}
For $i\geq 2$, we suppose that
\[
q(j|i, u) =\left\{
\begin{array}{llll}
\lambda i + u & \text{for}\; j=i+1,
\\
\mu i + u & \text{for}\; j=i-1,
\\ 
- (\lambda i + \mu i + 2u)& \text{for}\; j=i,
\\
0 & \text{otherwise}.
\end{array}
\right.
\]
Let the control parameter $u$ take values in some bounded set.
Also, assume that $q(j|1, u)=q(j|1)$ is positive for every $j\geq 2$ and decays exponentially fast
with $j$. Therefore, \cref{A3.3}(c) holds. Suppose that $\mu>\lambda>0$ and define
$\Lyap(i)=e^{\theta i}$ for some $\theta>0$ to be chosen later. Then note that
\begin{align}\label{EE3.1A}
\sum_{j\in S} \Lyap(j) q(j|i,u) &= e^{\theta i}\left((\lambda i + u) e^\theta 
+ (\mu i + u) e^{-\theta} -(\lambda i + \mu i + 2u)\right)\nonumber
\\
&=i\,\Lyap(i) \left( \lambda (e^\theta-1) + \mu (e^{-\theta}-1) + \frac{u}{i}(e^{\theta}+e^{-\theta}-2)\right).
\end{align}
Since for every small $\theta>0$ we have
$$\mu(e^{-\theta}-1)+ \lambda (e^\theta-1)<0\,
\Leftrightarrow (e^{\theta}-1)(\lambda-\mu e^{-\theta})<0
\Leftrightarrow \lambda<\mu e^{-\theta},$$
letting $\ell(i)= \alpha i$, for $2\alpha= -\mu(e^{-\theta}-1)- \lambda (e^\theta-1)>0$ , we get from \eqref{EE3.1A} that 
$$\sup_{u\in\Act(i)}\sum_{j\in S} \Lyap(j) q(j|i,u) \le -\ell(i)\Lyap(i), $$
for $i\in \cK^c$ where $\cK$ is some finite set satisfying 
$\alpha>\frac{u}{i}(e^\theta+e^{-\theta}-2)$ for all $i\in \cK^c$ and
all control parameter $u$.
Now we let $\theta$ small enough so that
$\sum_{j\in S} \Lyap(j)q(j|1)<\infty$. Hence \cref{A3.4} holds.
\end{example}
%%%%%%%%%%%%%%%%%%%%%%%%%%%%%%%%%%%%%%%%%%%%%%%%%%%%%%%%%%%%%%%%%%%%%%%%%%%%%%%%%%%%%%%%%%%%%%%%%%%%%%%
Let us now state our first main result of this section (compare it with \cref{T2.1})
\begin{theorem}\label{T3.1}
Grant \cref{A3.1,A3.1Exp,A3.3,A3.4}. Then the following hold.
\begin{itemize}
\item[(i)] There exists a unique positive function $\psi^*$, $\psi^*(i_0)=1$, satisfying
\begin{equation}\label{ET3.1A}
\lamstr\psi^*(i) = \min_{u\in\Act(i)}\left[\sum_{j\in S} \psi^*(j)q(j|i,u) + c(i, u)\psi^*(i)\right]\quad\text{for}\,\, i\in S\,.
\end{equation}
\item[(ii)] A stationary Markov control $v\in\Usm$ is optimal if and only if it satisfies
\begin{equation}\label{ET3.1B}
\min_{u\in\Act(i)}\left[\sum_{j\in S} \psi^*(j)q(j|i,u) + c(i, u)\psi^*(i)\right]
= \left[\sum_{j\in S} \psi^*(j)q(j|i,v(i))+ c(i, v(i))\psi^*(i)\right]
\end{equation}
for all $i\in S$.
\end{itemize}
\end{theorem}
%%%%%%%%%%%%%%%%%%%%%%%%%%%%%%%%%%%%%%%%%%%%%%%%%%%%%%%%%%%%%%%%%%%%%%%%%%%%%%%%%%%%%%%%%%%%%%%%%
The rest of this section is dedicated to the proof of \cref{T3.1}. The main strategy of the proof is
same as the proof of \cref{T2.1}.
We begin with our next result which is a counterpart of \cref{P1.1} for CTCMP.

\begin{proposition}\label{P3.1}
Grant \cref{A3.1} and \cref{A3.3}(a)-(b). Suppose $c < -\delta$ in $\sD$ for some positive 
constant $\delta$ and a finite set $\sD$. Then for any $f\in\cB_{\sD}$ there exist unique $\phi \in \cB_{\sD}$ satisfying
\begin{equation}\label{EP3.1A}
\min_{u\in\Act(i)}\left[\sum_{j\in S}\phi(j)q(j|i,u) + c(i,u)\phi(i)\right] \,=\, - f(i), \quad \forall\,\, i\in \sD\,, \quad\text{and}\;\; \phi(i) = 0\quad \forall\,\, i\in \sD^c\,.
\end{equation}
Let $\uptau = \uptau(\sD)\,\df\,\inf\{t>0\;\colon\; X_t\notin \sD\}$. Then 
the unique solution satisfies 
\begin{equation}\label{EP3.1B}
\phi(i) \,= \inf_{\zeta\in\Um}\Exp_i^{\zeta}\left[\int_{0}^{\uptau} e^{\int_{0}^{t}c(X_s, \zeta_s)\D s} f(X_t)\D t\right]\quad\forall \,\, i\in S\,.
\end{equation} 
\end{proposition}

\begin{proof}
Given a tuple $(y_i)_{i\in\sD}$ and a fixed $i\in \sD$, let us consider the map
$$\RR\ni x\mapsto G(x)=\min_{u\in \Act(i)}\left[\sum_{i\neq j\in \sD}y_jq(j|i,u) + (q(i|i,u)+c(i,u))x\right].$$
We note that $G$ is strictly decreasing. For, $x_1>x_2$ we have
$$G(x_2)-G(x_1) \geq \min_{u\in \Act(i)} \left[(q(i|i,u)+c(i,u))(x_2-x_1)\right]
\geq \delta (x_1-x_2)>0\,,$$
since $q(i|i, u)(x_2-x_1)\geq 0$ for all $u\in\Act(i)$ and $i\in S$.
Furthermore, $\lim_{x\to \pm \infty}G(x)=\mp \infty$. Therefore, for every 
$y\in\RR$ there exists a unique $x$ satisfying $G(x)=y$. Using $G$ we can now define a map $\cT_1:\cB_\sD\to\cB_\sD$ that satisfies
\begin{equation}\label{EP3.1C}
\min_{u\in\Act(i)}\left[\sum_{i\neq j\in \sD}\phi(j)q(j|i,u) + 
(q(i|i,u)+ c(i,u))(\cT_1\phi(i))\right] \,=\, - f(i)\quad i\in \sD.
\end{equation}
We now show that $\cT_1$ is a contraction. Recall the norm $\norm{\cdot}_\sD$ from
\cref{P1.1}. Let $\psi_{m} = \cT_1\phi_{m}$ for $m=1,2$. For each $i\in\sD$ we then have
from \eqref{EP3.1C} that
\begin{align*}
0&\geq \min_{u\in\Act(i)}\left[\sum_{i\neq j\in \sD}\phi_1(j)q(j|i,u) + 
(q(i|i,u)+ c(i,u))\psi_1(i)\right]
\\
&\, \quad -\min_{u\in\Act(i)}\left[\sum_{i\neq j\in \sD}\phi_2(j)q(j|i,u) + 
(q(i|i,u)+ c(i,u))\psi_2(i)\right]
\\
&\geq \min_{u\in\Act(i)}\left[\sum_{i\neq j\in \sD}(\phi_1(j)-\phi_2(j))q(j|i,u) + 
(q(i|i,u)+ c(i,u))(\psi_1(i)-\psi_2(i))\right].
\end{align*}
Let $\tilde{u}\in\Act(i)$ be point where the minimum on RHS is attained.
Then we get from above
$$(q(i|i,\tilde{u})+ c(i,\tilde{u}))(\psi_1(i)-\psi_2(i)) + q(i|i, \tilde{u}) \norm{\phi_1-\phi_2}_\sD\leq 0,$$
which in turn, gives
$$(\psi_2(i)-\psi_1(i))\leq \sup_{u\in\Act(i)}\frac{-q(i|i, u)}{-q(i|i,u)-c(i, u)}
\norm{\phi_1-\phi_2}_\sD\leq \vartheta \norm{\phi_1-\phi_2}_\sD$$
for some $\vartheta<1$. Interchanging $\psi_1$ and $\psi_2$ in the above calculation and using the arbitrariness of $i$ we have
$$\norm{\cT_1\phi_1-\cT_1\phi_2}_\sD\leq \vartheta\norm{\phi_1-\phi_2}_\sD.$$
Therefore, $\cT_1$ is a contraction and for Banach fixed point theorem, we get a 
unique solution to \eqref{EP3.1A}. \eqref{EP3.1B} follows from Dynkin's formula.

\end{proof}
%%%%%%%%%%%%%%%%%%%%%%%%%%%%%%%%%%%%%%%%%%%%%%%%%%%%%%%%%%%%%%%%%%%%%%%%%%%%%%%%%%%%%%%%%%%%%%%%%%%%%%%

As before, applying \cref{T-KR,P3.1}, we obtain the existence of an eigenpair.
 
\begin{lemma}\label{L3.1}
Grant \cref{A3.1}, \cref{A3.1Exp} and \cref{A3.3}(a)-(b). Then there exists $(\rho_{\sD}, \psi_{\sD})\in\RR\times \cB^+_{\sD}$,
$\psi_\sD\gneq 0$, satisfying
\begin{equation}\label{EL3.1A}
\rho_{\sD}\psi_{\sD}(i) = \min_{u\in\Act(i)}\left[\sum_{j\in S} \psi_{\sD}(j)q(j|i,u) + c(i,u)\psi_{\sD}(i)\right].
\end{equation}
Moreover, we have
\begin{equation}\label{EL3.1B}
\rho_{\sD} \le \inf_{\zeta\in\Uadm}\limsup_{T\to\infty}\frac{1}{T}\log \Exp_i^{\zeta} \left[e^{\int_{0}^{T} c(X_t, \zeta_t)\D t}\right]\,,
\end{equation} for all $i\in S$ such that $\psi_{\sD}(i) >0$\,.
\end{lemma}

\begin{proof} 
It is evident from \eqref{EL3.1A} that we may assume
$c < -\delta$ in $\sD$ for some positive constant $\delta$. Otherwise, subtract $-\sup_{\sD}c-\delta$
from both sides of \eqref{EL3.1A}.
Let $\cT :\cB_{\sD}\to\cB_{\sD}$ be an operator defined as
\begin{equation}\label{EL3.1C}
\cT(f)(i) \,\df\, \phi(i) = \inf_{\zeta\in\Um}\Exp_i^{\zeta}\left[\int_{0}^{\uptau} e^{\int_{0}^{t}c(X_s, \zeta_s)\D s} f(X_t)\D t\right]\quad\forall \,\, i\in S\,.
\end{equation}
Then $\phi$ is the solution to \cref{EP3.1A}. It is fairly straightforward to show that $\cT$ is
completely continuous, order-preserving and $1$-homogeneous. Now choose $f\in\sB_\sD$ such that
$f(i)=1$ for some $i\in \sD$ and zero elsewhere.
Then, from \eqref{EL3.1C}, it follows that
\begin{align*}
\phi(i)\geq \inf_{\zeta\in\Um} &\Exp_i^\zeta\left[\int_0^{T_1}
e^{\int_{0}^{t}c(X_s, \zeta_s)\D s} f(X_t) \D{t}\right]
\\
&\geq f(i) \Exp_i^\zeta\left[\int_0^{T_1}
e^{-t\norm{c}_\sD}  \D{t}\right]
\\
&\geq \frac{f(i)}{\norm{c}_\sD} \inf_{\zeta\in\Um} \Exp_i^\zeta\left[1-
e^{-\norm{c}_\sD T_1}\right],
\end{align*}
where $T_1$ denotes the first jump time. It is well-known (cf. \cite{PZ20}) that
\begin{equation}\label{00A1}
\Prob_i^\zeta(T_1>t)=e^{\int_0^t q(i|i,\zeta_s(i))\,\D{s}}.
\end{equation}
Therefore,
\begin{align*}
\Exp_i^\zeta\left[1-e^{-\norm{c}_\sD T_1}\right]
&= 1 - \Exp_i^{\zeta}[e^{-\norm{c}_\sD T_1}]
\\
&=1-\int_0^\infty \norm{c}_{\sD} e^{-\norm{c}_\sD s} \Prob_i^\zeta(T_1\leq t)\, \D{t}
\\
&=\norm{c}_{\sD} \int_0^\infty e^{-\norm{c}_\sD s} \Prob_i^\zeta(T_1> t)
\,\D{t}
\\
&= \norm{c}_{\sD} \int_0^\infty e^{-\norm{c}_\sD s} 
e^{\int_0^t q(i|i,\zeta_s(i))\,\D{s}} \,\D{t}
\\
&\geq \norm{c}_{\sD} \int_0^\infty e^{-\norm{c}_\sD s} 
e^{-t q(i)} \,\D{t}=\frac{\norm{c}_{\sD}}{\norm{c}_\sD+q(i)},
\end{align*}
where in the forth line we use \eqref{00A1}. Hence
$$\phi(i)\geq \frac{f(i)}{\norm{c}_\sD+q(i)}.$$
Thus for some $M>0$ we have $M \cT(f) \succeq  f$.

By \cref{T-KR} there exist a nontrivial $\psi_\sD\in\sB^+_\sD, \psi_\sD\neq 0,$ and 
$\lambda_\sD > 0$ such that $\cT(\psi_\sD) = \lambda_\sD \psi_\sD$.
Applying \cref{P3.1} we then obtain
\begin{equation}\label{EL3.2D}
\min_{u\in\Act(i)}\left[\sum_{j\in S}\psi_\sD(j)q(j|i,u) + \psi_{\sD}(i)c(i,u)\right] = \rho_\sD
\psi_\sD(i)\quad \forall\;\; i\in\sD\,,
\end{equation} 
where $\rho_\sD=-[\lambda_\sD]^{-1}$. This gives us \eqref{EL3.1A}.

Next we show \eqref{EL3.1B}. Consider $i\in\sD$ satisfying $\psi_\sD(i)>0$. Choose an
admissible control $\zeta\in\Uadm$. We plan to apply Dynkin's formula upto 
the stopping time $t\wedge\uptau$ where $\uptau=\uptau(\sD)$ is the first exit time from $\sD$.
To apply the results from \cite{GZ19} (see also, \cite[Theorem~3.1]{QW16}), we define 
$\tilde{q}(j|i, u)=q(j|i, u)$ for $j, i\in \sD$, $\tilde{q}(\Delta|i, u)=\sum_{j\in\sD^c} q(j|i, u)$,
 and $q(j|\Delta, u)=0$ for all $j\in S, u\in\Act$.
Here $\Delta$ is an absorbing state. Also, define $\Act(\Delta)=u_\infty$. Recall the definition 
\eqref{ES1} and given a history dependent control $\zeta\in\Uadm$ we can redefine
another admissible control $\tilde\zeta$ to satisfy $\tilde{\zeta}_t\in\Act(\Delta)$ if $X_{t-}\in \sD^c$.
Let $\tilde{\mathbf X}$ be a process, corresponding to the control $\tilde\zeta$, taking
values in $\sD\cup\{\Delta\}$. Note that the law of $\{(X_t, \zeta_t)\;:\; t<\uptau\}$ is same as
$\{(\tilde{X}_t, \tilde\zeta_t)\; :\; t<\uptau_*\}$, where $\uptau_*$ denotes the first hitting time to $\Delta$ by 
$\tilde{\mathbf X}$. Now we can apply Dynkin's formula \cite[Lemma~3.2]{GZ19} to $\tilde{\mathbf X}$.
We apply it on $\psi_\sD$. Set $\psi_D(\Delta)=0$. Then 
\begin{align*}
\psi_{\sD}(i) &= \tilde\Exp_{i}^{\tilde\zeta}\left[ e^{\int_{0}^{T}(c(\tilde{X}_s,\tilde{\zeta}_{s}) - \rho_{\sD})\D s}\psi_{\sD}(\tilde{X}_{T}) \Ind_{\{T < \uptau_*\}}\right]
\\
&\quad - \tilde\Exp_{i}^{\tilde\zeta}\left[\int_0^{T\wedge\uptau_*} e^{\int_{0}^{t}(c(\tilde{X}_s,\tilde{\zeta}_{s}) - \rho_{\sD})\D{s}} \left(\sum_{j\in \sD} \psi_\sD(j)q(j|\tilde{X}_t,\tilde{\zeta}_t) + (c(\tilde{X}_t,\tilde{\zeta}_t)-\rho_\sD)\psi_\sD(\tilde{X}_t)\right)\D{t}\right]
\\
&= \Exp_{i}^{\zeta}\left[ e^{\int_{0}^{T}(c({X}_s,{\zeta}_{s}) - \rho_{\sD})\D s}\psi_{\sD}(X_{T})
 \Ind_{\{T < \uptau\}}\right]
\\
&\quad - \Exp_{i}^{\zeta}\left[\int_0^{T\wedge\uptau} e^{\int_{0}^{t}(c({X}_s,{\zeta}_{s}) - \rho_{\sD})\D{s}} \left(\sum_{j\in \sD} \psi_\sD(j)q(j|{X}_t,{\zeta}_t) + (c({X}_t, {\zeta}_t)-\rho_\sD)
\psi_\sD({X}_t)\right)\D{t}\right]
\\
&\le \Exp_{i}^{\zeta}\left[ e^{\int_{0}^{T}(c({X}_s,{\zeta}_{s}) - \rho_{\sD})\D s}\psi_{\sD}(X_{T})
 \Ind_{\{T < \uptau\}}\right]
\\
&\le (\sup_{\sD} \psi_{\sD}) \Exp_{i}^{\zeta}\left[ e^{\int_{0}^{T}(c(X_s,\zeta_{s}) - \rho_{\sD})\D s}\right],
\end{align*} 
where in the first inequality we use \eqref{EL3.2D}.
Now taking logarithm on both sides, dividing by $T$ and letting $T\to \infty$, we obtain
\begin{equation*}
\rho_{\sD} \le \limsup_{T\to\infty}\frac{1}{T}\log \Exp_{i}^{\zeta}\left[ e^{\int_{0}^{T}c(X_t,\zeta_{t})}\right]\,.
\end{equation*}
Since $\zeta\in\Uadm$ is arbitrary, we obtain \cref{EL3.1B}. This completes the proof.
\end{proof}

%%%%%%%%%%%%%%%%%%%%%%%%%%%%%%%%%%%%%%%%%%%%%%%%%%%%%%%%%%%%%%%%%%%%%%%%%%%%%%%%%%%%%%%%%%%%%%%%%%%%%%

We begin with the following hitting time estimate which follows from \cref{A3.4} (compare
with \cref{L2.2}).
\begin{lemma}\label{L3.2}
Suppose that \cref{A3.4} holds. Let $\sB$ be a finite set containing $\cK$. Then for any
$\zeta\in\Um$ we get the following.
\begin{itemize}
\item[(i)] Under \cref{A3.4}(a), we have
\begin{equation}\label{EL3.2A}
\Exp_i^{\zeta}\left[e^{\gamma \uuptau(\sB)}\Lyap(X_{\uuptau(\sB)})\right] \,\le\, \Lyap(i)\ \ \text{ for all\ } i\in\sB^c\,,
\end{equation}
where $\uuptau(\sB)\,=\,\inf\{t>0\colon X_t\in \sB\}$\,.
\item[(ii)] Under \cref{A3.4}(b), we have
\begin{equation}\label{EL3.2B}
\Exp_i^{\zeta}\left[e^{\int_{0}^{\uuptau(\sB)}\ell(X_s)\D s}\Lyap(X_{\uuptau(\sB)})\right] \,\le\, \Lyap(i)\ \ \text{ for all\ } i\in\sB^c\,.
\end{equation}
\end{itemize}
\end{lemma}

\begin{proof}
We only provide a proof for (i) and the proof for (ii) would be analogous. 
Suppose \cref{A3.4}(a) holds. Let $\sD_n$ be a collection of finite, increasing sets converging to $S$.
By $\uptau_n$ we denotes the first exit time from $\sD_n$. Choose $n$ large enough so that 
$\sB\Subset \sD_n$.
Applying Dynkin's formula \cite[Appendix~C.3]{GHL09} and using \cref{Lyap1} it follows that
\begin{equation*}
\Exp_i^{\zeta}\left[e^{\gamma (\uuptau(\sB)\wedge T \wedge \uptau_n)}\Lyap(X_{\uuptau(\sB)\wedge T \wedge \uptau_n})\right] \,\le\, \Lyap(i)\ \ \text{ for all\ } i\in\sB^c\cap \sD_n\,,
\end{equation*}
for $T>0$. Letting $T\to\infty$ first and, then $n\to\infty$ and applying Fatou's lemma we obtain \cref{EL3.2A}. This completes the proof.  
\end{proof}

%%%%%%%%%%%%%%%%%%%%%%%%%%%%%%%%%%%%%%%%%%%%%%%%%%%%%%%%%%%%%%%%%%%%%%%%%%%%%%%%%%%%%%%%%%%%%%%%%%%%%

Let $\{\sD_n\}$ be a collection of finite, increasing sets converging to $S$. 
Denote by $(\rho_n, \psi_n)$ the eigenpair in the domain $\sD_n$ obtained by \cref{L3.1}.
Next we study limit of $\rho_n$ as $n\to\infty$.

\begin{lemma}\label{L3.3}
Grant \cref{A3.1,A3.1Exp,A3.3,A3.4}. Then the following holds.
\begin{itemize}
\item[(i)]
\begin{equation}\label{EL3.3A}
\rho_{n} \le \inf_{\zeta\in\Uadm}\limsup_{T\to\infty}\frac{1}{T}\log \Exp_{i_0}^{\zeta} \left[e^{\int_{0}^{T} c(X_t, \zeta_t)\D t}\right]=\inf_{\zeta\in\Uadm} \sE_{i_0}(c, \zeta)\,,
\end{equation}
and $\{\rho_n\}$ is bounded above.
\item[(ii)] The sequence $\{\rho_n\}$ is bounded and we have $\liminf_{n\to\infty}\rho_n \ge 0$\,.
\end{itemize}
\end{lemma} 

\begin{proof}
Since
\begin{equation}\label{EL3.3B}
\rho_{n}\psi_{n}(i) = \min_{u\in\Act(i)}\left[\sum_{j\in S} \psi_{n}(j)q(j|i,u) + c(i,u)\psi_{n}(i)\right]\quad i\in\sD_n,
\end{equation}
and $\psi_n\gneq 0$, it follows from \cref{A3.3}(c) that $\psi_n(i_0)>0$. Then \eqref{EL3.3A} follows
from \eqref{EL3.1B}. To complete the proof of (i) we only need to show
that
\begin{equation}\label{00A}
\sE_i(c, \zeta)\leq \kappa\quad \forall \; \zeta\in\Uadm, \; \; i\in S,
\end{equation}
for some constant $\kappa$. Since $c$ is bounded under \cref{A3.4}(a), the
above is immediate. Under \cref{A3.4}(b), we write \eqref{Lyap2} as
$$\sup_{u\in\Act(i)} \sum_{j\in S} \Lyap(j) q(j|i,u) \le (\kappa_1 - \ell(i))\Lyap(i) \quad \forall \quad i\in S\,,$$
where $\kappa_1=\widehat{C} \max_{\cK}[\Lyap]^{-1}$. Applying the arguments
of \cref{L3.2} we then get
$$\Exp_i^{\zeta}\left[e^{\int_0^T (\ell(X_t)-\kappa_1)\D{t}}\Lyap(X_{T})\right] \,\le\, \Lyap(i).$$
Since $\Lyap\geq 1$, taking logarithm in the above, dividing by $T$ and
letting $T\to\infty$ we obtain that 
$$\sE_i(\ell, \zeta)\leq \kappa_1\quad \forall \; \zeta\in\Uadm, \; \; i\in S.$$
Again, since $\max_{u\in\Act(\cdot)} c(\cdot, u)\leq \ell(\cdot) + \kappa_2$ for some $\kappa_2>0$ by \cref{A3.4}(b), we have \eqref{00A} from the above estimate.

Next we consider (ii).
First we show that $\rho_n$ is bounded below.  Since $\psi_n(i_0) > 0$, normalizing $\psi_n$ we can assume that $\psi_n(i_0) = 1$. Since $c \ge 0$, from \cref{EL3.3B}, it follows that
\begin{equation*}
\rho_n\geq \min_{u\in\Act(i_0)}\left[\sum_{j\in S}\psi_n(j)q(j|i_0,u)\right]\geq 
\min_{u\in\Act(i_0)}q(i_0|i_0,u).
\end{equation*} 
Thus $\rho_n$ is bounded from below.

Thus we remain to show that $\hat{\rho} = \liminf_{n\to\infty} \rho_n \ge 0$.
Suppose, on the contrary, that $\hat{\rho} < 0$. We therefore have, along some subsequence, $\rho_n \to \hat{\rho}$, as $n\to\infty$.
Thus, using \cref{A3.1}(c) and \eqref{EL3.3B},  we have 
\begin{equation}\label{EL3.3C}
\psi_n(j)\le \sup_{u\in\Act(i_0)}\frac{-q(i_0|i_0, u)}{q(j|i_0,u)}\df \kappa_1
\quad \text{for all}\; j\in S\setminus\{i_0\},
\end{equation}
for all large $n$.
 Hence, by a standard diagonalization argument, there exists a function $\psi$ with $\psi(i_0) = 1$
 such that along some subsequence $\psi_n(i)\to\psi(i)$, as $n\to\infty$, for all $i\in S$. 
Let $\tilde{v}_n$ be a minimizing selector of \eqref{EL3.3B}. 
  Since $\Act(i)$ is compact for each $i\in S$, along a further subsequence, 
  $\tilde{v}_{n}(i)\to\tilde{v}(i)$, as $n\to\infty$, for all $i\in S$.
Therefore, letting $n\to\infty$ in 
\begin{equation*}
\rho_{n}\psi_{n}(i) = \left[\sum_{j\in S} \psi_{n}(j)q(j|i,\tilde{v}_n(i)) + 
c(i,\tilde{v}_n(i))\psi_{n}(i)\right],
\end{equation*}
using \cref{A3.3}(a)-(b) and Fatou's lemma, we obtain
\begin{equation}\label{EL3.3D}
\hat{\rho}\psi(i) \geq \left[\sum_{j\in S} \psi(j)q(j|i,\hat{v}(i)) + c(i,\tilde{v}(i))\psi(i)\right]\,
\quad i\in S\,.
\end{equation}
Since $\hat{\rho} < 0$ and $c \geq 0$, from \cref{EL3.3D} we deduce that
\begin{equation}\label{EL3.3E}
\left[\sum_{j\in S} \psi(j)q(j|i,\tilde{v}(i))\right] \leq 0\quad\forall \;\; i\in S\,.
\end{equation}
Applying Dynkin's formula to \cref{EL3.3E}, for any $t > 0$ and $i\in S$, it follows that
\begin{equation*}
\Exp_{i}^{\hat{v}}[\psi(X_t)] \le \psi(i)\,.
\end{equation*} 
Therefore, $\{\psi(X_t)\}$ is a supermartingale with respect to the canonical filtration of
{\bf X}, and thus, by Doob's martingale convergence theorem $\psi(X_t)$ converges as $t\to\infty$. By \cref{A3.4},
{\bf X} is recurrent implying the skeleton process $\{X_n\; :\; n\in\NN\}$ is also recurrent
(cf. \cite[Proposition~5.1.1]{And91}). Therefore, $\{X_n\; :\; n\in\NN\}$ visits every
state of $S$ infinitely often and this is possible only if $\psi\equiv 1$. This contradicts 
\eqref{EL3.3D}. Thus we must have $\hat{\rho}\geq 0$.
 This completes the proof.
\end{proof}

%%%%%%%%%%%%%%%%%%%%%%%%%%%%%%%%%%%%%%%%%%%%%%%%%%%%%%%%%%%%%%%%%%%%%%%%%%%%%%%%%%%%%%%%%%%%%
Using \cref{L3.3} and ideas from \cref{L2.4} we can now establish the existence of an eigenfunction on $S$.
\begin{lemma}\label{L3.4}
Consider \cref{A3.1,A3.1Exp,A3.3,A3.4}. Then there exists $(\rho, \psi^*)\in \RR_{+}\times\order(\Lyap)$, with $\psi^* > 0$, satisfying 
\begin{equation}\label{EL3.4A}
\rho\psi^*(i) \,=\, \min_{u\in\Act(i)}\left[\sum_{j\in S} \psi^*(j)q(j|i,u) + c(i,u)\psi^*(i)\right]
\quad i\in S\,.
\end{equation}
Moreover, we have 
\begin{itemize}
\item[(i)]
\begin{equation}\label{EL3.4B}
\rho \le \lamstr\,.
\end{equation}
\item[(ii)] There exists a finite set $\sB \supset \cK$ such that for any minimizing selector $v^{*}$ of \cref{EL3.4A} we have
\begin{equation}\label{EL3.4C}
\psi^*(i) = \Exp_i^{v^*} \left[e^{\int_{0}^{\uuptau(\sB)}(c(X_t, v^*(X_t)) - \rho)\D{s}}\psi^*(X_{\uuptau(\sB)})\right]\quad\forall \,\, i\in \sB^c\,.
\end{equation}  
\end{itemize}

\end{lemma}

\begin{proof}
From \cref{L3.3} we know that the sequence $\{\rho_n\}$ is a bounded and $\liminf_{n\to\infty}\rho_n \ge 0$. Thus one can find a subsequence such that along this subsequence, $\rho_n$ converges to some 
$\rho \ge 0$, as $n\to\infty$. Now we repeat the method of \cref{L2.4}.
Using \cref{A3.4} and the fact $c \ge 0$, we can find a finite set $\sB$ containing $\cK$ such that
\begin{itemize}
\item[(i)]Under \cref{A3.4}(a): 
\begin{equation}\label{EL3.4D}
(\max_{u\in\Act(i)}c(i, u) - \rho_n) < \gamma \quad \forall\,\, i\in\sB^c, \; \text{for all $n$ large}.
\end{equation}
\item[(ii)]Under \cref{A3.4}(b): 
\begin{equation}\label{EL3.4E}
(\max_{u\in\Act(i)}c(i, u) - \rho_n) < \ell(i) \quad \forall\,\, i\in\sB^c, \; \text{for all $n$ large}.
\end{equation}
\end{itemize}
Now we scale $\psi_n$  by multiplying a suitable
scalar so that it touches $\Lyap$ from below. In particular, 
define 
\begin{equation*}
\theta_n \,=\, \sup\{ \kappa>0 \;\colon\; (\Lyap - \kappa\psi_n) > 0 \quad\text{in\ } S\}\,.
\end{equation*}
Replacing $\psi_n$ by $\theta_n\psi_n$ and from the arguments of \cref{L2.4} we see that
$\psi_n$ touches $\Lyap$ inside $\sB$.
Since $\psi_n \le \Lyap$ for all large $n$, by a standard diagonalization argument, one can extract a subsequence so that along this subsequence, $\psi_n(i) \to \psi^*(i)$ for all $i\in S$, as $n\to\infty$, and $\psi^*\le\Lyap$. Using \eqref{EL3.3B} and Fatou's lemma we get that (see \cref{EL3.3D})
\begin{equation}\label{EL3.4F}
\rho\psi^*(i) \ge\, \min_{u\in\Act(i)}\left[\sum_{j\in S} \psi^*(j)q(j|i,u) +  c(i,u)\psi^*(i)\right]
\quad i\in S\,.
\end{equation} 
On the other hand, for every $i\in S$ and $u\in \Act(i)$, we obtain from \eqref{EL3.3B} that
\begin{align*}
\rho\psi^*(i)=\lim_{n\to\infty} \rho_n\psi_n
&= \lim_{n\to\infty}\, \min_{u\in\Act(i)}\left[\sum_{j\in S} \psi_n(j)q(j|i,u) +  c(i,u)\psi_n(i)\right]
\\
&\leq \lim_{n\to\infty}\, \left[\sum_{j\in S} \psi_n(j)q(j|i,u) +  c(i,u)\psi_n(i)\right]
\\
&= \sum_{j\in S} \psi^*(j)q(j|i,u) +  c(i,u)\psi^*(i),
\end{align*}
by dominated convergence theorem, where we use the fact that $\psi_n\leq \Lyap$ for all large $n$.
Since $u$ is arbitrary, combining with \eqref{EL3.4F} we get \eqref{EL3.4A}. Next we show that $\psi^*>0$.
By our construction we have $(\Lyap - \psi_n) = 0$ at some point in $\sB$ for all $n$ large. Hence $(\Lyap - \psi^*) = 0$ at some point in $\sB$. Since $\Lyap \ge 1$, we deduce that $\psi^*$ is nonzero. We claim that $\psi^* > 0$. If not, then we must have $\psi^*(i) = 0$ for some $i\in S$. Then for any minimizing selector $v^*$ of \cref{EL3.4A}, we have
\begin{equation}\label{EL3.4H}
\sum_{j\neq i} \psi^*(j)q(j|i,v^*(i)) = 0\,.
\end{equation} 
Since the Markov chain {\bf X} is irreducible under $v^*$, from \cref{EL3.4H} it follows that 
$\psi^*\equiv 0$. This is a contradiction to fact that $\psi^*$ is nontrivial. This proves the claim.

Now we prove (i). In view of \eqref{EL3.1B}, it is enough to show that given $i\in S$, $\psi_n(i)>0$
for all $n$ large.
Since $\psi^* >0$ and $\psi_n(i)\to\psi^*(i)$ as $n\to\infty$, we have $\psi_n(i)>0$ for all large enough $n$. Hence $\lim_{n\to\infty}\rho_n\leq \inf_{\zeta\in \Uadm}\sE_i(c, \zeta)$ for all $i$. This gives us
\eqref{EL3.4B}.

(ii) follows from an argument similar to \cref{L2.4}.  
\end{proof}

\begin{remark}\label{R3.1}
It is easy to check that we can also apply the argument of \cref{L3.4} for every stationary Markov control. More precisely, if we impose \cref{A3.1,A3.1Exp,A3.3,A3.4},
then for every Markov control $v\in\Usm$ there exists $(\psi_v, \rho_v)\in\RR_+\times\order(\Lyap)$,
$\psi_v>0$, satisfying
\begin{equation}\label{ER3.2A}
\rho_v\psi_v(i) = \sum_{j\in S} \psi_v(j) q(j|i,v(i)) + c(i, v(i))\psi_v(i)\quad\forall\,\, i\in S\,.
\end{equation}
Furthermore, $\rho_v\leq \inf_i\,\sE_i(c,v)$ and for some finite set $\sB\supset\sK$ 
\begin{equation}\label{ER3.2B}
\psi_v(i) = \Exp_i^{v} \left[e^{\int_{0}^{\uuptau(\sB)}(c(X_t, v(X_t)) - \rho)}\psi_v(X_{\uuptau(\sB)})\right]\quad\forall \,\, i\in \sB^c\,.
\end{equation}
\end{remark}

Next we show that $\rho=\lamstr$. To do so we can use the perturbed cost $\tilde{c}_n$ introduced in
\cref{S-DT}. In fact, following an argument similar to \cref{L2.6} we can prove the following.
\begin{lemma}\label{L3.5}
Assume \cref{A3.1,A3.1Exp,A3.3,A3.4}. Then any minimizing selector of \cref{EL3.4A}, that is, any $v^*\in\Usm$ satisfying
\begin{equation}\label{EL3.5A}
\min_{u\in\Act(i)}\left[\sum_{j\in S} \psi^*(j) q(j | i, u)
+ c(i, u)\psi^*(i)\right] = \left[\sum_{j\in S} \psi(j) q(j | i, v^*(i)) +
c(i, v^*(i))\psi^*(i)\right]\quad\forall\,\, i\in S\,,
\end{equation}
is an optimal control and $\rho=\lamstr$. Moreover, $\psi^*$ is the unique solution of \eqref{EL3.4A} with $\psi^*(i_0)=1$.
\end{lemma}
Now we are ready to complete the proof of \cref{T3.1}.
\begin{proof}[Proof of \cref{T3.1}]
(i) follows from \cref{L3.4,L3.5}. By \cref{L3.5} we also get that any minimizing selector
of \eqref{ET3.1B} is an optimal Markov control. Using \cref{R3.1} and the arguments of \cref{T2.1} we
can also show the converse direction, that is, if for some $v\in\Usm$ we have $\sE_i(c, v)=\lamstr$
then $v$ satisfies \eqref{EL3.5A}. This gives us (ii).
\end{proof}

We conclude this section with the following remark.
\begin{remark}
\cref{A3.3}(c) can be replaced by other similar assumption. For instance,
if the killed process communicates with every state from
$i_0$ before leaving the domain $\sD_n$, for large $n$, then our method applies. More precisely,
for every $\sD_n$, $v\in\Usm$ and
 for every $j\in \sD_n\setminus\{i_0\}$, if there exists distinct $i_1, i_2, \ldots, i_m \in \sD_n\setminus\{i_0\}$
satisfying
$$q(i_1|i_0, v(i_0))q(i_2|i_1, v(i_1))\cdots q(j|i_m, v(i_m))>0\,,$$
then the conclusion of \cref{T3.1} holds. Note that in this case we also get $\psi_n(i_0)>0$ in $\sD_n$.
\end{remark}

%%%%%%%%%%%%%%%%%%%%%%%%%%%%%555%%%%%%%%%%%%%%%%%%%%%%%%%%%%%%%%%%%%%%%%%%%%%%%%%%%%%%%%%
\subsection{Near-monotone cost.}
In this section we replace \cref{A3.4} with a near-monotone assumption stated below.

\begin{assumption}\label{A3.5}
Define $\lambda_{\rm m}=\inf_{i\in S}\inf_{v\in \Usm}\sE_i(c, v)$, and 
$$\inf_{v\in\Usm}\sE_i(c, v)<\infty\quad \forall\; i\in S,$$
and the cost function $c$ satisfies the {\it near-monotone} condition with respect to $\lambda_{\rm m}$,
that is,
\begin{equation}\label{EA3.5A}
\liminf_{n\to\infty}\, \inf_{k\geq n}\inf_{u\in\Act(k)} c(k, u)\, >\, \lambda_{\rm m}.
\end{equation}
\end{assumption}
Note that by \cref{A3.5}, $\lamstr\leq\lambda_{\rm m}<\infty$ where $\lamstr =\inf_{i\in S}\inf_{\zeta\in \Uadm}\sE_i(c, \zeta)$ is
given by \eqref{00A2}.
The following result gives the existence of an
optimal stationary Markov control. This result should be compared with \cite{SKP15} where existence of 
optimal stationary Markov control is obtained under \eqref{EA3.5A}, but our
 hypotheses are weaker and we also allow history dependent controls (see (A1), (A2)(ii) in \cite{SKP15}).

\begin{theorem}\label{T3.2}
Grant \cref{A3.1,A3.1Exp,A3.3,A3.5}. We also assume that {\bf X} is recurrent under any control $v\in\Usm$.
Then there exists a positive function $\psi^*$ satisfying
\begin{equation}\label{ET3.2AA}
\lamstr \psi^*(i) \geq \min_{u\in\Act(i)}\left[\sum_{j\in S} \psi^*(j)q(j|i,u)
+ c(i, u)\psi^*(i)\right]\quad\text{for}\,\, i\in S\,.
\end{equation}
Futhermore, we have $\lamstr=\lambda_{\rm m}$ and any measurable selector of \eqref{ET3.2AA} is an
optimal stationary Markov control.
\end{theorem}

\begin{proof}
The proof is similar to \cref{T2.3}. Recall the eigenpair $(\rho_n, \psi_n)$ satisfying
\begin{equation}\label{ET3.2A}
\rho_{n}\psi_{n}(i) = \min_{u\in\Act(i)}\left[\sum_{j\in S} \psi_{n}(j)q(j|i,u) + c(i,u)\psi_{n}(i)\right]\quad i\in\sD_n.
\end{equation}
Since the CTCMP {\bf X} is  recurrent under any stationary Markov control, the argument of 
\cref{L3.3} works, and therefore, we have
\begin{equation}\label{ET3.2B}
0\leq \liminf_{n\to\infty}\rho_n\leq \limsup_{n\to\infty} \rho_n\leq 
\inf_{\zeta\in \Uadm}\sE_{i_0}(c, \zeta).
\end{equation}
Normalize $\psi_n$ to satisfying $\psi_n(i_0)=1$. Let $\tilde{v}_n$ be a minimizing selector of 
\eqref{ET3.2A}. Using \eqref{ET3.2A}-\eqref{ET3.2B} it then follows that 
$\psi_n(i)\leq \kappa_i$ for all $n$ (see \eqref{EL3.3C}), for some constant $\kappa_i$. Using a
standard diagonalization argument, we can find a subsequence along which we have
\begin{equation*}
\rho_n\to \rho,\quad \psi_n(i)\to \psi^*(i)\quad \text{and}\quad \tilde{v}_n(i)\to v(i),
\end{equation*}
for all $i\in S$. From Fatou's lemma, we then have
\begin{align}\label{ET3.2C}
\rho\psi^*(i) &\geq \sum_{j\in S} \psi^*(j)q(j|i,v(i)) +  c(i,v(i))\psi^*(i)\nonumber
\\
&\ge\, \min_{u\in\Act(i)}\left[\sum_{j\in S} \psi^*(j)q(j|i,u) +  c(i,u)\psi^*(i)\right]
\quad i\in S\,.
\end{align} 
Note that $\psi^*(i_0)=1$. Using irreducibility it is then easy to see that $\psi^*>0$. This in particular,
implies that for any $i\in S$ we have $\psi_n(i)>0$ for all large $n$. Using 
\eqref{EL3.1B} and \eqref{ET3.2A} we obtain $\rho\leq \lamstr\leq \lambda_{\rm m}$.
Since $\rho\leq\lambda_{\rm m}$, applying the near-monotonicity
condition \eqref{EA3.5A}, we find a finite set $\sB$ such that 
\begin{equation}\label{ET3.2D}
\inf_{u\in \Act(i)} c(i, u)-\rho>0\quad \text{for}\; i\in \sB^c.
\end{equation}
Now consider a measurable selector $v$ of \eqref{ET2.3C} and by $\uuptau=\uuptau(\sB)$ we denote the first hitting time to $\sB$. Since {\bf X} is recurrent under $v$, we have $\Prob_i(\uuptau<\infty)=1$
for all $i\in\sB^c$ (cf. \cite[Proposition~5.1.1]{And91}). Thus applying
Dynkin's formula on \eqref{ET3.2C}, followed by Fatou's lemma, we get
$$ \psi^*(i)\geq \Exp_i^v\left[e^{\int_{t=0}^{\uuptau-1}(c(X_t, v(X_t))-\rho)} \psi^*(X_{\uuptau})\right]
\ge \min_{j\in \sB}\psi^*(j)
\quad i\in\sB^c,$$
using \eqref{ET3.2D}.
Thus, $\inf_{i\in S}\psi^*(i)>0$.  Now we can repeat the argument of \cref{T3.2} to show that
$\rho=\lamstr=\lambda_{\rm m}=\sE_i(c, v)$ for all $i$. This completes the proof.
\end{proof}

In a similar fashion we can extend \cref{T2.4} to a continuous time set-up. 
\begin{theorem}\label{T3.3}
Grant \cref{A3.5}. Also, assume that
there exists a function $W:S\to [1, \infty)$ satisfying $W(i)\geq i$ for all large $i$ and 
\begin{equation*}
\sup_{u\in\Act(i)}\sum_{j\in S} W(i) q(j|i, u)\leq g(i)\quad \text{for}\; i\in S,
\end{equation*}
for some function $g:S\to \RR$ satisfying $\lim_{i\to\infty} g(i)=0$. Furthermore, for some
$\eta>0$ we have
\begin{equation*}
\min_{u\in\Act(i)}q(i-1|i, u)\,\geq\, \eta\quad \text{for all}\; i\geq 1,
\end{equation*}
and $q(\cdot|1, u)$ supported in a finite set $C$, independent of $u$. We also assume that
 for $\sD_n\df\{1, \ldots, n\}$, $v\in\Usm$ and any $j\in\sD_n\setminus\{1\}$ there exists
distinct $i_1, i_2, \ldots, i_k\in \sD_n$ we have
\begin{equation*}
q(i_1|1, v(1))q(i_2|i_1, v(i_1))\cdots q(j|i_k, v(i_k))>0\,.
\end{equation*}
Furthermore, {\bf X} is irreducible under any stationary Markov control.
Then there exists a positive function $\psi^*$ satisfying
\begin{equation}\label{ET3.3A}
\lamstr \psi^*(i) \geq \min_{u\in\Act(i)}\left[\sum_{j\in S} \psi^*(j)q(j|i,u)
+ c(i, u)\psi^*(i)\right]\quad\text{for}\,\, i\in S\,.
\end{equation}
Futhermore, we have $\lamstr=\lambda_{\rm m}$ and any measurable selector of \eqref{ET3.3A} is an
optimal stationary Markov control.
\end{theorem}
The proof is analogous to \cref{T2.4} and thus omitted.

%%%%%%%%%%%%%%%%%%%%%%%%%%%%%%%%%%%%%%%%%%%%%%%%%%%%%%%%%%%%%%%%%%%%%%%%%%%%%%%%%%%%%%%%%%%%%%%%%%%%%%%

%%%%%%%%%%%%%%%%%%%%%%%%%%%%%%%%%%%%%%%%%%%%%%%%%%%%%%%%%%%%%%%%%%%%%%%%%%%%%%%%%%%%%%%%%%%%%%%%%%%%%%%
\section{Policy iteration}\label{S-PIA}
In this section we propose policy improvement algorithms (PIA) and establish its convergence to the optimal
value. To do so, we introduce the notion of generalized Perron-Frobenius eigenvalue.
Our definition below can be seen as the counterpart of the elliptic generalized eigenvalue in the case of discrete Markov chain (cf. \cite{Berestycki-94,Berestycki-15,ABS19,NP92}).
For discrete time CMP we define the generalized Perron-Frobenius eigenvalue as follows
\begin{equation}\label{E4.1}
\lamd=\inf\{\lambda\in\RR\; :\; \exists\; \Psi>0\; \text{satisfying}\;
\min_{u\in\Act(i)} e^{c(i, u)}\sum_{j\in S} \Psi(j) P(j|i, u)\leq e^\lambda\Psi(i)\; \forall \; i\in S\}.
\end{equation}
Similarly, we can also define a generalized Perron-Frobenius eigenvalue for every stationary
Markov control $v$ as follows.
\begin{equation}\label{E4.2}
\lamd(v)=\inf\{\lambda\in\RR\; :\; \exists\; \Psi>0\; \text{satisfying}\;
e^{c(i, v(i))}\sum_{j\in S} \Psi(j) P(j|i, v(i))\leq e^\lambda\Psi(i)\; \forall \; i\in S\}.
\end{equation}
In relation with the risk-sensitive control we would be interested to know 
whether $\lamd=\lamstr$. Note that
for nondegenerate elliptic operator this is not true in general. See for instance, Example~3.1
in \cite{ABS19}. Our next result would be helpful in answering this question.

\begin{lemma}\label{L4.1}
Assume that {\bf X} is irreducible under every stationary Markov control and $\lamd$ is finite. Let
$\sD$ be a finite domain and $(\rho, \psi)\in\RR\times \cB^+_\sD, \psi\neq 0,$ be such that
\begin{equation}\label{EL4.1A}
\min_{u\in\Act(i)}e^{c(i,u)}\sum_{j\in S}\psi(j)P(j|i,u) \,=\, e^\rho\psi(i), \quad \forall\,\, i\in \sD\,, \quad\text{and\ } \psi(i) = 0\quad \forall\,\, i\in \sD^c\,.
\end{equation}
Then we must have $\rho\leq \lamd$. Similar result also holds under every stationary Markov control.
\end{lemma}

\begin{proof}
Suppose, on the contrary, that $\rho>\lamd$. From the definition of $\lamd$, we find
a pair $(\lambda, \Psi)$ with $\Psi>0$ satisfying
\begin{equation}\label{EL4.1B}
\min_{u\in\Act(i)} e^{c(i, u)}\sum_{j\in S} \Psi(j) P(j|i, u)\leq e^\lambda\Psi(i)\quad \text{for all}
\; i\in S.
\end{equation}
Consider a minimizing selector $v\in\Usm$ of the left-hand side of \eqref{EL4.1B}. From \eqref{EL4.1A} we then 
have
\begin{equation}\label{EL4.1C}
e^{c(i,v(i))}\sum_{j\in S}\psi(j)P(j|i,v(i)) \,\ge\, e^\rho\psi(i), \quad \forall\,\, i\in \sD\,, \quad\text{and\ } \psi(i) = 0\quad \forall\,\, i\in \sD^c\,.
\end{equation}
Define 
$$\theta=\sup\{\kappa>0\; :\; \Psi-\kappa\psi>0\; \text{in}\; \sD\}.$$
Since $\psi\gneq 0$, we have $\theta\in (0, \infty)$ and $V\df \Psi-\theta\psi$ is non-negative in 
$S$. Furthermore, $V$ must vanish at some point, say $j_0$, in $\sD$ and $V>0$ in $\sD^c$.

Since $\rho \geq \lambda$, using \eqref{EL4.1B} and \eqref{EL4.1C} we also get
\begin{equation}\label{EL4.1D}
V(i)\geq e^{c(i,v(i))-\lambda}\sum_{j\in S}V(j)P(j|i,v(i))\quad \forall\,\, i\in \sD\,.
\end{equation}
Denote by $\uptau=\uptau(\sD)$ the first exit time from $\sD$. From optional sampling theorem we
then obtain from \eqref{EL4.1D} that (see \eqref{EL2.1D}) 
$$0=V(j_0)\geq \Exp^v_{j_0}\left[ e^{\sum_{t=0}^{\uptau\wedge T -1}(c(X_t, v(X_t)-\lambda)}
V(X_{\uptau\wedge T})\right].$$
Letting $T\to\infty$ and applying Fatou's lemma we get from above that
$$0\geq \Exp^v_{j_0}\left[ e^{\sum_{t=0}^{\uptau -1}(c(X_t, v(X_t)-\lambda)}
V(X_{\uptau})\Ind_{\{\uptau<\infty\}}\right].$$
Since {\bf X} is irreducible, we have $\Prob^v_{j_0}(\uptau<\infty)>0$ and also $V>0$ in $\sD^c$. This
is clearly a contradiction to the above. Thus we must have $\rho\leq\lamd$.
\end{proof}

The following remark is immediate from \cref{R2.1}, \cref{L4.1} and the proof of \cref{T2.3}.
\begin{remark}\label{R4.1}
Suppose that \cref{A1.1} holds and {\bf X} is irreducible under every stationary Markov control.
Then $\lim_{n\to\infty}\rho_n=\lamd$. Moreover, if $\lamd<\infty$, there exists a positive eigenvector $V$
satisfying
$$\min_{u\in\Act(i)}e^{c(i, u)}\sum_{j\in S}V(j) P(j|i, u)\leq e^{\lamd} V(i)\quad i\in S\,.$$
\end{remark}

\subsection{Discrete time stable case}\label{S-DPIA}
In this section we propose a policy improvement algorithm (PIA) and show that it converges to 
the optimal value $\lamstr$. To this aim we use a stronger hypothesis compared to \cref{EA2.2}.
\begin{assumption}\label{A4.1}
We suppose that \cref{EA2.2} holds for a norm-like function $\Lyap$. Furthermore, in case of 
\cref{EA2.2}(b), we have $\max_{u\in\Act(i)} c(i, u)\leq \eta\ell(i)$ for $i\in S$ and some 
$\eta\in(0,1)$. Also, there exists a state $z_\circ$ in $S$ such that
\begin{equation}\label{EA4.1A}
\inf_{u\in \Act(i)} P(z_\circ|i, u)>0\quad \text{for all}\; i\in S.
\end{equation}
\end{assumption}
\eqref{EA4.1A} will be used to construct a suitable small set and to apply certain convergence result
from \cite{MeynT-94b}. \cref{A1.1,A4.1} are imposed throughout this section. Suppose that $v$ is
a stationary Markov control and $(\rho_v, \psi_v)$ be the corresponding eigenpair obtained in 
\cref{R2.2}. Let $\sB\supset \cK$ be such that for $i\in\sB^c$ we have
\begin{align}\label{E4.8}
\begin{rcases}
\max_{u\in\Act(i)} c(i, u) -\rho_v &< \alpha \gamma, \quad \text{under \cref{EA2.2}(a)},
\\
\max_{u\in\Act(i)} c(i, u) -\rho_v &< \alpha \ell(i), \quad \text{under \cref{EA2.2}(b)},
\end{rcases}
\end{align}
for some $\alpha\in (0, 1)$.
Then the arguments of \cref{L2.4} gives us
\begin{equation}\label{E4.9}
\psi_v(i) = \Exp_i^{v} \left[e^{\sum_{t=0}^{\uuptau(\sB) - 1}(c(X_t, v(X_t)) - \rho_v)}\psi_v(X_{\uuptau(\sB)})\right]\quad\forall \,\, i\in \sB^c\,.
\end{equation}
We are going to use this observation in the later part of this section.
 We now describe our 
PIA.

\begin{algorithm}\label{Alg4.1}
Policy iteration.
\begin{itemize}
\item[1.] Initialization. Set $k=0$ and select any $v_0\in\Usm$.
\smallskip
\item[2.] Value determination. Let $V_k$
be the unique principal eigenfunction satisfying $V_k(i_0)=1$ and
\begin{equation}\label{E4.10}
e^{\rho_k} V_k(i) = e^{c(i, v_k(i))} \sum_{j\in S} V_k(j) P(j|i, v_k(i))\quad i\in S.
\end{equation}
Existence of a unique principal eigenfunction in \cref{E4.10} follows from
\cref{R2.2,L2.6}. We let $\lambda_k\df\lamd(v_k)=\sE_i(c, v_k)=\rho_k$.
\smallskip
\item[3.] Policy improvement. Choose any $v_{k+1}\in\Usm$ satisfying
$$v_{k+1}(i)\in \Argmin_{u\in\Act(i)}\,
\left[e^{c(i, u)} \sum_{j\in S} V_k(j) P(j|i, u)\right], \quad i\in S\,.$$
\end{itemize}
\end{algorithm}

Then our main result of this section is
\begin{theorem}\label{T4.1}
Under \cref{A1.1,A4.1} the following hold.
\begin{itemize}
\item[(i)] For all $k\in\NN$, we have $\lambda_{k+1}\leq\lambda_k$ and $\lim_{k\to\infty}\lambda_k=\lamstr$.
\item[(ii)] $V_k$ converges pointwise, as $k\to\infty$, to $\psi^*$ where $\psi^*$ is the unique solution to  \eqref{ET2.1A}.
\end{itemize}
\end{theorem}
Our proof of \cref{T4.1} is inspired from \cite{ABP20} which also establishes convergence of PIA for controlled 
diffusions. The proof technique of \cite{ABP20} uses several estimates from elliptic partial differential equations which are not obvious in the present situation. So our proofs requires a more careful analysis. We denote by
$c_k(i)=c(i, v_k(i))$. It is also obvious from \eqref{E4.2} that $\lambda_{k+1}\leq \lambda_k$ for all $k\geq 0$. Fix a set $\sB$
containing $\cK$. Since $\rho_v\geq 0$, from \cref{A4.1} \eqref{E4.8} holds for $\rho_{v_0}=\lambda_0$. Let $\kappa_k=\min_{\sB}\frac{\Lyap}{V_k}$
and replace $V_k$  by $\kappa_k V_k$. Using \eqref{E4.9}
and \cref{L2.2}, it then follows that
$V_k\leq \Lyap$ in $S$.
Define
\begin{equation}\label{theta}
\theta_{k+1}(i) = 1-\frac{1}{V_k(i)}e^{c_{k+1}(i)-\lambda_k}\sum_{j\in S} V_k(j)P(j|i, v_{k+1}(i)).
\end{equation}
Since
\begin{align*}
V_k(i)\geq \min_{u\in\Act(i)}\left[e^{c(i, u)-\lambda_k} \sum_{j\in S} V_k(j) P(j|i, u)\right]
= e^{c_{k+1}(i)-\lambda_k}\sum_{j\in S} V_k(j)P(j|i, v_{k+1}(i)),
\end{align*}
we have $0\leq \theta_{k}\leq 1$ for all $k\in\NN$. We begin with the following estimate 
which will be useful to establish convergence of $V_k$.
%%%%%%%%%%%%%%%%%%%%%%%%%%%%%%%%%%%%%%%%%%%%%%%%%%%%%%%%%%%%%%%%%%%%%%%%%%%%%%%%%%%%%%%
\begin{lemma}\label{L4.2}
Grant \cref{A1.1,A4.1}. Then the following hold.
\begin{itemize}
\item[(i)] There exists $\kappa$, independent of $k$, such that
\begin{equation}\label{EL4.2A}
V_k(i)\leq \kappa (\Lyap(i))^\alpha\quad \text{for all}\; i\in S,
\end{equation}
where $\alpha$ is given by \eqref{E4.8}.

\item[(ii)] For every $i\in S$ we have $\inf_{k\in\NN} V_k(i)>0$.
\end{itemize}
\end{lemma}

\begin{proof}
(i) actually follows from \eqref{E4.8} and \cref{L2.2}. Since  $\lambda_k\leq\lambda_0$, we see from
\eqref{E4.8} that
\begin{align*}
\max_{u\in\Act(i)} c(i, u) -\lambda_k &< \alpha \gamma, \quad \text{under \cref{EA2.2}(a)},
\\
\max_{u\in\Act(i)} c(i, u) -\lambda_k &< \alpha \ell(i), \quad \text{under \cref{EA2.2}(b)},
\end{align*}
for all $i\in\sB^c$.
Therefore, the stochastic representation of $V_k$ is possible with respect to $\sB$, that is, 
$$V_k(i) = \Exp_i^{v_k} \left[e^{\sum_{t=0}^{\uuptau(\sB) - 1}(c_k(X_t) - \lambda_k)}V_k(X_{\uuptau(\sB)})\right]\quad\forall \,\, i\in \sB^c\,.$$
Let us consider \cref{EA2.2}(a) first.
Since $\alpha\in (0,1)$, from \cref{L2.2} it follows that
\begin{align*}
V_k(i) &\leq\Exp_i^{v_k} \left[e^{\alpha\gamma\uuptau(\sB)}V_k(X_{\uuptau(\sB)})\right]
\\
&\leq \Exp_i^{v_k}\left[e^{\gamma\uuptau(\sB)}(V_k)^{\nicefrac{1}{\alpha}}(X_{\uuptau(\sB)})\right]^\alpha
\\
&\leq (\max_\sB\Lyap )^{1-\alpha} \Exp_i^{v_k}\left[e^{\gamma\uuptau(\sB)}\Lyap(X_{\uuptau(\sB)})\right]^\alpha
\\
& \leq (\max_\sB\Lyap )^{1-\alpha}\, \Lyap^\alpha(i)
\end{align*}
for all $i\in \sB^c$. This gives (i). Similar argument also works for \cref{EA2.2}(b).

Next we consider (ii). Fix $i\in S$. Suppose, on the contrary, that $V_k(i)\to 0$, along some
subsequence, as $k\to\infty$. Using a standard diagonalization argument and the bound in (i),
we can find a further subsequence so that
$$\lambda_k\to\lambda_\circ,\quad v_k(j)\to v(j), \quad V_k(j)\to V(j)$$
for all $j\in S$, as $k\to\infty$. It is also evident from (i) that 
$V(j)\leq \kappa (\Lyap(j))^\alpha$ for all $j$. We claim that
\begin{equation}\label{EL4.2B}
e^{c(j, v_k(j))} \sum_{z\in S} V_k(z) P(z|j, v_k(j))\to e^{c(j, v(j))} \sum_{z\in S} V(z) P(z|j, v(j))
\end{equation}
as $k\to\infty$, for all $j$. Note that for any given $\epsilon>0$, since $\Lyap$ is norm-like,
we can find $z_1\in S$ such that 
$$\sum_{z\geq z_1} V_k(z) P(z|j, v_k(j))
\leq [\kappa \sup_{j\geq z_1} \Lyap^{\alpha-1}(j)]\, \sum_{z\ge z_1} \Lyap(z) P(z|j, v_k(j))<\epsilon,$$
using \eqref{EL4.2A} and \cref{EA2.2}. Thus, applying \cref{A1.1}(a), we get \eqref{EL4.2B}.
Now passing the limit in \eqref{E4.10} we have
$$e^{\lambda_\circ} V(j)=e^{c(j, v(j))} \sum_{z\in S} V(z) P(z|j, v(j))\quad j\in S,$$
and $V(i)=0$. On the other hand, since $\max_{\sB}(\Lyap-V_k)=0$, we must have 
$\max_{\sB}(\Lyap-V)=0$ implying $V$ is positive at some point in $\sB$. Applying the arguments
of \cref{L2.4} (see \eqref{EL2.4JJ}) we get a contradiction. This proves (ii).
\end{proof}
We also need the following uniqueness result.
\begin{theorem}\label{T4.2}
Suppose that $(\rho, W)\in\RR_+\times\sorder(\Lyap)$ be such that $W>0$
and 
\begin{equation}\label{ET4.2A}
W(i)=\min_{u\in\Act(i)}\left[e^{(c(i, u)-\rho)}\sum_{j\in S} W(j) P(j|i, u)\right]
\quad i\in S.
\end{equation}
Moreover, assume that $\rho\geq \lamstr$. Then we must have $\rho=\lamstr$
and $W$ is a scalar multiple of $\psi^*$.
\end{theorem}

\begin{proof}
Let $v$ be a minimizing selector of \eqref{ET2.1A}. Then for a finite set $\sB$
containing $\cK$ and satisfying 
\begin{align*}
\max_{u\in\Act(i)} c(i, u) -\lamstr &< \alpha \gamma, \quad \text{under \cref{EA2.2}(a)},
\\
\max_{u\in\Act(i)} c(i, u) -\lamstr &< \alpha \ell(i), \quad \text{under \cref{EA2.2}(b)},
\end{align*}
for all $i\in \sB^c$, we have the representation
\begin{equation}\label{ET4.2B}
\psi^*(i) = \Exp_i^{v} \left[e^{\sum_{t=0}^{\uuptau(\sB) - 1}(c(X_t, v^*(X_t)) - \lamstr)}\psi^*(X_{\uuptau(\sB)})\right]\quad\forall \,\, i\in \sB^c\,,
\end{equation}
from \cref{L2.4}. Using \eqref{ET4.2A} we get
$$W(i)\leq e^{(c(i, u)-\rho)}\sum_{j\in S} W(j) P(j|i, v(i))
\leq e^{(c(i, u)-\lamstr)}\sum_{j\in S} W(j) P(j|i, v(i)).$$
Denoting $\uptau_n$, the first exit time from $\{1, 2, \ldots, n\}$, and applying
Dynkin's formula we obtain
$$W(i)\leq \Exp_i^{v} \left[e^{\sum_{t=0}^{\uuptau(\sB)\wedge\uptau_n\wedge T - 1}(c(X_t, v^*(X_t)) - \lamstr)}W(X_{\uuptau(\sB)\wedge\uptau_n\wedge T})\right],$$
for $i\in\sB^c\cap\sD_n$. In view of \cref{L2.2}, we can let $T\to\infty$ in
the above to obtain
\begin{equation}\label{ET4.2C}
W(i)\leq \Exp_i^{v} \left[e^{\sum_{t=0}^{\uuptau(\sB)\wedge\uptau_n - 1}(c(X_t, v^*(X_t)) - \lamstr)}W(X_{\uuptau(\sB)\wedge\uptau_n})\right].
\end{equation}
Since $W\in \sorder(\Lyap)$, it also follows from \eqref{EA2.2A} that
\begin{align*}
\Exp_i^{v} \left[e^{\sum_{t=0}^{\uptau_n - 1}(c(X_t, v^*(X_t)) - \lamstr)}W(X_{\uptau_n})\Ind_{\{\uptau_n<\uuptau(\sB)\}}\right]
&\leq
\sup_{j> n}\frac{W(j)}{\Lyap(j)}\, \Exp_i^{v} \left[e^{\gamma\uptau_n}\Lyap(X_{\uptau_n})\Ind_{\{\uptau_n<\uuptau(\sB)\}}\right]
\\
&\leq \sup_{j> n}\frac{W(j)}{\Lyap(j)}\, \Lyap(i)\to 0,
\end{align*}
as $n\to\infty$. Thus, letting $n\to\infty$ in \eqref{ET4.2C}, we obatin
\begin{equation*}
W(i)\leq \Exp_i^{v} \left[e^{\sum_{t=0}^{\uuptau(\sB) - 1}(c(X_t, v^*(X_t)) - \lamstr)}W(X_{\uuptau(\sB)})\right] \quad i\in \sB^c.
\end{equation*}
Now we can apply the arguments in \cref{L2.6} (see \eqref{EL2.6H} and the argument that follows) to obtain that $W$ is a scalar multiple of $\psi^*$. Using
\eqref{ET4.2A} it then follows that $\rho=\lamstr$ and this completes the proof.
\end{proof}

Next we need the notion of twisted kernel. Consider the eigenpair $(\lambda_k, V_k)$ and define
the transition kernel
\begin{equation}\label{E4.13}
P^{(k)}(j|i)=\frac{V_k(j) P(j|i)}{\sum_{j\in S} V_k(j) P(j|i, v_k(i))}\quad j\in S.
\end{equation}
Let ${\mathbf Y^{(k)}}$ be the Markov process associated to this kernel. Since $P(j|i, v_k(i))$ generates
an irreducible Markov chain, ${\mathbf Y^{(k)}}$ is also irreducible. We denote by $\Exp^{(k)}[\cdot]$ the expectation operator with respect to the kernel $P^{(k)}$. For any nonnegative function $g$
we then have
\begin{equation}\label{E4.14}
\Exp_i^{v_k}\left[e^{\sum_{t=0}^{m-1}(c_k(X_t)-\lambda_k)} V_k(X_m)g(X_m)\right]=
V_k(i)\Exp^{(k)}_i[g(Y^{(k)}_m)]\quad i\in S.
\end{equation}
This can be proved by induction and Markov property. In particular, for $m=1$, \eqref{E4.14} follows from
\eqref{E4.10} and \eqref{E4.13}. Suppose \eqref{E4.14} is holds for $m-1$. Then 
\begin{align*}
\Exp_i^{v_k}\left[e^{\sum_{t=0}^{m}(c_k(X_t)-\lambda_k)} V_k(X_{m+1})g(X_{m+1})\right]
&= \Exp_i^{v_k}\left[e^{\sum_{t=0}^{m}(c_k(X_t)-\lambda_k)} \Exp_{X_m}[V_k(X_{1})g(X_{1})]\right]
\\
&=\Exp_i^{v_k}\left[e^{\sum_{t=0}^{m-1}(c_k(X_t)-\lambda_k)} V_k(X_m) \Exp^{(k)}_{X_m}[g(Y^{(k)}_{1})]\right]
\\
&= V_k(i) \Exp^{(k)}_i\left[ \Exp^{(k)}_{Y^{(k)}_m}[g(Y^{(k)}_{1})]\right]
= V_k(i) \Exp^{(k)}_i\left[ g(Y^{(k)}_{m+1})\right].
\end{align*}
This gives \eqref{E4.14}.

Let $C$ be a finite set containing $z_\circ, \sB$ and $\cK$ (see \cref{EA2.2}). Let
$\tilde{\Lyap}_k = \frac{\Lyap}{V_k}$.
Using \eqref{E4.10} and \eqref{EA2.2A} see note that
\begin{equation}\label{E4.15}
\sum_{j\in S} \tilde{\Lyap}_k(j) P^{(k)}(j|i)
= \frac{\sum_{j\in S} \Lyap(j) P(j|i, v_k(i))}{\sum_{j\in S} V_k(j) P(j|i, v_k(i))}
\leq e^{-(1-\alpha)\gamma}\tilde\Lyap_k(i) + \kappa_1 \Ind_{C},
\end{equation}
for some constant $\kappa_1$, independent of $k$, which is possible by \cref{L4.2}(ii). Similar
estimate also possible under \cref{EA2.2}(b). We also note that since $V_k\leq\Lyap$
$$\sum_{j\in S} V_k(j) P(j|i, v_k(i))\leq \sum_{j\in S} \Lyap(j) P(j|i, v_k(i))
\leq \kappa(i),$$
by \cref{EA2.2}, for some constant $\kappa(i)$.
Therefore, letting $\nu=\delta_{z_\circ}$, we see
from \eqref{EA4.1A} that 
$$P^{(k)}(A|i)\geq \kappa_2 \nu(A)\quad \text{for}\; A\subset C, \; \text{and}\; i\in C,$$
for some $\kappa_2$, independent of $k$. This is possible due to \cref{L4.2}(ii). Therefore,
$C$ is a small set for the chain ${\mathbf Y^{(k)}}$. Applying \cite[Theorem~2.3]{MeynT-94b}
we then obtain the following
%%%%%%%%%%%%%%%%%%%%%%%%%%%%%%%%%%%%%%%%%%%%%%%%%%%%%%%%%%%%%%%%%%%%%%%%%%%%%%%%%%%%%%%%%%%%%%
\begin{lemma}\label{L4.3}
For every $k\in \NN$, the Markov chain ${\mathbf Y^{(k)}}$ has a unique stationary probability measure 
$\pi_k$
and there are constant $M$ and $r\in (0,1)$, not depending on $k$, that satisfy
\begin{equation}\label{EL4.3A}
\Bnorm{\left(P^{(k)}\right)^n -\pi_k}_{\tilde\Lyap_k}\leq M\tilde\Lyap_k(i) r^n\quad \text{for all}\; n\in\NN,
\end{equation}
when the chain starts from the initial state $i$.
\end{lemma}
In the above
$$\norm{\mu_1-\mu_2}_V\df \sup_{|f|\leq V}|\mu_1(f)-\mu_2(f)|,$$
where $\mu_{m}(f) = \sum_{j\in S}f(j)\mu_{m}(j)$ for $m = 1,2$.
Recall from \eqref{EL4.2A} that
\begin{equation*}
\tilde\Lyap_k(i)\geq \kappa^{-1}\Lyap^{1-\alpha}(i)\quad i\in S.
\end{equation*}
Hence using \eqref{E4.15}, we get
\begin{equation}\label{E4.17}
\sum_{j\in S} \Lyap^{1-\alpha}(j) \pi_k(j)\,\leq \kappa_3\quad \text{for all}\; k\in\NN
\end{equation}
for some constant $\kappa_3$. This of course, implies that $\{\pi_k\}$ is tight. We claim that every
limit points of $\{\pi_k\}$ will have support in all of $S$. To see this, suppose along some
subsequence, $\pi_k\rightharpoonup \pi$ as $k\to\infty$. In view of \cref{L4.2} and compactness of action space,
we can extract a further subsequence along which
$$V_k(i)\to V(i)>0\quad \text{and}\quad v_k(i)\to v(i)\quad \text{for all}\; i\in S,$$
as $k\to\infty$.  This of course, implies (see \eqref{EL4.2B})
$$P^{(k)}(j|i) = \frac{V_k(j) P(j|i, v_k(i))}{\sum_{j\in S} V_k(j) P(j|i, v_k(i))}
\to \frac{V(j) P(j|i, v(i))}{\sum_{j\in S} V(j) P(j|i, v(i))} \df \tilde{P}(j|i), $$
as $k\to\infty$. Therefore, for any bounded function $f:S\to \RR$ supported on
a finite set, we get
\begin{align*}
\sum_{j\in S} f(j) \pi(j) &= \lim_{k\to\infty} \sum_{j\in S} f(j) \pi_k(j)
\\
&=\lim_{k\to\infty} \sum_{j\in S }\pi_k(j)\left[\sum_{z\in S} f(z) P^{(k)}(z|j)\right]
\\
&=\sum_{j\in S }\pi(j)\left[\sum_{z\in S} f(z) \tilde{P}(z|j)\right],
\end{align*}
where in the last line we use tightness of $\{\pi_k\}$. By a limiting argument we see that 
$$\sum_{j\in S} f(j) \pi(j)=\sum_{j\in S }\pi(j)\left[\sum_{z\in S} f(z) \tilde{P}(z|j)\right]$$
for all bounded function $f$. Thus $\pi$ is a stationary distribution to the Markov process generated
by the kernel $\tilde{P}$. Since $\tilde{P}$ is irreducible, which follows from the irreducibility of 
$P(j|i, v(i))$, $\pi$ must have support in all of $S$. Now we are ready to prove the following key lemma.
%%%%%%%%%%%%%%%%%%%%%%%%%%%%%%%%%%%%%%%%%%%%%%%%%%%%%%%%%%%%%%%%%%%%%%%%%%%%%%%%%%%%%%%%%%%%%%%%%%%
\begin{lemma}\label{L4.4}
Grant \cref{A1.1,A4.1}. Then we have $\theta_k(i)\to 0$ as $k\to\infty$, for all $i\in S$.
\end{lemma}

\begin{proof}
From \eqref{theta} we see that for $k\in\NN$
\begin{equation}\label{EL4.4A}
\theta_{k+1}(i)V_k(i) = V_k(i)-e^{c_{k+1}(i)-\lambda_k}\sum_{j\in S} V_k(j)P(j|i, v_{k+1}(i))
\quad i\in S.
\end{equation}
Applying Dynkin's formula to \eqref{EL4.4A} gives
\begin{align}\label{EL4.4B}
V_k(i) & = \Exp^{v_{k+1}}_i\left[e^{\sum_{t=0}^T (c_{k+1}(X_t)-\lambda_k)}V_k(X_{T+1})\right]
+ \sum_{t=0}^T\Exp^{v_{k+1}}_i\left[e^{\sum_{n=0}^{t-1} (c_{k+1}(X_n)-\lambda_k)} \theta_{k+1}(X_t)
V_{k}(X_t)\right]\nonumber
\\
&\geq \sum_{t=0}^T\Exp^{v_{k+1}}_i\left[e^{\sum_{n=0}^{t-1} (c_{k+1}(X_n)-\lambda_k)} \theta_{k+1}(X_t)
V_{k}(X_t)\right]
\end{align}
for all $T\in\NN$ and $i\in S$. Let us now define
$$h_{k+1}(i)=\theta_{k+1}(i)\frac{V_k(i)}{V_{k+1}(i)}\quad i\in S.$$
Then combining \eqref{E4.14} and \eqref{EL4.4B} we have
\begin{equation}\label{EL4.4C}
V_k(i)\geq V_{k+1}(i)\sum_{t=0}^T e^{(\lambda_{k+1}-\lambda_k)t}\,\tilde\Exp_i^{(k+1)}\left[h_{k+1}\bigl(Y^{(k+1)}_t\bigr)\right].
\end{equation}
Let
$$U_{k+1}(i)=\frac{V_k(i)}{V_{k+1}(i)}\quad i\in S.$$
Since 
$$V_k(i)\geq e^{c_{k+1}(i)-\lambda_k}\sum_{j\in S} V_k(j)P(j|i, v_{k+1}(i)),
\quad V_{k+1}(i)=e^{c_{k+1}(i)-\lambda_{k+1}}\sum_{j\in S} V_{k+1}(j)P(j|i, v_{k+1}(i)),$$
we obtain from \eqref{E4.13}
\begin{equation}\label{EL4.4D}
U_{k+1}(i)\geq e^{\lambda_{k+1}-\lambda_k}\sum_{j\in S} U_{k+1}(j) P^{(k+1)}(j|i)
\quad \text{for all}\; i\in S.
\end{equation}
We now split the proof into two cases.

{\bf Case 1.} Suppose that some $k\geq 0$ we have $\lambda_{k+1}=\lambda_k$.
It then follows from \eqref{EL4.4D} that $\{U_{k+1}(Y^{(k+1)}_n)\}$ is a
super-martingale. Since ${\mathbf Y}^{(k+1)}$ is recurrent, we must have
$U_{k+1}$ constant. From \eqref{theta} we get $\theta_{k+1}=0$ and 
$$V_k(i) = e^{c_k(i)-\lambda_k} \sum_{j\in S} V_k(j) P(j|i, v_k(i))
= \min_{u\in\Act(i)}\left[e^{c(i, u)-\lambda_k} \sum_{j\in S} V_k(j) P(j|i, u)\right].$$
Hence $v_k$ is also a minimizing selector. From \cref{T4.2} we then see that
$V_k$ is a scalar multiple of $\psi^*$ which in turn, implies from \cref{T2.1}
that $\lamstr=\lambda_k=\lambda_{k+1}=\lambda_{k+2}=\cdots$ and $\theta_j=0$ for $j\geq k+1$.

{\bf Case 2.} Suppose that the sequence $\{\lambda_k\}$ is strictly decreasing. Note from
\cref{L4.2} that $\sup_{k}h_k(i)<\infty$ for all $i$. Let $\sD$ be any finite set.
Then using \eqref{EL4.3A} and \eqref{EL4.4C} we have
\begin{align*}
U_{k+1}(i) \geq \sum_{t=0}^T e^{(\lambda_{k+1}-\lambda_k)t}
\Bigl[\pi_{k+1}(h_{k+1}\Ind_{\sD})- M\tilde{\Lyap}(i) r^t\Bigr].
\end{align*}
Letting $T\to\infty$ we get
$$U_{k+1}(i)\geq [1-e^{(\lambda_{k+1}-\lambda_k)}]^{-1}\pi_{k+1}(h_{k+1}\Ind_{\sD})
- \kappa_4 \tilde{\Lyap}(i),$$
for some constant $\kappa_4$, not dependent on $k$. Since $\{U_k(i)\}$ is a bounded
sequence, by \cref{L2.2}, and $\{\lambda_k\}$ is decreasing, it follows from above that
\begin{equation}\label{EL4.4E}
\lim_{k\to\infty} \pi_{k+1}(h_{k+1}\Ind_{\sD})=0,
\end{equation}
for every finite set $\sD$. Now suppose that for some $i\in S$, 
$\limsup_{k\to\infty} \theta_k(i)>0$. We choose $\sD=\{i\}$.
Since every subsequential limit of $\{\pi_k\}$ has support in $S$, 
$\liminf_{k\to\infty} V_k(i)>0$ by \cref{L4.2}(ii), this gives a contradiction
to \eqref{EL4.4E}. Hence we must have $\lim_{k\to\infty} \theta_k(i)=0$ for all
$i\in S$. This completes the proof.
\end{proof}
%%%%%%%%%%%%%%%%%%%%%%%%%%%%%%%%%%%%%%%%%%%%%%%%%%%%%%%%%%%%%%%%%%%%%%%%%%%%%%%%%%
Now we are ready to complete the proof of \cref{T4.1}.
\begin{proof}[Proof of \cref{T4.1}]
Suppose that $\lim_{k\to\infty}\lambda_k=\rho$. It is obvious that $\rho\geq\lamstr$.
Using \eqref{theta} we write
\begin{equation}\label{ET4.1A}
\theta_{k+1}(i) V_k(i) + \min_{u\in\Act(i)}\left[e^{c(i, u)-\lambda_k}
\sum_{j\in S}V_k(j) P(j|i, u)\right]= V_k(i).
\end{equation}
Using \cref{L4.2} and a diagonalization argument we can find positive $V\in\sorder(\Lyap)$
so that, along some subsequence, $V_k(i)\to V(i)$ for all $i\in S$. Passing
the limit in \eqref{ET4.1A} and using \cref{L4.4} we obtain
$$\min_{u\in\Act(i)}\left[e^{c(i, u)-\rho}
\sum_{j\in S}V(j) P(j|i, u)\right]= V(i)\quad i\in S.$$
From \cref{T4.2} we see that $\rho=\lamstr$ and $V$ is a scalar multiple of $\psi^*$.
Since $\psi^*$ is unique upto a normalization, (ii) follows.
\end{proof}

%%%%%%%%%%%%%%%%%%%%%%%%%%%%%%%%%%%%%%%%%%%%%%%%%%%%%%%%%%%%%%%%%%%%%%%%%%%%%%%%%%
\subsection{Continuous time stable case}
In this section we prove a PIA for the CTCMP we considered in \cref{S-CT}.
 Most of the statement and proofs in the section are continuous time analogue
 of \cref{S-DPIA}, therefore we mainly provide sketches for most of the results.
 We begin with the following assumption which we impose in this section,
 compare it with \cref{A4.1}.
 
\begin{assumption}\label{A4.2}
We suppose that \cref{A3.4} holds for a norm-like function $\Lyap$. Furthermore, in case of 
\cref{A3.4}(b), we have $\max_{u\in\Act(i)} c(i, u)\leq \eta\ell(i)$ for $i\in S$ and some $\eta\in(0,1)$.
\end{assumption}

As before, we define the generalized Perron-Frobenius eigenvalue as follows
\begin{equation}\label{E4.28}
\lamc=\inf\{\lambda\in\RR\; :\; \exists\; \Psi>0\; \text{satisfying}\;
\min_{u\in\Act(i)} \left[\sum_{j\in S} \Psi(j) q(j|i, u) + c(i, u)\Psi(i)\right]\leq \lambda\Psi(i)\; \forall \; i\in S\},
\end{equation}
and for every stationary Markov control $v$ we also define
\begin{equation}\label{E4.29}
\lamc(v)=\inf\{\lambda\in\RR\; :\; \exists\; \Psi>0\; \text{satisfying}\;
\left[\sum_{j\in S} \Psi(j) P(j|i, v(i)) + c(i, v(i))\Psi(i)\right]\leq \lambda\Psi(i)\; \forall \; i\in S\}.
\end{equation}
A claim analogous to \cref{L4.1} holds true for CTCMP and under the setting of
\cref{T3.1} we also have $\lamc=\lamstr$ and $\lamc(v)=\rho_v$ (see \cref{R3.1}).
We now describe our 
PIA.

\begin{algorithm}\label{Alg4.2}
Policy iteration.
\begin{itemize}
\item[1.] Initialization. Set $k=0$ and select any $v_0\in\Usm$.
\smallskip
\item[2.] Value determination. Let $V_k$
be the unique principal eigenfunction satisfying $V_k(i_0)=1$ and
\begin{equation}\label{E4.30}
\rho_k V_k(i) = \sum_{j\in S} V_k(j) q(j|i, v_k(i)) + c(i, v_k(i)) V_k(i)
\quad i\in S.
\end{equation}
Existence of a unique principal eigenfunction in \cref{E4.30} follows from
\cref{R3.1,L3.5} which is based on \cref{L2.6}. 
We let $\lambda_k\df\lamc(v_k)=\sE_i(c, v_k)=\rho_k$.
\smallskip
\item[3.] Policy improvement. Choose any $v_{k+1}\in\Usm$ satisfying
$$v_{k+1}(i)\in \Argmin_{u\in\Act(i)}\,
\left[\sum_{j\in S} V_k(j) q(j|i, u) + c(i, u) V_k(i)\right], \quad i\in S\,.$$
\end{itemize}
\end{algorithm}
%%%%%%%%%%%%%%%%%%%%%%%%%%%%%%%%%%%%%%%%%%%%%%%%%%%%%%%%%%%%%%%%%%%%%%%%%%%%%%%
We show that \cref{Alg4.2} converges. 
\begin{theorem}\label{T4.3}
Under \cref{A3.1,A3.1Exp,A3.3,A4.2} the following holds.
\begin{itemize}
\item[(i)] For all $k\in\NN$, we have $\lambda_{k+1}\leq\lambda_k$ and $\lim_{k\to\infty}\lambda_k=\lamstr$.
\item[(ii)] $V_k$ converges pointwise, as $k\to\infty$, to $\psi^*$ where $\psi^*$ is the unique solution to  \eqref{ET3.1A}.
\end{itemize}
\end{theorem}
We adapt the proof of \cref{T4.1} with suitable modification. Our next 
lemma follows by adapting the arguments of \cref{L4.2,T4.2} in a straightforward
manner.

\begin{lemma}\label{L4.5}
Grant the setting of \cref{T4.3}. Then the following hold.
\begin{itemize}
\item[(i)] There exists $\kappa$, independent of $k$, such that
\begin{equation}\label{EL4.5A}
V_k(i)\leq \kappa (\Lyap(i))^\alpha\quad \text{for all}\; i\in S,
\end{equation}
where $\alpha$ is given by \eqref{E4.8}.

\item[(ii)] For every $i\in S$ we have $\inf_{k\in\NN} V_k(i)>0$.

\item[(iii)] If for some $(\rho, W)\in\RR_+\times\sorder(\Lyap)$
 with $W>0$ we have
\begin{equation*}%\label{EL4.5B}
\rho W(i)=\min_{u\in\Act(i)}\left[\sum_{j\in S} W(j) q(j|i, u) + c(i, u) W(i)\right]
\quad i\in S,
\end{equation*}
where $\rho\geq \lamstr$. Then we must have $\rho=\lamstr$
and $W$ is a scaler multiple of $\psi^*$ where $\psi^*$ is given by \eqref{ET3.1A}.
\end{itemize}
\end{lemma}

As before, we denote by $c_k(i)=c(i, v_k(i))$.
Define $\Lyap_\upkappa(i)=(\Lyap(i))^{\upkappa}$ where $\upkappa\in (0, 1)$.
Since $t\mapsto t^\upkappa$ is concave in $(0, \infty)$,
we observe from \cref{A3.4} that
\begin{align}\label{E4.32}
\sum_{j\in S} \Lyap_\upkappa(j) q(j|i, u)
& = \sum_{j\neq i} (\Lyap^\upkappa(j)-\Lyap^\upkappa(i)) q(j|i, u)\nonumber
\\
&\leq \sum_{j\neq i} \upkappa \Lyap^{\upkappa-1}(i)(\Lyap(j)-\Lyap(i)) q(j|i, u)
\nonumber
\\
&\leq
\begin{cases}
&\upkappa \Lyap^{\upkappa-1}(i)\widehat{C}\Ind_\cK-\upkappa\gamma \Lyap_\upkappa(i),\quad \text{under \eqref{Lyap1}},
\\
&\upkappa \Lyap^{\upkappa-1}(i)\widehat{C}\Ind_\cK-\upkappa\ell(i) \Lyap_\upkappa(i),\quad \text{under \eqref{Lyap2}}.
\end{cases}
\end{align}
We fix $\upkappa\in  (\eta, 1)$ where $\eta$ is given by \cref{A4.2}.
For \eqref{Lyap1}, we shall fix $\upkappa$ close to $1$ so that 
$\norm{c}_\infty<\upkappa\gamma$. Also, note that we may choose $\alpha<\upkappa$ in  \eqref{E4.8}.

Let us now introduce the {\it twisted rate kernel}. For the eigenpair
$(\lambda_k, V_k)$ we define the kernel
$$q^{(k)}(j|i)=\frac{V_k(j)}{V_k(i)}q(j|i, v_k(i))\quad \text{for}\; i\neq j,
\quad \text{and}\quad q^{(k)}(i|i)=-\sum_{j\neq i} q^{(k)}(j|i).$$
From \eqref{E4.30} we see that
$$-q^{(k)}(i|i) = \rho_k-c_k-q(i|i,v_k(i))<\infty\quad i\in S.$$
Let $\tilde{\Lyap}_k(i)=\frac{\Lyap(i)}{V_k(i)}$ for $i\in S$. From 
\cref{A3.4} it then follows that
\begin{align}\label{E4.33}
\sum_{j\in S} \tilde\Lyap_k(j)q^{(k)}(j|i)
&= \sum_{j\neq} (\tilde\Lyap_k(j)-\tilde\Lyap_k(i))\frac{V_k(j)}{V_k(i)}q(j|i, v_k(i))\nonumber
\\
&= \frac{1}{V_k(i)}\sum_{j\in S} \Lyap(j) q(j|i, v_k(i))
- \frac{\Lyap(i)}{V^2_k(i)}\sum_{j\in S} V_k(j) q(j|i, v_k(i))\nonumber
\\
&\leq 
\begin{cases}
&\frac{1}{V_k(i)}\widehat{C}\Ind_\cK - (\gamma-c_k(i)+\lambda_k)\tilde\Lyap_k(i),
\quad \text{by}\; \eqref{Lyap1},
\\
&\frac{1}{V_k(i)}\widehat{C}\Ind_\cK - (\ell(i)-c_k(i)+\lambda_k)\tilde\Lyap_k(i),
\quad \text{by}\; \eqref{Lyap2},
\end{cases}
\nonumber
\\
&\leq 
\begin{cases}
&\widehat{C}_1\Ind_\cK - (1-\alpha)\gamma\tilde\Lyap_k(i),
\\
&\widehat{C}_1\Ind_\cK - (1-\alpha)\ell(i)\tilde\Lyap_k(i),
\end{cases}
\end{align}
for some constant $\widehat{C}_1$, where we use \eqref{E4.8}. Using 
\eqref{E4.33} and \cite[Theorem~2.2.4]{PZ20} we find a non-explosive 
Markov process ${\mathbf Y^{(k)}}$  corresponding to the kernel $q^{(k)}$.
Furthermore, since $q$ is irreducible for every stationary Markov control, we have
${\mathbf Y^{(k)}}$ irreducible. Letting 
$$\Phi_{\upkappa, k}(i)=\frac{\Lyap_\upkappa(i)}{V_k(i)}\quad i\in S,$$
from \eqref{E4.32} and \eqref{E4.33} we obtain
\begin{equation}\label{E4.34}
\sum_{j\in S} \Phi_{\upkappa,k}(j)q^{(k)}(j|i)
\leq
\begin{cases}
&\widehat{C}_2\Ind_\cK - (\upkappa-\alpha)\gamma \Phi_{\upkappa,k}(i),
\\
&\widehat{C}_2\Ind_\cK - (\upkappa-\alpha)\ell(i) \Phi_{\upkappa,k}(i),
\end{cases}
\end{equation}
for some constant $\widehat{C}_2$. By \cref{A4.2}, 
$\tilde\Lyap_k(i)/\Phi_{\upkappa, k}=\Lyap_{1-\upkappa}(i)\to \infty$ as
$i\to\infty$. Therefore, by \cite[Theorem~3.13]{PRHL}, ${\mathbf Y^{(k)}}$ is
exponentially ergodic with a unique invariant measure $\pi_k$. Using
\eqref{EL4.5A} and \eqref{E4.33} we get that
$$\sup_{k}\sum_{j\in S}\Lyap^{1-\alpha}(j)\pi_k(j)\leq \kappa,$$
for some constant $\kappa$. Thus $\{\pi_k\}$ is tight. As before, see \cref{S-DPIA}, we next show that any subsequential limit of $\{\pi_k\}$ is supported 
on whole of $S$. Since we do not have an exact analogue of \cref{L4.3} for
CTCMP, we modify the argument a bit. Consider a subsequnce of $\{\pi_k\}$
along which $\pi_k\rightharpoonup \pi$. Using a diaginalization argument and
selecting a further subsequence, if required, we can assure that
$$V_k(i)\to V(i), \quad v_k(i)\to v(i)\quad \text{for all}\; i\in S.$$
Using \eqref{EL4.5A} it is easily seen that
$$q^{(k)}(j|i)\to \tilde{q}(j|i),\; \tilde\Lyap_k(i)\to \tilde\Lyap(i),\;
\Phi_{\upkappa, k}(i)\to\Phi_\upkappa(i) \quad \text{for all}\; i, j,$$
where
$$\tilde\Lyap(i)=\frac{\Lyap(i)}{V(i)}, \; 
\Phi_\upkappa(i)=\frac{\Lyap_\upkappa(i)}{V(i)}\,.$$
\eqref{E4.33} and \eqref{E4.34} holds true for the kernel $\tilde{q}$. It can also
be easily checked that $\pi$ is the invariant measure corresponding to the kernel
$\tilde{q}$. Since $\tilde{q}$ generates an irreducible Markov process, $\pi$ must 
have its support in all of $S$. This proves the claim.

Let us know define the error term
\begin{equation}\label{ctheta}
\theta_{k+1}(i)=\lambda_k-c_{k+1}(i)-\frac{1}{V_k(i)}\sum_{j\in S} V_k(j) q(j|i, v_{k+1}(i))\quad i\in S.
\end{equation}
It follows from the definition (see step 3 of \cref{Alg4.2}) that $\theta_k\geq 0$.
On the other hand,
$$\theta_{k+1}(i)\leq \lambda_k + (\sup_{u\in \Act(i)} -q(i|i, u))
\leq \lambda_0 + (\sup_{u\in \Act(i)} -q(i|i, u))<\infty\quad i\in S.$$
Thus, $\{\theta_k\}$ is locally bounded in $k$. Next we show the following.
\begin{lemma}\label{L4.6}
Grant the setting of \cref{T4.3}. Then we have $\lim_{k\to\infty} \theta_k(i)=0$
for all $i\in S$.
\end{lemma}

\begin{proof}
Suppose, on the contrary, that for some $i\in S$, we have 
$\limsup_{k\to\infty}\theta_k(i)>0$. Passing to the subsequence we assume that
$\theta_k(i)\to \tilde\theta>0$. Now applying Dynkin's formula to \eqref{ctheta}
we have
\begin{equation}\label{EL4.6A}
V_{k-1}(i)\geq \int_0^T \Exp^{v_k}_i\left[e^{\int_0^t (c_k(X_s)-\lambda_{k-1})\D{s}} \theta_k(X_t)V_{k-1}(X_t)\right]\D{t},
\end{equation}
for all $T>0$. On the other hand, for $U_k(i)=\frac{V_{k-1}(i)}{V_k(i)}$, we have
$$\sum_{j\in S} U_k(j) q^{(k)}(j|i)\leq (\lambda_{k-1}-\lambda_k) U_k(i).$$
Thus, if $\lambda_{k-1}=\lambda_{k}$ the proof follows from the argument of 
\cref{L4.4}, Case 1. So we assume that $\{\lambda_k\}$ is strictly decreasing.

To this end, we need a continuous time counterpart of \eqref{E4.14}. Suppose that
$g$ is a non-negative function supported on a finite subset of $S$. Defining
$u(t)(i)=\Exp^{(k)}_i[g(Y^{(k)}_t)]$ we know that
\begin{equation*}
\frac{\D u(t)(i)}{\D{t}}= \sum_{j\in S} u(t)(j) q^{(k)}(j|i)\quad i\in S.
\end{equation*}
Using \eqref{E4.30} this can be rewritten as
$$\frac{\D V_k(i)u(t)(i)}{\D{t}}= \sum_{j\in S} V_k(j)u(t)(j) q(j|i, v_k(i))
+ (c_k(i)-\lambda_k) V_k(i)u(t)(i).$$
Thus, from Dynkin's formula, we obtain
\begin{equation}\label{EL4.6B}
V_k(i)\Exp^{(k)}_i[g(Y^{(k)}_t)]=\Exp^{v_k}_i\left[e^{\int_0^t (c_k(i)-\lambda_k)\D{s}} g(X_t)V_k(t)\right].
\end{equation}
By a standard approximation the above relation can be extended to all nonnegative
functions $g$ on $S$. Let $h_k(i)=\theta_k(i)V_{k-1}(i)/V_k(i)$. Using 
\eqref{EL4.6A} and \eqref{EL4.6B} we obtain
$$ U_{k}(i) \geq  \int_0^T e^{(\lambda_k-\lambda_{k-1})t} 
\Exp^{(k)}_i[\Ind_{\{i\}}(Y^{(k)}_t)h_k(Y^{(k)}_t)].$$
Now we let $k\to\infty$, so that $q^{(k)}\to\tilde{q}$ and $\pi_k\to\pi$,
along some subsequnece. Since ${\mathbf Y^{(k)}}$ converges in distribution
to $\tilde{\mathbf Y}$ where $\tilde{\mathbf Y}$ is the Markov process
associated to $\tilde{q}$ (this can be seen by adapting the arguments of 
\cite[Lemma~5.8]{PRHL} ) and, $U_k(i)$ is bounded above and $h_k(i)$ is bounded 
below (by \cref{L4.5}), we
get from above
$$\int_0^T  
\Exp^{(k)}_i[\Ind_{\{i\}}(\tilde{Y}_t)]\leq \kappa_1$$
for some $\kappa_1$. But $\tilde{\mathbf Y}$ is exponentially ergodic
\cite[Theorem~3.13]{PRHL} with
invariant measure $\pi$ having support in $i$. Letting $T\to\infty$, we get a contradiction.
\end{proof}

Now we can complete the proof of \cref{T4.3}.
\begin{proof}[Proof of \cref{T4.3}]
The proof follows from \cref{L4.5,L4.6} together with the arguments of \cref{T4.1}.
\end{proof}

%%%%%%%%%%%%%%%%%%%%%%%%%%%%%%%%%%%%%%%%%%%%%%%%%%%%%%%%%%%%%%%%%%%%%%%%%%%%%%%%%%%%%%%%%%%%%%%%%%%%%%%%%%%%%%%%%%%%%%%%%%%%%%%%%%%%%%%%%%%%%%%%%%%%%%%%%%%%%%%%%%%%%%%%
\subsection*{Acknowledgement}
We thank the anonymous reviewers for their careful reading of our    manuscript and suggestions.
The authors are grateful to Mrinal Ghosh, Chandan Pal and Subhamay Saha for their comments
on this article.
The research of Anup Biswas was supported in part by a SwarnaJayanti fellowship and DST-SERB grant MTR/2018/000028. Somnath Pradhan was supported in part by a National Postdoctoral Fellowship PDF/2020/001938.
%%%%%%%%%%%%%%%%%%%%%%%%%%%%%%%%%%%%%%%%%%%%%%%%%%%%%%%%%%%%%%%%%%%%%%%%%%%%%%%%
\bibliographystyle{plain}
\bibliography{Risk_Markov}

\end{document}